\documentclass{elsarticle}
\usepackage[english]{babel}
\usepackage{a4wide}
\usepackage{mathtools}      
\usepackage{amstext, amsmath, amssymb}
\usepackage{nccmath}
\usepackage{graphicx}
\usepackage{latexsym}
\usepackage{booktabs}
\usepackage{url}
\usepackage{import}
\usepackage{fancyvrb}
\usepackage{stmaryrd}
\usepackage{pdfsync}
\usepackage{algorithm}
\usepackage{tikz}
\usepackage{pgfplots}
\usepackage{pgfplotstable}
\pgfplotsset{compat=newest}
\usepackage{subcaption}
\usepackage[nice]{nicefrac}
\usepackage{tabularx}
\usepackage{todonotes}
\usepackage{enumitem}
\biboptions{numbers,square,sort&compress}

\usepackage{array}
\usepackage{verbatim}

\usepackage[plainpages=false]{hyperref}
\hypersetup{
		bookmarksopen=true,
		bookmarksnumbered=true,
		pdfborder={0 0 0},
		colorlinks=true,
		linkcolor=blue,
		urlcolor=blue	
              }

\newcommand{\RR}{\mathbb{R}}

\newcommand{\PP}{\mathbb{P}}
\newcommand{\st}{\;|\;}

\newcommand{\bfx}{\boldsymbol{x}}

\newcommand{\bfbeta}{\boldsymbol{\bfbeta}}

\newcommand{\bfzero}{\boldsymbol{0}}

\newcommand{\mcP}{\mathcal{P}}

\newcommand{\mcF}{\mathcal{F}}

\newcommand{\mcA}{\mathcal{A}}

\newcommand{\mcT}{\mathcal{T}}

\newcommand{\mcO}{\mathcal{O}}

\newcommand{\tn}{|\mspace{-1mu}|\mspace{-1mu}|}

\newcommand{\jump}[1]{[#1]}
\newcommand{\mean}[1]{\{{#1} \}}
\newcommand{\avg}[1]{\{{#1} \}}

\newcommand{\wavg}[1]{\{{#1} \}_{\omega}}
\newcommand{\wavgd}[1]{\langle {#1} \rangle_{\omega}}

\newcommand{\foralls}{\forall\,}

\newcommand{\dx}{\,\mathrm{d}x}

\newcommand{\mcTone}{\mcT_{h,1}}

\newcommand{\nablan}{\dfrac{\partial}{\partial n}}
\newcommand{\partialbar}{{\partial}}

\newcommand{\onehalf}{\nicefrac{1}{2}}

\DeclareMathOperator{\dist}{dist}
\DeclareMathOperator{\diam}{diam}

\DeclareMathOperator{\atantwo}{atan_2}

\DefineVerbatimEnvironment{code}{Verbatim}{frame=single,rulecolor=\color{blue}}

\numberwithin{equation}{section}
\numberwithin{figure}{section}
\numberwithin{table}{section}

\newtheorem{theorem}{Theorem}[section]
\newtheorem{proposition}[theorem]{Proposition}
\newtheorem{lemma}[theorem]{Lemma}
\newtheorem{assumption}{Assumption}
\newtheorem{corollary}[theorem]{Corollary}

\newdefinition{remark}[theorem]{Remark}
\newdefinition{definition}[theorem]{Definition}
\newproof{proof}{Proof}

\numberwithin{equation}{section}

\journal{arXiv}

\begin{document}
\begin{frontmatter}
  
  \title{A stabilized cut discontinuous Galerkin framework: I. Elliptic boundary value and interface problems.}

\author[umu]{Ceren G\"urkan}
\ead{ceren.gurkan@umu.se}

\author[umu]{Andr\'e Massing\corref{cor1}}
\ead{andre.massing@umu.se}
\cortext[cor1]{Corresponding author}

\address[umu]{Department of Mathematics and Mathematical Statistics, Ume{\aa} University, SE-90187 Ume{\aa}, Sweden}

\begin{abstract}
  We develop a stabilized cut discontinuous Galerkin framework for the
  numerical solution of elliptic boundary value and interface problems
  on complicated domains.  The domain of interest is embedded in a
  structured, unfitted background mesh in $\mathbb{R}^d$, so that the
  boundary or interface can cut through it in an arbitrary
  fashion.  The method is based on an unfitted variant of the
  classical symmetric interior penalty method using
  piecewise discontinuous polynomials defined on the background
  mesh. Instead of the cell agglomeration technique commonly used in
  previously introduced unfitted discontinuous Galerkin methods, we
  employ and extend ghost penalty techniques from recently
  developed continuous cut finite element
  methods, which allows for a minimal
  extension of existing fitted discontinuous Galerkin software to
  handle unfitted geometries.
  Identifying four abstract assumptions on the ghost penalty, we
  derive geometrically robust a priori error and condition number
  estimates for the Poisson boundary value problem which hold
  irrespective of the particular cut configuration. Possible
  realizations of suitable ghost penalties are discussed.
  We also demonstrate how the framework can be elegantly applied to
  discretize high contrast interface problems.
  The theoretical results are illustrated by
  a number of numerical experiments for various approximation orders
  and for two and three-dimensional test problems.
  
\end{abstract}

\begin{keyword}
  Elliptic problems \sep discontinuous Galerkin
  \sep cut finite element method \sep stabilization \sep condition number \sep a priori error estimates
\end{keyword}
\end{frontmatter}

\section{Introduction}
\subsection{Background}
A fundamental prerequisite for the finite element based numerical
solution of partial differential equations (PDEs) is the generation of
high quality meshes to resolve geometric domain features and to ensure
a sufficiently accurate approximation of the unknown solution.  But
despite continuously growing computer power, the generation of meshes
for realistic, complex three-dimensional domains can still be a
challenging task that can easily account for large portions of the
time, human and computing resources in the overall simulation work
flow.  In applications where the domain geometry is main subject of
interest, e.g., in shape optimization
problems~\cite{AllaireDapognyFrey2014,BurmanElfversonHansboEtAl2018,BernlandWadbroBerggren2017},
the need of frequent remeshing can be the major computational cost.
When the domain boundary is exposed to large or even topological
changes, for instances in large deformation fluid-structure
interaction problems \cite{TezduyarSatheKeedyEtAl2006} or multiphase
flows
\cite{AlbertOneil1986,GrosReicheltReusken2006,MarchandiseRemacle2006,GanesanTobiska2009},
even modern mesh moving algorithms may break down and then a costly
remeshing is the only resort.  Even if the domain of interest is
stationary but rather complex, creating a 3D high quality mesh is a
computationally demanding task.  For instance, the simulation of
geological flow and transport problems requires a series of highly
non-trivial preprocessing steps to transform geological image data
into conforming domain discretizations which respect intricate geometric
structures such as faults and large-scale networks of
fractures~\cite{DassiPerottoFormaggiaEtAl2014}.  Similar non-trivial
preprocessing steps are necessary when mesh-based domain descriptions
are generated from biomedical image data, e.g., when creating flow
models~\cite{AntigaPeiroSteinman2009}, bone
models~\cite{UralZiouposBuchananetal2011} or tissue
models~\cite{CattaneoZunino2014a}.  As a possible remedy to the mesh
generation challenges, so-called \emph{unfitted} finite elements
methods have gained much attention in recent years.  The fundamental
idea is to avoid creating meshes fitted to the domain boundary by simply
embedding the domain inside an easy-to-create background
mesh. This way, the geometry description is decoupled from the
numerical approximation and hence, complex static or evolving
geometry can be handled.

\subsection{Earlier work}
Starting with~\citep{DinhGlowinskiHeEtAl1992}, Glowinski et al.\
presented several fictitious domain formulations for the finite
element method
\citep{GlowinskiPanPeriaux1994,GlowinskiPanPeriaux1995,GlowinskiKuznetsov2007}
with applications to electromagnetics~\citep{DinhGlowinskiHeEtAl1992},
elliptic equations~\citep{GlowinskiKuznetsov2007}, and mainly fluid
related equations
\citep{DinhGlowinskiHeEtAl1992,Glowinski1997,GlowinskiPanPeriaux1999,GlowinskiPanHeslaEtAl1999,GlowinskiPanHeslaEtAl2001},
focusing on flows around moving rigid bodies.

In \cite{MoesDolbowBelytschko1999}, Mo\"es et al.\ 
introduced the so-called eXtended
finite element method (XFEM) to avoid remeshing during crack propagation
by representing discontinuities within an single mesh element,
Based on the partition of unity method (PUM)~\cite{MelenkBabuska1996},
the polynomial spaces in the elements cut by the crack are enriched.
The idea was later picked up by many authors and extended to a variety
of applications, including two-phase
flows~\cite{ChessaBelytschko2003}, dendritic
solidification~\cite{ZabarasGanapathysubramanianTan2006}, shock
capturing problems~\cite{AbbasAlizadaFries2009}, fluid-structure
interaction problems~\cite{GerstenbergerWall2008}, and flow and
transport in fractured porous media~\cite{Fumagalli2012,
  FumagalliScotti2014}.  For an early overview over XFEM and its
applications, the reader is referred to~\cite{FriesBelytschko2000} and
the references therein.

Parallel to the XFEM methodology, ~\citet{HansboHansbo2002} proposed
an alternative unfitted finite element formulation to treat elliptic
interface problems by imposing weak discontinuities within the
elements using a variant of Nitsche's method~\cite{Nitsche1971}.
Optimal a priori error and a posteriori error estimates were derived,
independent of the interface position.  Soon after, the idea was
extended to composite grids~\cite{HansboHansboLarson2003} and to
linear elasticity problems with strong and weak
discontinuities~\cite{HansboHansbo2004}.
Later~\citet{AreiasBelytschko2006} showed that the approach
in~\cite{HansboHansbo2004} can be recast into a XFEM formulation with
a Heaviside enrichment.  To deal with incompressible elasticity, the
approach from~\cite{HansboHansbo2004} was extended in
\cite{BeckerBurmanHansbo2009} by considering a stabilized mixed
formulation using $\PP_1$ continuous displacements and $\PP_0$
discontinuous pressures. A critical ingredient in the analysis was the
extension of the jump penalty based pressure stabilization from the
``physical'' part of the faces to the entire faces, leading to the
first ``ghost penalty'' stabilized unfitted finite element
formulation.  The key idea of employing ghost penalties to extend the
control of the relevant norms from the physical domain to the entire
active mesh then crystallized in a series of papers
\cite{BurmanHansbo2010,Burman2010,BurmanHansbo2012} proposing Lagrange
multiplier and Nitsche-based, optimally convergent fictitious domain
methods for the Poisson problem.  Additionally, the condition numbers
of the associated system matrices turned out to be insensitive to the
particular cut configuration and scaled similiar with respect to the
mesh size as their fitted mesh counterparts.

Building upon and extending these ideas, the cut finite element method
(CutFEM) as a particular unfitted finite element framework has gained
rapidly increasing attention in the science and engineering community,
see~\cite{BurmanClausHansboEtAl2014,BordasBurmanLarsonEtAl2018} for
some recent overviews.
A distinctive feature of the
CutFEM approach is that it provides a general, theoretically founded
stabilization framework which, roughly speaking, transfers stability
and approximation properties from a finite element scheme posed on a
standard mesh to its cut finite element counterpart.
As a result, a wide range of problem classes
has been treated 
including, e.g.,
elliptic interface problems~\cite{BurmanZunino2012,GuzmanSanchezSarkis2015,BurmanGuzmanSanchezEtAl2017},
Stokes and Navier-Stokes type problems
\cite{BurmanHansbo2013,MassingLarsonLoggEtAl2013,BurmanClausMassing2015,
CattaneoFormaggiaIoriEtAl2014,MassingSchottWall2017,WinterSchottMassingEtAl2017,KirchhartGrosReusken2016,GrosLudescherOlshanskiiEtAl2016,GuzmanOlshanskii2016},
two-phase and fluid-structure
interaction
problems~\cite{SchottRasthoferGravemeierEtAl2015,GrosReicheltReusken2006,GrossReusken2011,MassingLarsonLoggEtAl2015}.
Building up on the seminal work by
~\citet{OlshanskiiReuskenGrande2009,OlshanskiiReusken2010},
stabilized CutFEMs and so-called TraceFEMs were also developed for 
surface and surface-bulk
PDEs~\cite{BurmanHansboLarson2015,
BurmanHansboLarsonEtAl2016,
BurmanHansboLarsonEtAl2016c,
HansboLarsonMassing2017a,
GrandeLehrenfeldReusken2016,
HansboLarsonZahedi2016,Reusken2014,OlshanskiiReusken2014}.
As a natural application area,
unfitted finite element methods have also been proposed
for problems in fractured porous media~\cite{FlemischFumagalliScotti2016,Fumagalli2012,
  DAngeloScotti2012,FormaggiaFumagalliScottiEtAl2013}.
Finally, we also mention the finite cell method as another important instance of an unfitted finite element
framework with
applications to flow and mechanics problems~\cite{ParvizianDuesterRank2007,
XuSchillingerKamenskyEtAl2015,VarduhnHsuRuessEtAl2016},
see~\cite{SchillingerRuess2014} for a review and references therein.

In addition to the aforementioned unfitted \emph{continuous} finite
element methods, unfitted \emph{discontinuous} Galerkin methods have
successfully been devised to treat boundary and interface problems on
complex and evolving
domains~\cite{BastianEngwer2009,BastianEngwerFahlkeEtAl2011,Saye2015},
including flow problems with moving boundaries and
interfaces~\cite{SollieBokhoveVegt2011,HeimannEngwerIppischEtAl2013,Saye2017,MuellerKraemer-EisKummerEtAl2016,KrauseKummer2017}.
In contrast to stabilized continuous cut finite element methods, in
unfitted discontinuous Galerkin methods, troublesome small cut
elements can be merged with neighbor elements with a large
intersection support by simply extending the local shape functions
from the large element to the small cut element.  As inter-element
continuity is enforced only weakly, no additional measures need to be
taken to force the modified basis functions to be globally continuous.
Consequently, cell merging in unfitted discontinuous Galerkin methods
provides an alternative stabilization mechanism to ensure that the
discrete systems are well-posed and well-conditioned.  For a very
recent extension of the cell merging approach to continuous finite
elements, we refer to~\cite{BadiaVerdugoMartin2017,BadiaVerdugo2017a}.
Thanks to their favorable conservation and stability properties,
unfitted discontinuous Galerkin methods remain an attractive
alternative to continuous CutFEMs, but some drawbacks are the almost
complete absence of numerical analysis except
for~\cite{Massjung2012,JohanssonLarson2013}, the implementational
labor to reorganize the matrix sparsity patterns when agglomerating
cut elements, and the lack of natural discretization approaches for
PDEs defined on surfaces.

Finally, we also point out that alternative discontinuous Galerkin method for PDEs on complicated domains
have been devised by exploiting the possibility to use
non-standard element geometries which allows for a more flexible
meshing of complicated geometries, see, e.g.,~\cite{AntoniettiCangianiCollisEtAl2016,
AntoniettiFacciolaRussoEtAl2016a,
GianiHouston2014a}.

\subsection{Novel contributions and outline of this paper}
In this work, we initiate the development of a novel \emph{stabilized}
cut discontinuous Galerkin (cutDG) framework by extending the
stabilization techniques from the continuous CutFEM approach to
discontinuous Galerkin based discretizations.
Such an approach allows for a minimally invasive
extension of existing fitted discontinuous Galerkin software to
handle unfitted geometries.
Only additional quadrature routines need to be added
to handle the  numerical integration on cut geometries, and while not being a completely trivial implementation
task, we refer to the 
numerous quadrature algorithms capable of higher order geometry approximation
~\cite{MuellerKummerOberlack2013,Saye2015,Lehrenfeld2016,FriesOmerovic2015,FriesOmerovicSchoellhammerEtAl2017}
which have been proposed in the last 5 years.
Additionally, with a suitable choice of the ghost penalty,
the sparsity pattern associated with the final system matrix does not require any
manipulation and is identical to its fitted DG counterpart.

To lay out the
theoretical foundations in the most simple setting, we here introduce
and analyze cutDGMs for elliptic boundary and interface problems in
Section~\ref{sec:bound-value-probl} and
Section~\ref{sec:interface-problems}, respectively.
Boundary and
interface conditions are imposed weakly using Nitsche's method, and
the discrete bilinear forms are augmented with an abstract ghost
penalty stabilization.
Hyperbolic and advection dominant advection-diffusion-reaction problems
are considered in~\cite{GuerkanMassing2018a},
while in the upcoming work~\cite{GuerkanMassing2018b}, we will combine the presented framework
with extension of~\cite{BurmanHansboLarsonEtAl2016a,Massing2017}
to introduce cutDGM for mixed-dimensional, coupled problems.

We start with the Poisson boundary problem as
model problem in Section~\ref{ssec:model-problem-bvp} followed by the
presentation of a symmetric interior penalty based cutDGM in
Section~\ref{ssec:cutDGM-bvp}.  In the course of our stability and a
priori error analysis of the cutDGM for the Poisson boundary problem
in Section~\ref{ssec::stab-prop}--\ref{ssec::apriori}, we identify two
abstract assumptions on the ghost penalty to derive geometrically
robust and optimal approximation properties.  In contrast to
continuous cutFEMs, those do not automatically guarantee that the condition number
of the resulting linear system is insensitive to the particular cut
configuration. We find two additional abstract assumptions, allowing
us to prove geometrically robust condition number estimates in
Section~\ref{ssec::condition-number-est}.  Afterwards in
Section~\ref{ssec:ghost-penalty-realizations}, we discuss a number of
possible ghost penalty realizations which satisfy our abstract
assumptions for the considered piecewise polynomial space of order
$p$. The discussion of cutDGM for the Poisson boundary problem
concludes with a series of numerical results in two and three
dimensions to corroborate our theoretical findings and to examine the
effects and properties of the ghost penalties in detail, see
Section~\ref{ssec:ghost-penalty-realizations}. Finally, in
Section~\ref{sec:interface-problems}, we demonstrate how the framework
can easily be used to devise cutDGM for high contrast interface
problems. After the presentation of the interface model
and the corresponding cutDGM in Section~\ref{ssec:ifp:model-problem}
and Section~\ref{ssec:cutDGM-if}, respectively,
we derive optimal a priori
estimates in a concise manner in Section~\ref{sec:ip-stabilityprop},
followed by a number of convergence rate experiments for low and high
contrast interface problems presented in
Section~\ref{ssec:ifp:numerical-examples-ifp}.

\section{Elliptic boundary value problems} 
\label{sec:bound-value-probl} 
\subsection{Basic notation}
Throughout this work, $\Omega \subset \RR^d$, $d = 2,3$ denotes an
open and bounded domain\footnote{The precise
  regularity assumptions on $\Omega$ will be stated at the end of
  Subsection~\ref{ssec:cutDGM-bvp}}
with piecewise smooth boundary
$\partial \Omega$, while $\Gamma$ denotes a piecewise smooth manifold
of codimension $1$ embedded into $\RR^d$.
For $U \in \{\Omega, \Gamma \}$
and $ 0 \leqslant m < \infty$, $1 \leqslant q \leqslant \infty$, let
$W^{m,q}(U)$ be the standard Sobolev spaces consisting of those
$\RR$-valued functions defined on $U$ which possess $L^q$-integrable
weak derivatives up to order $m$. Their associated norms are denoted
by $\|\cdot \|_{m,q,U}$.  As usual, we write $H^m(U) = W^{m,2}(U)$ and
$(\cdot,\cdot)_{m,U}$ and $\|\cdot\|_{m,U}$ for the associated inner
product and norm. If unmistakable, we occasionally write
$(\cdot,\cdot)_{U}$ and $\|\cdot \|_{U}$ for the inner products and
norms associated with $L^2(U)$, with $U$ being a measurable subset of
$\RR^d$.
Any norm $\|\cdot\|_{\mcP_h}$ used in this work which
involves a collection of geometric entities $\mcP_h$ should be
understood as broken norm defined by
$\|\cdot\|_{\mcP_h}^2 = \sum_{P\in\mcP_h} \|\cdot\|_P^2$ whenever
$\|\cdot\|_P$ is well-defined, with a similar convention for scalar
products $(\cdot,\cdot)_{\mcP_h}$.  Any set operations
involving $\mcP_h$ are also understood as element-wise operations,
e.g., $ \mcP_h \cap U = \{ P \cap U \st P \in \mcP_h \} $ and $
\partial \mcP_h = \{ \partial P \st P \in \mcP_h \} $ allowing for
compact short-hand notation such as
$ (v,w)_{\mcP_h \cap U} = \sum_{P\in\mcP_h} (v,w)_{P \cap U} $ and
$ \|\cdot\|_{\mcP_h\cap U} = \sqrt{\sum_{P\in\mcP_h} \|\cdot\|_{P\cap
    U}^2}$.
Finally, throughout this work, we use the notation $a \lesssim b$ for
$a\leqslant C b$ for some generic constant $C$ (even for $C=1$) which varies with the context
but is always independent of the mesh size $h$ and the position of
$\Gamma$ relative to the background $\mcT_h$.

\subsection{Poisson problem}
\label{ssec:model-problem-bvp}
We consider the following model boundary value problem: given $f\in H^1(\Omega)$ and $g\in H^{1/2}(\Gamma)$, find
$u : \Omega \to \RR$ such that
\begin{subequations}
  \label{eq:laplace-strong}
  \begin{align}
    -\Delta u &= f \quad \mbox{in } \Omega,
                                   \label{eq:laplace-strong-pde}
    \\
    u & = g  \quad \mbox{on } \Gamma.
        \label{eq:laplace-strong-dirichlet}
  \end{align}
\end{subequations}
Setting 
$V_{g} = \{v \in H^1(\Omega) : v|_{\Gamma}=g \}$
and defining the bilinear and linear forms
\begin{align}
  a(u,v)=(\nabla u, \nabla v)_{\Omega},
          \qquad
  l(v) = (f,v)_{\Omega},
  \label{eq:cont-forms}
\end{align}
the weak or variational formulation of the strong problem~\eqref{eq:laplace-strong-pde} 
is to seek $u \in V_{g}$ such that 
\begin{equation}
  a(u,v)=l(v) \quad \foralls v \in V_0.
  \label{eq:variational_BVP}
\end{equation} 
\subsection{A cut discontinuous Galerkin method for the Poisson problem}
\label{ssec:cutDGM-bvp}

Let $\widetilde{\mcT}_h$ be a quasi-uniform background mesh consisting of 
$d$-dimensional, shape-regular (closed) simplices $\{T\}$
covering $\overline{\Omega}$. As usual, we introduce the
local mesh size $h_T = \diam(T)$ and the
global mesh size $h = \max_{T \in \widetilde{\mcT}_h} \{h_T\}$.
For $\widetilde{\mcT}_h$ we define the so-called \emph{active} (background) mesh
\begin{align} 
  \mcT_h &= \{ T \in \widetilde{\mcT}_{h} \st \cap \Omega^{\circ} \neq \emptyset \},
  \label{eq:active-mesh-bvp}
\end{align}
and its submesh $\mcT_{\Gamma}$ consisting of all cut elements,
\begin{align} 
  \mcT_{\Gamma} &= \{ T \in \widetilde{\mcT}_{h} \st \cap \Gamma \neq \emptyset \}.
  \label{eq:cut-elements-mesh-bvp}
\end{align}
Note that since the elements $\{T\}$ are
closed by definition, the active mesh $\mcT_h$ still covers ${\Omega}$.
The set of interior faces in the active background mesh is given by
 \begin{align} 
   \mcF_h &= \{ F =  T^+ \cap T^- \st T^+, T^- \in \mcT_h \}.
  \label{eq:faces-interior-bvp}
 \end{align}
\begin{figure}[htb]
  \begin{center}
    \begin{minipage}[t]{0.45\textwidth}
    \vspace{0pt}
    \includegraphics[width=1.0\textwidth]{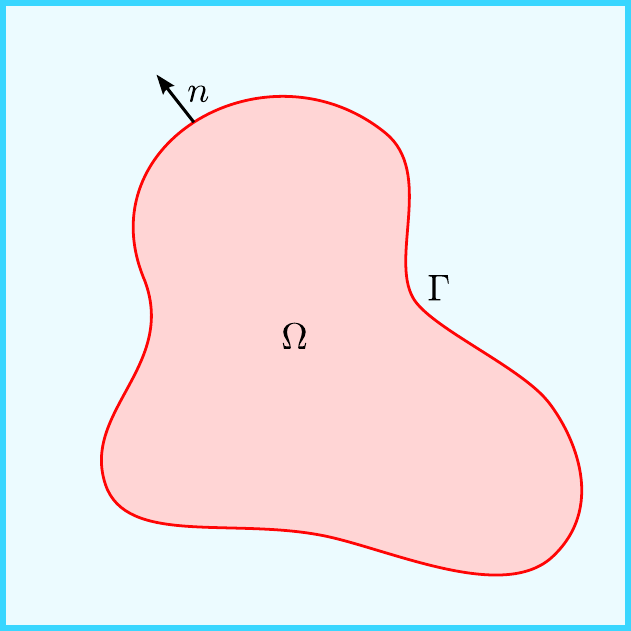}
  \end{minipage}
  \hspace{0.02\textwidth}
  \begin{minipage}[t]{0.45\textwidth}
    \vspace{0pt}
    \includegraphics[width=1.0\textwidth]{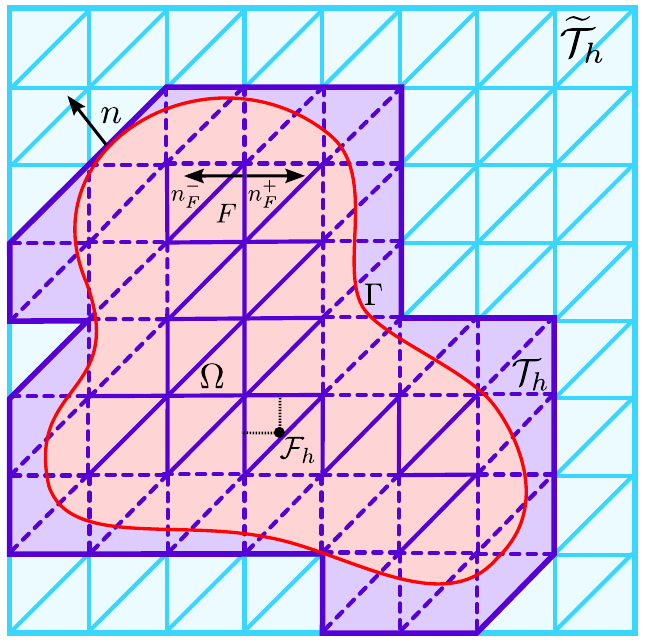}
  \end{minipage}
\end{center}
\caption{Computational domains for the boundary value problem~(\ref{eq:laplace-strong}).
  (Left) Physical domain $\Omega$ with boundary $\Gamma$ and outer normal $n$.
  (Right) Background mesh and active
  mesh used to define the approximation space.
  Faces on which the face based ghost penalty~(\ref{eq:jh-def-face-based}) is defined are
  plotted as dashed faces.}
  \label{fig:domain-set-up}
\end{figure}
On the active mesh $\mcT_h$, we define the discrete function space $V_h$ as the broken polynomial space
of order~$k$,
\begin{align}
  V_h
  \coloneqq \PP_k(\mcT_h)
  \coloneqq \bigoplus_{T \in \mcT_h} \PP_k(T).
\label{eq:Vsh-def-bvp}
\end{align}
To formulate our cut discontinuous Galerkin method for the boundary
value problem~\eqref{eq:laplace-strong}, we recall the definition of 
the averages
\begin{align}
\avg{\sigma}|_F &= \dfrac{1}{2} (\sigma_F^{+} + \sigma_F^{-}),
  \label{eq:mean-std-def-F}
  \\
  \avg{n \cdot \sigma }|_F &=
                             \dfrac{1}{2} n \cdot (\sigma_F^{+} + \sigma_F^{-}),
  \label{eq:mean-def-F}
\end{align}
and the jump across an interior face $F \in
\mcF_h$, 
\begin{align}
  \jump{w}|_F &= w_F^{+} - w_F^{-}.
\end{align}
Here,  $w(x)^\pm = \lim_{t\rightarrow 0^+} w(x \pm t n)$ for some chosen unit facet normal $n$ on $F$.
\begin{remark}
  To keep the notation at a moderate level, we usually do not use subscripts to indicate
  whether a normal belongs to a $F$ or to the boundary $\Gamma$ 
  as it will be clear from context.
\end{remark}
With these definitions in place, we can now define the discrete
counterparts to the continuous bilinear and linear form~\eqref{eq:variational_BVP}
and set
\begin{align}
  a_h(v,w) &= (\nabla v, \nabla w)_{\mcT_h\cap\Omega} 
             - (\partial_n v, w)_{\Gamma}
             - (v, \partial_n w)_{\Gamma}
             + \beta (h^{-1} v,w)_{\Gamma} 
             \nonumber
  \\                    
           &\qquad 
             - (\mean{\partial_n v }, \jump{w})_{\mcF_h \cap \Omega}
             - (\jump{v}, \mean{\partial_n w })_{\mcF_h \cap \Omega}
             + \beta (h^{-1}\jump{v},\jump{w})_{\mcF_h \cap \Omega},
             \label{eq:ah}
  \\           
l_h(v) &= (f,v)_{\mcT_h\cap\Omega} - (\partial_n v, g)_{\Gamma}  
  +\beta (h^{-1} g, v)_{\Gamma},
             \label{eq:lh}
\end{align} 
for $v, w \in V_h$ where we used the short-hand notation $\partial_n v = n \cdot \nabla v$.
The symmetric interior penalty based  cut discontinuous Galerkin method
for the Poisson problem~(\ref{eq:laplace-strong}) then reads: find $u_h \in V_h $ such that
$\foralls v \in V_h$
\begin{align}
  A_h(u_h, v) \coloneqq a_h(u_h,v) + g_h(u_h, v) = l_h(v).
  \label{eq-cutdg-formulation}
\end{align}
In contrast to the classical symmetric interior penalty method
formulated on \emph{fitted} meshes, we here augment the bilinear form
$a_h$ with an additional stabilization form $g_h$, which is typically only active on
elements in the vicinity of the embedded boundary $\Gamma$.
The role
of this stabilization is to ensure that, irrespective of the
particular cut configuration, the bilinear form $A_h$ defined
in~\eqref{eq-cutdg-formulation} is coercive and bounded with
respect to certain discrete energy-norms, and that the system matrix
associated with $A_h$ is well-conditioned.  To obtain these properties
while maintaining the approximation qualities of original
symmetric interior penalty method, the stabilization has to satisfy
certain assumptions which we will extract from the forthcoming
numerical analysis. Concrete realizations of $g_h$ are presented
and discussed in Section~\ref{ssec:ghost-penalty-realizations}.
\begin{remark}
  The idea of augmenting an unfitted finite element scheme by
  certain stabilization forms acting in the vicinity of the boundary $\Gamma$
  was first formulated in~\cite{BurmanHansbo2012,Burman2010}
  in the context of Nitsche-based fictitious domain methods for the Poisson problem
  employing continuous, piecewise polynomial ansatz functions.
  As realizations of $g_h$ typically require
  the evaluation of discrete functions $v \in V_h$ \emph{outside}
  the physical domain $\Omega$, the term \emph{ghost penalty}
  was coined in~\cite{BurmanHansbo2012,Burman2010,BeckerBurmanHansbo2009}.
\end{remark}

We conclude this section by formulating a few reasonable geometric
assumptions on $\Gamma$ and $\mcT_h$, which allow us to keep the
technical details the forthcoming numerical analysis at a moderate
level.
\def\theassumption{G\arabic{assumption}}
\begin{assumption}
  \label{ass:Gamma-C2}
  The boundary $\Gamma$ is of class $C^2$.
\end{assumption}
\begin{assumption}
  The mesh $\mcT_h$ is quasi-uniform.
  \label{ass:mesh-quasi-uniformity}
\end{assumption}
Finally,
we require that
$\Gamma$ is reasonably resolved by the active mesh $\mcT_h$.
More
specifically, for a boundary $\Gamma$ of class $C^2$ and its tubular
neighborhood
$U_{\delta}(\Gamma) = \{ x \in \RR^d : \dist(\Gamma, x) < \delta\}$ of
radius $\delta$, it is well-known, that there is a $\delta_0 > 0$ such
that $\foralls x \in U_{\delta}(\Gamma)$, there is a unique closest
point $p(x)$ on $\Gamma$ satisfying $|x - p(x)| = \dist(\Gamma, x)$,
see, e.g,~\cite{GilbargTrudinger2001}[Section 14.6].  We assume that
$h < \delta_0$ such that $\mcT_{\Gamma}$ is contained in a tubular
neighborhood for which such a closest point projection $p(x)$ is
defined. In~\cite{BurmanGuzmanSanchezEtAl2017}[Proposition 3.1] it was
then shown that the following geometric assumption on the active mesh $\mcT_h$ is
satisfied:
\begin{assumption}
  \label{ass:fat-intersection-property}
  For $T\in \mcT_{\Gamma}$ there is an element $T'$ in $\omega(T)$
  with a ``fat'' intersection\footnote{The constant $c_s$ is typically
    user-defined.} such that
  \begin{align} | T' \cap \Omega |_{d}
    \geqslant c_s |T'|_d
  \label{eq:fat-intersect-prop}
 \end{align}
for some mesh independent $c_s > 0$.
Here, $\omega(T)$ denotes the set of elements sharing at least one
node with $T$.
\end{assumption}
\begin{remark}
  The assumed quasi-uniformity of $\mcT_h$ is mostly for notional
  convenience. Except for the condition number estimates, most given estimates
  can be easily localized to element or patch-wise estimates.
\end{remark}
\subsection{Stability properties}
\label{ssec::stab-prop}
We start our theoretical investigation of the proposed cutDG
method~(\ref{eq-cutdg-formulation}) by introducing various natural
discrete norms and semi-norms. For $v \in V_h$ we define
\begin{align}
  \label{eq:sp:discrete-energy-norm}
\tn  v  \tn^2_{a_h}
&= \| \nabla v \|^2_{\mcT_h \cap \Omega}
  + \|h^{-1/2} [v] \|^2_{\mcF_h \cap \Omega},
  \\
  | v |^2_{g_h}
  &= g_h(v, v),
   \\
  \tn  v  \tn^2_{A_h} &= \tn  v  \tn^2_{a_h} + |v|^2_{g_h},
\end{align}
while for $v \in H^2(\mcT_h) + V_h$, we will also consider the norm
\begin{align}
  \tn v \tn_{a_h, \ast}^2
  = \tn  v  \tn_{a_h}^2
  +  \| h^{\onehalf} \avg{\partial_n v}\|_{\mcF_h \cap \Omega}^2
  + \| h^{\onehalf} \partial_n v\|_{\Gamma}^2.
\end{align}
Next, we show that the bilinear form $A_h$ is coercive and bounded
with respect to the discrete energy norm $\tn \cdot \tn_{A_h}$. Recall
that a pivotal ingredient in the numerical analysis of classical symmetric
interior penalty method is the inverse inequality
\begin{align}
  \| \partial_n v \|_F \leqslant C_I h_T^{-\onehalf}\| \nabla v \|_T,
  \label{eq:inverse-est-normal-flux}
\end{align}
which holds for discrete functions $v \in \PP_k(T)$ and
$F\in\mcF_T \coloneqq \{ F_h \in \mcF : F \subset \partial T\}$. Here, the inverse constant
$C_I$ depends on the dimension $d$, the degree $k$ and the shape regularity of $T$.
Unfortunately, a
corresponding inverse inequality for the cut faces $F\cap\Omega \neq F$ of the form
\begin{align*}
  \| \partial_n v \|_{F\cap\Omega} \leqslant C_I h_T^{-\onehalf}\| \nabla v \|_{T \cap \Omega}
\end{align*}
\emph{does not} hold as the ratio $\tfrac{|F \cap \Omega|_{d-1}}{|T\cap\Omega|_{d}}$
between the $d-1$ dimensional surface area of the cut face $F \cap \Omega$ and
the $d$-dimensional volume of the cut element $T\cap\Omega$
is highly dependent on the cut configuration; in fact, it
can become arbitrarily large
as illustrated by the ``sliver case'' in Figure~\ref{fig:critical-cut-cases}.
\begin{figure}[htb]
  \begin{center}
  \begin{minipage}[c]{0.49\textwidth}
    \vspace{8.5pt}
    \includegraphics[width=1.0\textwidth]{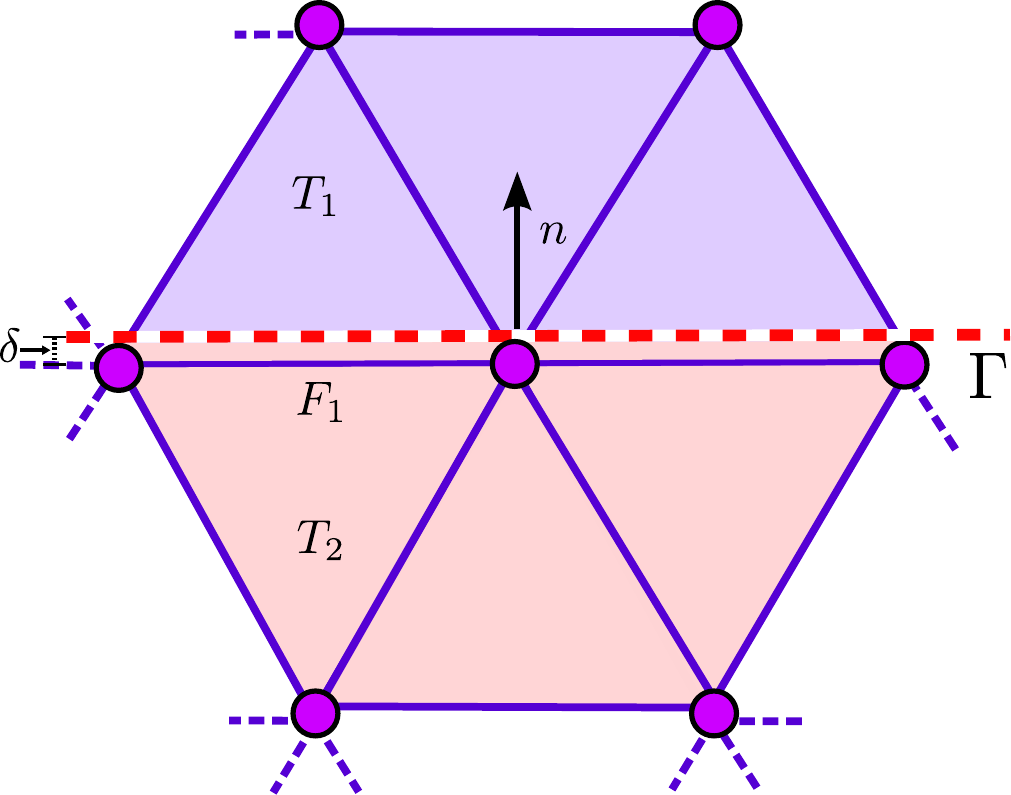}
  \end{minipage}
  \hspace{0.01\textwidth}
  \begin{minipage}[c]{0.44\textwidth}
    \vspace{0pt}
    \includegraphics[width=1.0\textwidth]{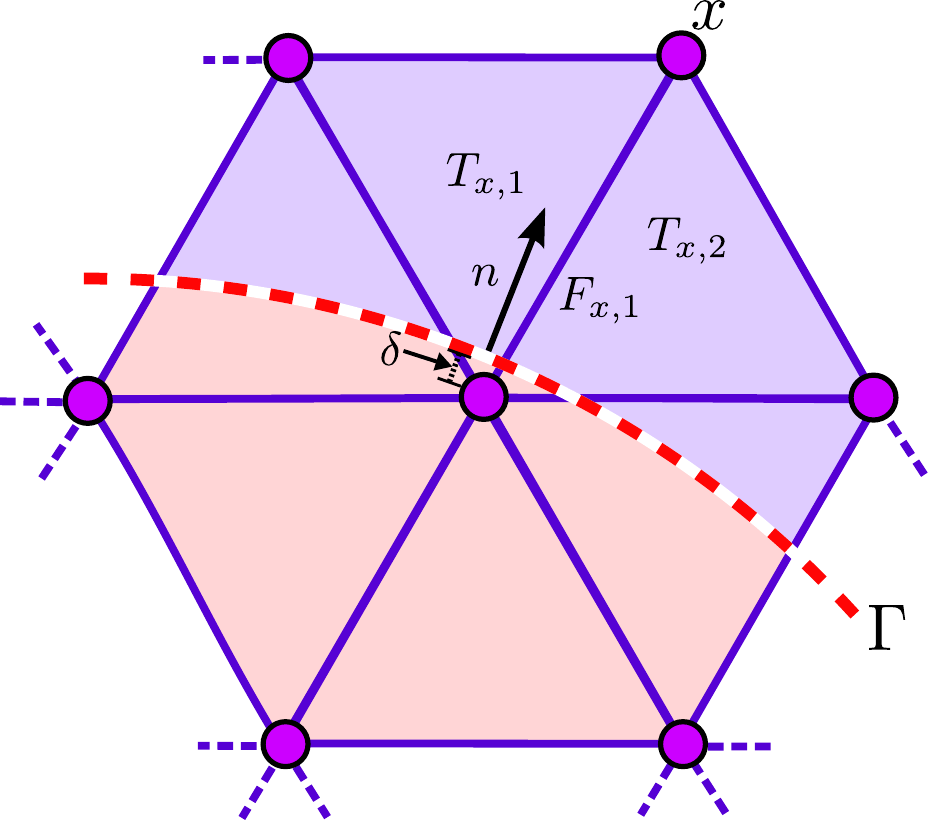}
  \end{minipage}
\end{center}
\caption{Critical cut configurations. (Left): Sliver case.  For
  $\delta \to 0$, the ratio
  $\tfrac{|\Gamma \cap T_1 |_{d-1} }{|T_1\cap \Omega|_d} \sim
  \tfrac{h^d}{\delta h^{d-1}}$ can become arbitrarily large and thus
  the corresponding local inverse constant $C$
  in~\eqref{eq:inverse-est-invalid} is practically unbounded.  A
  similar observation holds for the ratio
  $\tfrac{| F_1 |_{d-1} }{|T_1\cap \Omega|_d}$.  (Right) Dotting
  case. Observe that for the faces $F_x$ and elements $T_x$ associated
  with node $x$, the corresponding face and element related
  contributions to the stiffness matrix associated with $a_h$ become
  arbitrarily small as $\delta \to 0$. Consequently, the stiffness matrix
  is almost singular and ill-conditioned if no proper ghost penalty is added.}
  \label{fig:critical-cut-cases}
\end{figure}
As a partial replacement, one might be tempted to use simply the estimate 
\begin{align}
  \| \partial_n v \|_{F\cap \Omega}
  \leqslant 
  \| \partial_n v \|_{F}
  \leqslant C_I h_T^{-\onehalf}\| \nabla v \|_T
  \label{eq:inverse-est-normal-flux-cut-F}
\end{align}
instead. A similar issue arises when one wishes to control the normal flux on
the boundary~$\Gamma$ since an inequality of the form
  \begin{equation}
  \|\partial_n v \|_{\Gamma \cap T} \leqslant C h_T^{-\onehalf}\| \nabla v\|_{T \cap \Omega}
  \label{eq:inverse-est-invalid}
\end{equation}
cannot hold with a constant $C$ which is independent of the cut configuration.
Instead, we have only the inverse inequality
\begin{align}
  \|\partial_n v \|_{\Gamma \cap T} \leqslant C h_T^{-\onehalf}\| \nabla v\|_T
  \label{eq:inverse-est-normal-flux-Gamma}
\end{align}
at our disposal, see~\cite{HansboHansboLarson2003} for a proof.
Note that compared to constant $C_I$ in \eqref{eq:inverse-est-normal-flux-cut-F}, the constant in~\eqref{eq:inverse-est-normal-flux-Gamma} depends also on the local
curvature of $\Gamma$.
To fully exploit~(\ref{eq:inverse-est-normal-flux-cut-F}) and
(\ref{eq:inverse-est-normal-flux-Gamma}),
it is necessary to extend the control
of the $\|\nabla v\|^2_{\mcT_h\cap\Omega}$-part in natural energy norm $\tn \cdot \tn_{a_h}$
from the physical domain $\Omega$ to the entire active mesh $\mcT_h$.
This is precisely one important role of the ghost penalty term~$g_h$ which we formulate
as our first assumption on $g_h$:
\setcounter{assumption}{0}
\def\theassumption{EP\arabic{assumption}}
\begin{assumption}
  \label{ass:ghost-penalty-coerc}
  The ghost penalty $g_h$ extends the $H^1$ semi-norm
  from the physical domain $\Omega$ to the entire active mesh $\mcT_h$ in the sense that
  \begin{gather}
    \| \nabla v \|_{\mcT_h}^2
    \lesssim
    \| \nabla v \|_{\Omega}^2  + |v|_{g_h}^2
    \end{gather}
    holds for $v \in V_h$,
    with the hidden constants depending only on the dimension $d$, the polynomial order
  $k$ and the shape-regularity of $\mcT_h$.
\end{assumption}
An immediate result of our discussion is the following important corollary.
\begin{corollary}
  \label{cor:ghost-penalty-bulk}
  Let $g_h$ satisfy~\ref{ass:ghost-penalty-coerc}, then it holds that
  \begin{gather}
    \| h^{\onehalf} \partial_n  v \|_{\Gamma\cap\Omega}^2
    + \| h^{\onehalf} \partial_n v \|_{\mcF_h\cap\Omega}^2
    \lesssim
    \| \nabla v \|_{\Omega}^2  + |v|_{g_h}^2
    \lesssim \tn  v  \tn^2_{A_h} \quad \foralls v \in V_h,
    \label{eq:normal-flux-estimate}
  \end{gather}
    with the hidden constants depending only on the dimension $d$, the polynomial order
  $k$, the shape-regularity of $\mcT_h$, and the curvature of $\Gamma$. In particular, we observe that
\begin{align}
  \| v \tn_{a_h, \ast} \lesssim \| v \tn_{A_h} \quad \foralls v \in V_h.
    \label{eq:ahast-norm-est}
\end{align}
\end{corollary}

Having managed to control the normal flux on the cut geometries
$\mcF_h \cap \Omega$ and $\Gamma \cap \Omega$, we can now prove the
main result of this section.
\begin{proposition}
  \label{prop:Ah-coercivity}
  The discrete form $A_h$ is coercive and stable with respect
  to the discrete energy norm $\tn  \cdot  \tn_{A_h}$; that is,
  \begin{alignat}{3}
    \tn  v  \tn_{A_h}^2 &\lesssim A_h(v, v) && \quad \foralls v \in V_h,
  \label{eq:coercivity-Ah}
                   \\
                   A_h(v, w) &\lesssim
                   \tn  v \|_{A_h} \| v  \tn_{A_h} &&\quad \foralls v,w \in V_h.
                                 \label{eq:continuity-Ah}
  \end{alignat}
  {Moreover, for $v\in H^2(\mcT_h)+ V_h$ and $w\in V_h$, the discrete form $a_h$ satisfies }
  \begin{align}
    \label{eq:continuity-ah}
    a_h(v, w) &\lesssim
                \tn  v \|_{a_h,\ast} \| v  \tn_{A_h}.
  \end{align}
\end{proposition}
\begin{proof}
  Thanks to Corollary~\ref{cor:ghost-penalty-bulk}, the proof follows the standard arguments
  in the analysis of the classical symmetric interior penalty method.
  We start with~(\ref{eq:coercivity-Ah}).
  Setting $w = v$ in~(\ref{eq-cutdg-formulation}) and
  combining an
  $\epsilon$-Young inequality of the form $2 ab \leqslant
  \epsilon a^2 + \epsilon^{-1} b^2$ with the inverse
  estimates~(\ref{eq:inverse-est-normal-flux-cut-F}),(\ref{eq:inverse-est-normal-flux-Gamma}),
  and Corollary~\ref{cor:ghost-penalty-bulk}, we see that
  \begin{align}
    A_h(v, v)
    &= \| \nabla v \|_{\Omega}^2
      + |v|_{g_h}^2
      + \beta \|h^{-\onehalf}\jump{v}\|_{\mcF_h}^2
      + \beta \|h^{-\onehalf}v\|_{\Gamma}^2
      \nonumber
      \\
    &\quad
      - 2 (\mean{\partial_n v}, \jump{v})_{\mcF_h \cap \Omega}
      - 2 ({\partial_n v}, {v})_{\Gamma \cap \Omega}
    \\
    &\gtrsim
     \| \nabla v \|_{\Omega}^2
      + |v|_{g_h}^2
      + \beta \|h^{-\onehalf}\jump{v}\|_{\mcF_h}^2
      + \beta \|h^{-\onehalf}v\|_{\Gamma}^2
      \nonumber
      \\
    &\quad
      - \epsilon \|h^{\onehalf}\mean{\partial_n v}\|_{\mcF_h \cap \Omega}
      - \epsilon^{-1} \| h^{-\onehalf}\jump{v}\|_{\mcF_h \cap \Omega}
      - \epsilon \|h^{\onehalf}{\partial_n v}\|_{\Gamma \cap \Omega}
      - \epsilon^{-1} \| h^{-\onehalf}{v}\|_{\Gamma \cap \Omega}
    \\
    &\gtrsim
      (1-2\epsilon)
      \bigl(\| \nabla v \|_{\Omega}^2
      + |v|_{g_h}^2
      \bigr)
      + \bigl(\beta - \epsilon^{-1}\bigr)
      \bigl(
      \| h^{-\onehalf}\jump{v}\|_{\mcF_h}^2
      + \| h^{-\onehalf}{v}\|_{\Gamma}^2
      \bigr)
    \\
    &
      \gtrsim \tn  v  \tn_{A_h}^2
  \end{align}
  if we choose $\epsilon > 0$ small enough and $\beta \gtrsim \epsilon^{-1}$.
  To prove~(\ref{eq:continuity-Ah}) and~\eqref{eq:continuity-ah},
  simply apply a standard Cauchy-Schwarz inequality 
  to see that terms involving the normal fluxes are bounded by
  \begin{align}
      (\mean{\partial_n v}, \jump{w})_{\mcF_h \cap \Omega} +
    ({\partial_n v}, {w})_{\Gamma \cap \Omega}
    \lesssim
    \tn  v  \tn_{a_h,\ast}
    \tn  w  \tn_{a_h,\ast}.
  \end{align}
  Then a further application of Corollary~\ref{cor:ghost-penalty-bulk}, Eq.\eqref{eq:ahast-norm-est}
  gives the desired estimates.
  \qed
\end{proof}
\begin{remark}
  If the Dirichlet boundary
  condition~(\ref{eq:laplace-strong-dirichlet}) is replaced by a
  natural boundary condition of Neuman or Robin type, the
  \emph{continuous} cut finite element method version
  of~(\ref{eq-cutdg-formulation}) does not need any additional ghost
  penalty to guarantee discrete coercivity and optimal convergence of
  the method, but one can add a (more weakly scaled) ghost penalty to
  ensure robust condition numbers. In contrast, in the case of our cutDG formulation, a
  proper ghost penalty has to be added irrespective of the imposed
  boundary condition as the normal flux terms on cut faces also require the
  additional control formulated as Assumption~\ref{ass:ghost-penalty-coerc}.
\end{remark}

\subsection{A priori error analysis}
\label{ssec::apriori}
We turn the error analysis of
the unfitted discretization scheme~(\ref{eq-cutdg-formulation}).
  To keep the technical details at a moderate level, we assume for a
  priori error analysis that the contributions from the cut elements
  $\mcT_h \cap \Omega$, the cut faces $\mcF_h \cap \Omega$ and the
  boundary parts $\Gamma \cap \mcT_h$ can be computed exactly. For a
  thorough treatment of variational crimes arising from the
  discretization of a curved boundary 
  element methods, we refer the reader to
  \cite{LiMelenkWohlmuthEtAl2010,BurmanHansboLarsonEtAl2014,GrossOlshanskiiReusken2014}.
  
Let us review some useful inequalities needed later
and explain how to construct a suitable approximation operator.
Recall that for $v \in H^1(\mcT_h)$, the local trace inequalities of
the form
  \begin{align}
    \label{eq:trace-inequality}
    \|v\|_{\partial T}
    &\lesssim
    h_T^{-1/2} \|v\|_{T} +
    h_T^{1/2}  \|\nabla v\|_{T}
      \quad \foralls T\in \mcT_h,
    \\
    \|v\|_{\Gamma \cap T}
    &\lesssim
      h_T^{-1/2} \|v\|_{T}
      + h_T^{1/2} \|\nabla v\|_{T}
    \quad \foralls T \in \mcT_h,
    \label{eq:trace-inequality-cut}
  \end{align}
  hold, see~\cite{HansboHansboLarson2003} for a proof of the second one.
  To construct a suitable approximation operator, we depart from the $L^2$-orthogonal projection
  $\pi_h : L^2(\mcT_h) \to V_h$
  which for  $T \in \mcT_h$ and $F \in \mcF_T$ satisfies the error estimates
  \begin{alignat}{3}
    | v - \pi_h v |_{T,r}  &\lesssim h_T^{s-r} | v |_{s,T},
    && \quad
     0 \leqslant r \leqslant s,
    \\
    | v - \pi_h v |_{F,r}  &\lesssim h_T^{s-r-\onehalf} | v |_{s,T},
    && \quad 0  \leqslant r \leqslant s - 1/2,
  \end{alignat}
 whenever $v \in H^s(T)$. Now to lift a function $v \in H^s(\Omega)$ to
 $H^s(\Omega_{h}^e)$, where we for the moment use the notation 
 $\Omega_h^e = \bigcup_{T \in \mcT_h} T$,
  we recall that for Sobolev spaces $W^{m,q}(\Omega)$, $0 < m \leqslant \infty, 1 \leqslant q \leqslant \infty$,
  there is a bounded extension operator satisfying
  \begin{align}
    (\cdot)^e: W^{m,q}(\Omega) \to W^{m,q}(\Omega^e), \quad \|v^e \|_{m,q, \Omega^e} \lesssim \| v \|_{m,q,\Omega}
  \end{align}
  for $u \in W^{m,q}(\Omega)$, see~\cite{Stein1970} for a proof.
  After choosing some $\Omega^e$ such that $\Omega_{h,e} \subset \Omega^e \;\foralls h \lesssim 1$,
  we can define an ``unfitted'' $L^2$ projection variant $\pi_h^e: H^r(\Omega_h^e) \to V_h$ by setting
  \begin{align}
   \pi_h^e v \coloneqq \pi_{h} v^e.
  \end{align}
  Note that this $L^2$-projection is slightly ``perturbed'' in the
  sense that it is not orthogonal on $L^2(\Omega)$ but rather on
  $L^2(\Omega_{h}^e)$. Combining the local approximation properties of $\pi_h$ with the stability
  of the extension operator $(\cdot)^e$, we see immediately that $\pi_{h,\ast}$
  satisfies the global error estimates
  \begin{alignat}{3}
    \| v - \pi_h^e v \|_{\mcT_h,r}  &\lesssim h^{s-r} \| v \|_{s,\Omega},
    && \quad  0 \leqslant r \leqslant s,
    \label{eq:interpol-est-cut-T}
   \\
    \| v - \pi_h^e v \|_{\mcF_h,r}  &\lesssim h^{s-r-\onehalf} \| v \|_{s,\Omega},
    && \quad  0 \leqslant r \leqslant s-1/2,
    \label{eq:interpol-est-cut-F}
    \\
    \| v - \pi_h^e v \|_{\Gamma,r}  &\lesssim h^{s-r-\onehalf} \| v \|_{s,\Omega},
    && \quad  0 \leqslant r \leqslant s-1/2.
    \label{eq:interpol-est-cut-Gamma}
  \end{alignat}
  As a direct consequence, we can easily estimate the approximation error in
  the $\tn  \cdot \tn_{a_h,\ast}$-norm.
  \begin{corollary}
    \label{cor:projection-error}
    Let $u \in H^{s}(\Omega)$ and assume that $V_h = P^k(\mcT_h)$. Then
    for $r = \min\{s, k+1\}$, the approximation error of $\pi_h^e$ satisfies
    \begin{align}
    \label{eq:projection-error}
      \tn u - \pi_h^eu \tn_{a_h,\ast} \lesssim h^{r-1} \|u\|_{r,\Omega}.
    \end{align}
  \end{corollary}
  \begin{proof}
    Set $e^{\pi} = u - \pi_h^e u$ and recall that
    \begin{align}
      \tn u - \pi_h^eu \tn_{a_h,\ast}^2
      &= \| \nabla e^{\pi} \|^2_{\mcT_h \cap \Omega}
  + \|h^{-1/2} [e^{\pi}] \|^2_{\mcF_h \cap \Omega},
  +  \| h^{\onehalf} \avg{\partial_n e^{\pi}}\|_{\mcF_h \cap \Omega}^2
        + \| h^{\onehalf} \partial_n e^{\pi}\|_{\Gamma}^2
    \end{align}
    The first term can be simply estimated using~\eqref{eq:interpol-est-cut-T},
    while estimate~\eqref{eq:interpol-est-cut-F} gives the desired bounds for
    second. Finally, the last term can be treated by applying~\eqref{eq:interpol-est-cut-Gamma}.
  \qed
  \end{proof}
  Before we formulate the main a priori error estimate, we need to quantify how the
  additional stabilization term $g_h$ affects the consistency of our method.
  First note that we have the following weak Galerkin orthogonality.
\begin{lemma}[Weak Galerkin orthogonality]
  \label{lem:weak-galerkin-orth}
  Let $u \in H^2(\Omega)$ be the solution to~(\ref{eq:laplace-strong})
  and let $u_h$ be the solution to the discrete formulation~(\ref{eq-cutdg-formulation}).
  Then
  \begin{align}
    a_h(u - u_h, v) = g_h(u_h, v) \quad \foralls v \in V_h.
    \label{eq:weak-galerkin-orth}
  \end{align}
\end{lemma}
\begin{proof}
  Follows directly from the observation that $u$ satisfies $a_h(u, v) = l_h(v)\; \foralls v \in V_h$. 
  \qed
\end{proof}
Next, to assure that the remainder $g_h$ does not deteriorate the convergence order,
we formulate our second assumption on the ghost penalty $g_h$.
\begin{assumption}[Weak consistency estimate]
  \label{ass:ghost-penalty-consist}
  For $v \in H^s(\Omega)$ and $r = \min\{s,k+1\}$, the semi-norm $|\cdot|_{g_h}$
  satisfies the estimate
  \begin{align}
    |\pi_h^e v |_{g_h} \lesssim h^{r-1} \| v\|_{r,\Omega}.
    \label{eq:ghost-penalty-consist}
  \end{align}
\end{assumption}
With these preliminaries in place, we can state and prove the main a priori error estimates.
\begin{theorem}[A prior error estimates]
  Let $u \in H^s(\Omega)$, $s \geqslant 2$ be the solution to~(\ref{eq:laplace-strong})
  and let $u_h \in \PP_k(\mcT_h)$ be the solution to the discrete formulation~(\ref{eq-cutdg-formulation}).
  Then with $r = \min\{s, k+1\}$, the error $u - u_h$ satisfies
  \begin{align}
    \tn  u - u_h  \tn_{a_h,\ast} &\lesssim h^{r-1} \|u\|_{r,\Omega},
                                  \label{eq:apriori-est-energy} 
    \\
    \| u - u_h \|_{\Omega} &\lesssim h^{r} \|u\|_{r,\Omega}.
                                  \label{eq:apriori-est-l2} 
  \end{align}
\end{theorem}
\begin{proof}
  With the ``extended'' $L^2$ projection $\pi_h^e$ and the proper
  cut variants of trace inequalities in place, the proof follows
  closely the standard arguments and is included only for completeness.
  
  {\bf Estimate \eqref{eq:apriori-est-energy}.}
  First, we decompose the total error $e = u - u_h$ into
  a discrete error $e_h = \pi_h^eu - u_h$ and a projection
  error $e_{\pi} = u - \pi_h^eu$.
  Observe that 
  $\tn  u - u_h\tn _{a_h} \leqslant \| e_{\pi} \tn_{a_h,\ast}  + \| e_h \tn_{A_h}$
  and thanks to Corollary~\ref{cor:projection-error}, it is enough to estimate
  the discrete error. Combining the coercivity result~(\ref{eq:coercivity-Ah})
  with the weak Galerkin orthogonality~(\ref{eq:weak-galerkin-orth}) and
  the boundedness~(\ref{eq:continuity-ah}) yields
  \begin{align}
    \tn  e_h  \tn_{A_h}^2
    \label{eq:discrete-error-est-step-first}
    &\lesssim 
      a_h(\pi_h^eu - u_h, e_h) + g_h(\pi_h^eu - u_h, e_h)
    \\
    &=
      a_h(\pi_h^eu - u, e_h) + g_h(\pi_h^eu, e_h)
    \\
    &\lesssim
      \bigl(
      \tn \pi_h^eu - u \tn_{a_h,\ast} + |\pi_h^eu|_{g_h}
      \bigr)
      \bigl(
      \tn e_h \tn_{a_h,\ast}
      +|e_h|_{g_h}
      \bigr)
    \\
    &\lesssim
      h^{r-1}\tn u\|_{r,\Omega} \tn e_h \tn_{A_h},
    \label{eq:discrete-error-est-step-last}
  \end{align}
  where in the last step, the projection error
  estimate~(\ref{eq:projection-error}) was used again together
  with the consistency error
  assumption~\ref{ass:ghost-penalty-consist} and the norm equivalence
  $ \tn e_h \tn_{a_h,\ast}
      +|e_h|_{g_h} \sim \tn e_h \tn_{A_h}$ valid for $e_h \in V_h$.
  Now dividing
  \eqref{eq:discrete-error-est-step-last} by $\tn e_h \tn_{A_h}$ gives the
  desired estimate for the discrete error.
  
  {\bf Estimate \eqref{eq:apriori-est-l2}.}
  As usual, we employ the Aubin-Nitsche duality trick, but we need to keep track
  of the weakly consistent ghost penalty $g_h$.
  Let $\psi \in L^2(\Omega)$, then thanks to assumption~\ref{ass:Gamma-C2},
  there is a $\phi \in H^2(\Omega) \cap H_0^1(\Omega)$
  satisfying $-\Delta \phi = \psi$ and the regularity estimate
  $\| \phi \|_{2,\Omega} \lesssim \| \psi \|_{\Omega}$. Since $\phi \in H^2(\Omega)\cap H_0^1(\Omega)$,
  an integration by parts argument shows that $(e, -\Delta \phi)_{\Omega} = a_h(e,\phi)$. 
  Hence, recalling the weak Galerkin orthogonality~(\ref{eq:weak-galerkin-orth}), we see that
  \begin{align}
    (e, \psi)_{\Omega}
      &=
    a_h(e, \phi)
    \\ 
      &=
    a_h(e, \phi - \pi_h \phi)
        + g_h(u_h, \pi_h \phi)
    \\
      &=
        a_h(e, \phi - \pi_h \phi)
        + g_h(u_h - \pi_h^e u  , \pi_h \phi)
    + g_h(\pi_h^e u , \pi_h \phi)
    \\
      &\lesssim
    \tn u - u_h \tn_{a_h,*}
        \tn \phi - \pi_h \phi \tn_{a_h,*}
        +
        \tn \pi_h^e u - u_h \tn_{a_h,*}
        | \pi_h \phi |_{g_h}
        +
        | \pi_h u |_{g_h}
        | \pi_h \phi |_{g_h}
    \\
      &\lesssim
        h^{r-1} \|u\|_{r,\Omega}
        h \| \phi \|_{2,\Omega}
      \lesssim
        h^{r} \|u\|_{r,\Omega}
        \| \psi \|_{\Omega},
  \end{align}
  where in the two last steps,  a Cauchy-Schwarz inequality
  for the symmetric bilinear forms $a_h$ and $g_h$ was combined with
  the weak consistency
  assumption~\ref{ass:ghost-penalty-consist},
  estimate~(\ref{eq:apriori-est-energy})
  and the energy-norm estimate
  for the discrete error
  $\pi_h^e u - u_h$ derived in~
    \eqref{eq:discrete-error-est-step-first}--\eqref{eq:discrete-error-est-step-last}.
  Consequently, we
  found that
  \begin{align}
     \|u-u_h\|_{\Omega}
     &= \sup_{\psi \in L^2(\Omega), \|\psi\|_{\Omega}=1}
       (e, \psi)_{\Omega}
      \lesssim
        h^{r} \|u\|_{r,\Omega}.
  \end{align}\qed
\end{proof}
\begin{remark}
  Note that since the proof of the $L^2$ error estimate is based on a
  duality argument, the employed ghost penalty is not required to
  extend the $L^2$ norm to actually guarantee optimal error estimates in the $L^2$ norm.
  \label{rem:L2convergence}
\end{remark}

\subsection{Condition number estimates}
\label{ssec::condition-number-est}
We conclude the theoretical analysis of the stabilized cutDGM for the
Poisson problem by demonstrating how the addition of a suitably
designed ghost penalty $g_h$ ensures that the condition
number of the resulting stiffness matrix can be bounded by
$\mcO(h^{-2})$, irrespective of how the boundary $\Gamma$ cuts the
background mesh $\mcT_h$.

Let $\{\phi_i\}_{i=1}^N$ be the standard piecewise polynomial basis
functions associated with $V_h = \PP_k(\mcT_h)$ so that any $v \in V_h$
can be writtes as 
$v = \sum_{i=1}^N V_i \phi_i$ with coefficients $V = \{V_i\}_{i=1}^N \in \RR^N$.
The stiffness matrix $\mcA$ is defined by the relation
\begin{align}
  ( \mcA V, W )_{\RR^N}  = A_h(v, w) \quad \foralls v, w \in
  V_h.
  \label{eq:stiffness-matrix}
\end{align}
Thanks to the coercivity of $A_h$,
the stiffness matrix $\mcA$ is a bijective linear mapping 
$\mcA:\RR^N \to \RR^N$ with its operator norm and condition number defined by
\begin{align}
  \| \mcA \|_{\RR^N}
  = \sup_{V \in \widehat{\RR}^N\setminus\bfzero}
  \dfrac{\| \mcA V \|_{\RR^N}}{\|V\|_{\RR^N}}
\quad \text{and}
\quad
  \kappa(\mcA) = \| \mcA \|_{\RR^N} \| \mcA^{-1} \|_{\RR^N},
  \label{eq:operator-norm-and-condition-number-def}
\end{align}
respectively.
Following the approach in~\cite{ErnGuermond2006}, 
a bound for the condition number can be derived
by combining three ingredients. The first one consists of
the well-known estimate
\begin{align}
  h^{d/2} \| V \|_{\RR^N}
  \lesssim
   \| v \|_{L^2(\mcT_h)}
  \lesssim
   h^{d/2} \| V \|_{\RR^N},
  \label{eq:mass-matrix-scaling}
\end{align}
which holds for any quasi-uniform mesh $\mcT_h$ and $v\in V_h$, and
allows us to pass between the discrete $l^2$ norm of 
coefficient vectors $V$ and the continuous $L^2$ norm of
finite element functions $v_h$.
Second, a discrete Poincar\'e-type estimate is
needed to pass from the $L^2$ norm to the discrete energy-norm.
Finally, we need an inverse inequality
which enables us to bound the discrete energy norm 
by the $L^2$ norm.

\subsubsection{Discrete Poincar\'e estimate}
Recall that for the standard symmetric interior penalty method on fitted, quasi-uniform meshes,
the discrete Poincar\'e inequality
\begin{align}
  \| v \|_{\mcT_h}
  \lesssim
  \| \nabla v \|_{\mcT_h} + \|h^{-\onehalf} \jump{u}\|_{\mcF_h} + \|h^{-\onehalf} u \|_{\Gamma}
  \label{eq:poinc-est-fitted}
\end{align}
holds for $v \in V_h$, see~\cite{Arnold1982,DiPietroErn2012}.
Proving such an inequality in the unfitted case
is a slightly more subtle undertaking.
First, to
apply~\eqref{eq:mass-matrix-scaling} when passing between discrete
$l^2$ and continuous $L^2$-norms, we have to work again on the
entire active mesh and not only on the physical part~$\Omega$.
In particular, $\Gamma$ does not longer constitutes the boundary of
the domain under consideration as in~\eqref{eq:poinc-est-fitted}.
Second, note that the fictitious domain
$\Omega_{h}^e$ associated with the active mesh $\mcT_h$ 
changes with decreasing mesh size.
To gain control over the $L^2(\mcT_h)$ and to be able to derive a
suitable discrete Poincar\'e inequality leads us to the next
assumption on $g_h$:
\begin{assumption}
  \label{ass:assumption-jh-l2norm-est}
  The ghost penalty $g_h$ extends the $L^2$ norm
  from the physical domain to the entire active mesh~$\mcT_h$
  in the sense that
  \begin{align}
  \label{eq:assumption-jh-l2norm-est}
    \| v \|_{\mcT_h}^2
    &\lesssim
    \| v \|_{\Omega}^2 + | v |_{g_h}^2,
  \end{align}
    holds for $v \in V_h$,
    with the hidden constants depending only on the dimension $d$, the polynomial order
    $k$ and the shape-regularity of $\mcT_h$.
  \end{assumption}
\begin{proposition}[Discrete Poincar\'e inequality]
  \label{prop:poincare-disc}
  For $v \in V_h$, it holds that
 \begin{align}
 \label{eq:poincare-disc} 
  \| v \|_{\mcT_h}
  \lesssim 
  \tn  v  \tn_{A_h}.
 \end{align}
\end{proposition}
\begin{proof}
  Since $ \| v \|_{\mcT_h} \lesssim \| v \|_{\Omega} + |v|_{g_h} $ by
  Assumption~\ref{ass:assumption-jh-l2norm-est}, the main task is
  to estimate $\|v\|_{\Omega}$, which can be done by combining the
  proof given~\cite{Arnold1982} with the trace inequalities~(\ref{eq:trace-inequality}),
  (\ref{eq:trace-inequality-cut})
  and the boundedness of the extension operator
  $(\cdot)^e: W^{m,p}(\Omega) \to W^{m,q}(\RR^d)$. More specifically, 
  thanks to Assumption~\ref{ass:Gamma-C2} and the implied 
  elliptic regularity, there is a $\psi \in H^2(\Omega) \cap H^1_0(\Omega)$
  satisfying $-\Delta \psi = v$ and $\|\psi^e \|_{2,\RR^d} \lesssim \| v \|_{\Omega}$.
  Consequently,
  \begin{align}
    \|v\|_{\Omega}^2
    &= (v, -\Delta \psi)_{\Omega}
    \\
    &= (\nabla v, \nabla \psi )_{\Omega}
      -(v, \partial_n \psi )_{\Gamma}
      -(\jump{v}, \avg{\partial_n \psi} )_{\mcF_h\cap\Omega}
    \\
    &\lesssim
      \bigl(
      \| \nabla v \|_{\Omega}
      +\| h^{-\onehalf} v \|_{\Gamma}
      +\| h^{-\onehalf} \jump {v} \|_{\mcF_h \cap \Omega}
      \bigr)
      \nonumber
      \\
      &\quad\cdot
      \bigl(
      \|\nabla\psi\|_{\Omega}
      +
      \| h^{\onehalf} \partial_n \psi \|_{\Gamma}
      +
      \| h^{\onehalf} \avg{\partial_n \psi}\|_{\mcF_h \cap \Omega}
      \bigr)
    \\
    &\lesssim
      \tn v \tn_{a_h}
      \bigl(
      \|\nabla\psi^e\|_{\mcT_h} +  h \| D^2 \psi^e \|_{\mcT_h}
      \bigr)
      \\
    &\lesssim \tn v \tn_{a_h} \|v\|_{\Omega}
  \end{align}
  and thus $ \| v \|_{\mcT_h} \lesssim \| v \|_{\Omega} + |v|_{g_h} \lesssim
      \tn v \tn_{a_h} + |v|_{g_h} \sim \tn v \tn_{A_h}$ which gives the desired bound.
  \qed
\end{proof}
\begin{remark}
  We point out that in \emph{continuous} cut finite element
  formulations for the Poisson problem as proposed in~\cite{BurmanHansbo2012,Burman2010},
  Assumption~\ref{ass:assumption-jh-l2norm-est} and thus
  Proposition~\ref{prop:poincare-disc} are automatically
  satisfied when Assumption~\ref{ass:ghost-penalty-coerc}
  holds,
  thanks to a standard Poincar\'e inequality of the form
  $\| v \|_{\Omega} \lesssim \|\nabla v \|_{\Omega} + \| v \|_{\Gamma}$
  valid for $H^1$ conform elements.
  \label{rem:difference-cutDG-cutFEM-I}
\end{remark}

\subsubsection{Inverse inequalities}
We note first that, similar to~\eqref{eq:inverse-est-normal-flux-Gamma}
and~\eqref{eq:inverse-est-normal-flux-cut-F}, we have the
following inverse estimates for $v\in \PP_k(T)$ and $F \in \mcF_T$,
\begin{gather}
  \| \nabla  v \|_{T \cap \Omega}
  \lesssim
  \| h^{-1}v \|_{T},
  \qquad
  \| v \|_{\Gamma \cap T}
  \lesssim
  \| h^{-\onehalf}u \|_{T},
  \qquad
  \| u \|_{F \cap \Omega}
  \lesssim
  \| h^{-\onehalf}u \|_{T}.
  \label{eq:inverse-est-cut-grad-and-trace}
\end{gather}
To prove the desired inverse estimate for the energy norm $\tn \cdot \tn_{A_h}$,
it is natural to require that the ghost penalty itself satisfies the same type of inverse inequality:
\begin{assumption}
  \label{ass:inverse-estimate-jh}
  For $v\in V_h$ it holds that
\begin{align}
  |v|_{g_h} \lesssim h^{-1}\| v \|_{\mcT_h},
  \label{eq:inverse-estimate-jh}
\end{align}
with the hidden constant independent of the particular configuration.
\end{assumption}
Then combining the inverse estimates~\eqref{eq:inverse-est-cut-grad-and-trace}
with Assumption~\ref{ass:inverse-estimate-jh}, it is straightforward to prove the next lemma.
\begin{lemma}[Inverse estimate for $\tn \cdot \tn_{A_h}$]
  \label{leminverse-estimate-Ah}
  For $v \in V_h$, it holds
  \begin{align}
    \tn  v  \tn_{A_h} \lesssim h^{-1} \| v \|_{\mcT_h}
  \label{eq:inverse-estimate-Ah}
  \end{align}
  with the hidden constant only depending on the dimension $d$, the polynomial degree $k$
  and the shape regularity of $\mcT$, but not on the particular cut configuration.
\end{lemma}
\begin{remark}
  Note again, that for a cut face with $F\cap T \subsetneq F$ (and
  similar for $\Gamma \cap T$), an inverse estimate of the form
  $\| v \|_{F \cap \Omega} \lesssim \| h^{-\onehalf}v
  \|_{T\cap\Omega}$ \emph{cannot} hold for arbitrary cut
  configurations, thus the application of the inverse
  estimates~\eqref{eq:inverse-est-cut-grad-and-trace}
  forces us to pass from the
  physical domain~$\Omega$ to the entire active mesh $\mcT_h$.
\end{remark}

\subsubsection{The condition number estimate}
Finally, we combine the mass-matrix scaling~\eqref{eq:mass-matrix-scaling},
the discrete Poincar\'e inequality~\eqref{eq:poincare-disc} and the
inverse estimate~\eqref{eq:inverse-estimate-Ah}
to derive a geometrically robust conditon number bound.
\begin{theorem} 
  \label{thm:condition-number-estimate}
  The condition number of the stiffness matrix satisfies
  the estimate
\begin{equation}
\kappa( \mcA )\lesssim h^{-2},
\end{equation}
where the hidden constant depends only on the
dimension $d$, the polynomial order $k$ and the quasi-uniformity of
$\mcT_h$, but not on the particular cut configuration.
\end{theorem}
\begin{proof} We need to bound $\| \mcA \|_{\RR^N}$ and $\| \mcA^{-1} \|_{\RR^N}$. 
 First observe that for $w \in V_h$,
\begin{equation}
  \tn  w  \tn_{A_h}
  \lesssim h^{-1} \| w \|_{\mcT_h}
  \lesssim h^{(d-2)/2}\|W\|_{\RR^N},
\end{equation}
where the inverse estimate~\eqref{eq:inverse-estimate-Ah}
and equivalence~\eqref{eq:mass-matrix-scaling}
were successively used.
Thus
\begin{align}
  \| \mcA V\|_{\RR^N} &= \sup_{W \in \RR^N \setminus \{0\}} 
  \frac{( \mcA V, W)_{\RR^N}}{\| W \|_{\RR^N}}
  = \sup_{w \in V_h \setminus\{0\}}  \frac{A_h(v,w)}{\tn  w  \tn_{A_h}} 
  \frac{\tn  w  \tn_{A_h}}{\| W \|_{\RR^N}}
\\
&\lesssim h^{(d-2)/2} \tn  v  \tn_{A_h} 
\lesssim h^{d-2}\|V\|_{\RR^N},
\end{align}
and thus 
$
\| \mcA \|_{\RR^N} \lesssim h^{d-2}
$
by the definition of the operator norm.
To estimate $\| \mcA^{-1}\|_{\RR^N}$,
start from \eqref{eq:mass-matrix-scaling} and combine the Poincar\'e
inequality~\eqref{eq:poincare-disc} with 
a Cauchy-Schwarz inequality to arrive at the following chain of
estimates:
\begin{align}
  \| V \|^2_{\RR^N} 
  \lesssim h^{-d} \| v \|^2_{\mcT_h} 
  \lesssim h^{-d} A_h(v,v) 
  = h^{-d} (V, \mcA V)_{\RR^N}
  \lesssim h^{-d} \| V \|_{\RR^N} \| \mcA V \|_{\RR^N}
\end{align}
and hence $\| V \|_{\RR^N} \lesssim h^{-d}\| \mcA V\|_{\RR^N}$. 
Now setting $ V = \mcA^{-1} W$ we conclude that
 $
 \| \mcA^{-1}\|_{\RR^N} \lesssim h^{-d}
 $
and combining the estimates for $\| \mcA\|_{\RR^N}$ and $\| \mcA^{-1}\|_{\RR^N}$ the theorem follows.
\qed
\end{proof}

\subsection{Ghost penalty realizations}
\label{ssec:ghost-penalty-realizations}
The goal of this section to discuss possible realizations of the ghost
penalty operator $g_h$ which meet the assumptions we made while
developing the theoretical properties of the cutDG
formulation~(\ref{eq-cutdg-formulation}):
\begin{itemize}
\item  {\bf EP1} $H^1$ semi-norm extension property for $v \in V_h$,
  \begin{align}
    \| \nabla v \|_{\mcT_h}
    \lesssim
    \| \nabla v \|_{\Omega}  + |v|_{g_h}
    \label{eq:ass-ep1}
  \end{align}
\item  {\bf EP2} Weak consistency for $v \in H^s(\Omega)$ and  $r = \min\{s,k+1\}$,
  \begin{align}
    |\pi_h^e v |_{g_h} \lesssim h^{r-1} \| v\|_{r,\Omega}
    \label{eq:ass-ep2}
  \end{align}
\item  {\bf EP3} $L^2$ norm extension property for $v \in V_h$,
  \begin{align}
    \| v \|_{\mcT_h} 
    &\lesssim
      \| v \|_{\Omega} + | v |_{g_h}
      \label{eq:ass-ep3b}
  \end{align}
\item {\bf EP4} Inverse inequality for $v \in V_h$,
  \begin{align}
    |v|_{g_h} \lesssim h^{-1} \| v \|_{\mcT_h}
  \end{align}
\end{itemize}
\begin{remark}
  Note that only the first two assumptions are needed to guarantee optimal convergence properties,
  while the last two ones are required to ensure that the condition number of the system matrix
  scales as in the fitted mesh case.
\end{remark}
 \begin{remark}
  \label{rem:difference-cutDG-cutFEM-II}
  Thus the necessity of
  Assumption~\ref{ass:assumption-jh-l2norm-est} reflects a subtle difference between
  continuous and discontinuous cut finite element methods, see
  also Remark~\ref{rem:difference-cutDG-cutFEM-II}.
\end{remark}
    
We start by discussing cutDG variants of the face based ghost
penalties introduced in~\cite{BeckerBurmanHansbo2009,BurmanHansbo2012}
and their higher-order generalizations proposed and analyzed
in~\cite{Burman2010,BurmanHansbo2012,MassingLarsonLoggEtAl2013a}.
Afterwards, we briefly review variants of the ghost penalty proposed
in~\cite{Burman2010,BurmanHansbo2013} which are based on a local
projection stabilization.

\subsubsection{Face-based ghost penalties}
\begin{figure}
 \begin{center}
    \begin{minipage}[t]{0.45\textwidth}
    \vspace{0pt}
    \includegraphics[width=1.0\textwidth]{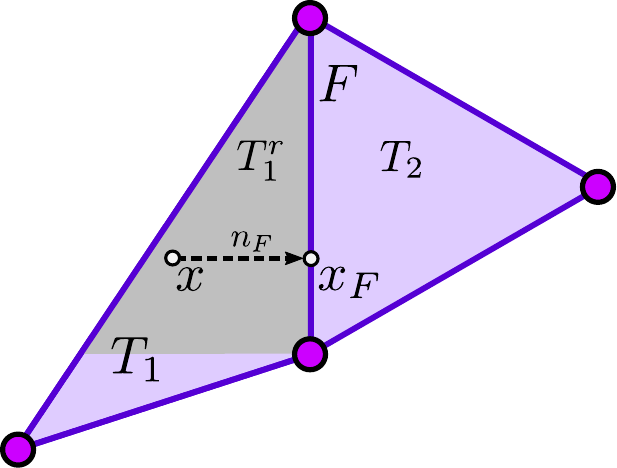}
  \end{minipage}
  \hspace{0.02\textwidth}
  \begin{minipage}[t]{0.45\textwidth}
    \vspace{0pt}
    \includegraphics[width=1.0\textwidth]{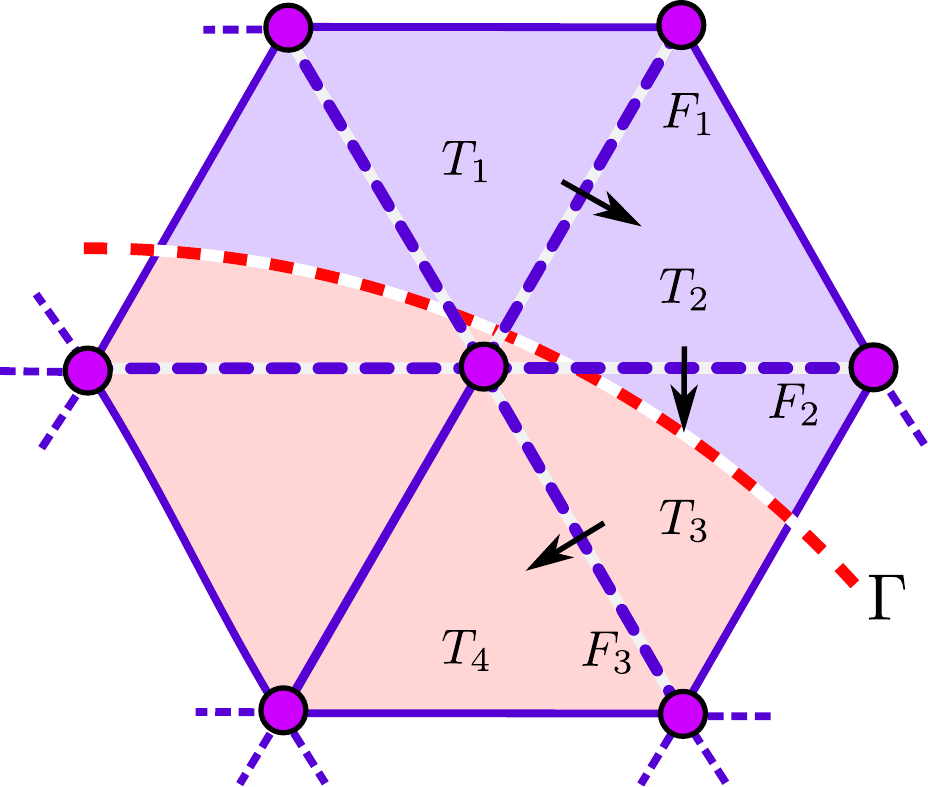}
  \end{minipage}
\end{center}
  \caption{Controlling the $L^2$-norm $\| v \|_{T_1}$
  of a finite element function $v$
  on a barely intersected, ``fictitious'' element $T_0$
  by $\| v \|_{T_4}$ and boundary zone jump-penalties.
  Starting from $T_0$, each term $\|v\|_{T_i}^2$
  can be estimated by the neighboring term
  $\| v \|_{T_{i+1}^2}$ when a sum of
  jump-terms of the form $h_{F_{i+1}}^{2j+i} \|\nablan v \|_{F_{i+1}}^2$
  is added.
}
\label{fig:boundary-zone-zoom}
\end{figure}

As a first step, we recall from~\cite{MassingLarsonLoggEtAl2013a} how
the local $L^2$ control of $v \in V_h$ can be passed between elements
by adding a penalty of the jumps of all higher order normal
derivatives.
\begin{lemma}
  \label{lem:l2norm-control-via-jumps}
  Let $T_1,\, T_2 \in \mcT_h$ be two elements sharing a common face $F$.
  Then for $v_h \in V_h$ the inequality
  \begin{equation}
    \| v \|_{T_1}^2 \lesssim
    \|v\|_{T_2}^2
    + \sum_{0 \leqslant j \leqslant k} h^{2 j + 1}
    (\jump{\partialbar_n^{j}
      v},\jump{\partialbar_n^{j} v})_F
    \label{p2:eq:l2norm-control-via-jumps}
  \end{equation}
  holds
  with the hidden constant depending only on the
  shape-regularity of $\mcT_h$, the polynomial order $k$,
  and the  dimension d. Here, we used the notation
  $\partialbar_n^j v \coloneqq \sum_{| \alpha | = j}\tfrac{D^{\alpha} v(x)n^{\alpha}}{\alpha !}$
   for multi-index $\alpha = (\alpha_1, \ldots,
   \alpha_d)$, $|\alpha| = \sum_{i} \alpha_i$ and $n^{\alpha} =
   n_1^{\alpha_1} n_2^{\alpha_2} \cdots n_d^{\alpha_d}$.
 \end{lemma}
   For the reader's convenience, we include a slightly improved version of the proof from~\cite{MassingLarsonLoggEtAl2013a}.
\begin{proof}
  Let $n_F$ be the inward pointing face normal vector associated with $F$.
  For a given point $x \in T_1$, we write $x_F = x_F(x)$
  for the normal projection of $x$ onto the plane defined by the
  face $F$. 
  We define the set 
  \begin{align}
    T_1^r = \{ x  \in T_1 \st x_F + t(x -  x_F) \in T_1 \;\foralls t \in [0,1] \},
  \end{align}
  see also Figure~\ref{fig:boundary-zone-zoom} (left).
  By shape regularity and a finite dimension argument,
  there is a constant~$C_1$ such that
  \begin{align}
    \| v \|_{T_1} \leqslant C_1 \| v \|_{T_1^r}.
  \end{align}
  Denote by $v_i$ the uniquely and globally defined polynomial
  satisfying $v_i |_{T_i} = v |_{T_i}$.
  For $i = 1, 2$ and $x \in T_1^r$, we may
  express $v_i(x)$ in terms of its
  Taylor expansion around $x_F$,
  \begin{equation}
    v_i(x) = \sum_{|\alpha| \leqslant k}
    \dfrac{D^{\alpha}v_i(x_F)}{\alpha !} (|x-x_F| n )^{\alpha},
  \end{equation}
  and subtracting the two Taylor expansions, we find that
  \begin{align}
    v_1(x)
    &= v_2(x)
      + \sum_{|\alpha| \leqslant k}
      \dfrac{\jump{D^{\alpha}v(x_F)n^{\alpha}}} {\alpha !} |x-x_F|^{\alpha}
    = v_2(x)
      + \sum_{j=0}^{k}
     |x-x_F|^{j} \jump{\partialbar_n^jv(x_F)}.
      \label{eq:taylor-expansion-subtracted}
  \end{align}
  After taking squares of the
  identity~\eqref{eq:taylor-expansion-subtracted}, multiple
  applications of a Cauchy-Schwarz inequality of the form
  $(\sum_i a_i b_i)^{2} \leqslant \sum a_i^2 \sum_i b_i^2$ show that
  \begin{align}
    v_1^2(x)
      \leqslant
      2v_2^2(x)
      + 2 (k+1)
      \sum_{j=0}^{k}
      |x-x_F|^{2j} \jump{\partialbar_n^jv(x_F)}^2.
      \label{eq:gp-pointwise}
  \end{align}
  Now, we integrate~\eqref{eq:gp-pointwise} over $T_1^r$ with respect to $x$
  and exploit the fact that $x_F(x)$ is constant when integrating in face normal direction.
  With the height of $T_1^r$ over $F$ being bounded by $h$, we thus find that
  \begin{align}
    C_1^{-1}
    \| v_1 \|_{T_1}^2
    &\leqslant 
    \| v_1 \|_{T_1^r}^2
    \leqslant 2
    \|v_2\|_{T_1^r}^2
      + 2(k+1)
      \sum_{j=0}^k
      h^{2j}
      \int_{T_1^r}
      | \jump{ \partialbar_n^j v( x_F(x) }|^2
      \dx
      \\
    &\leqslant 2
    \|v_2\|_{T_1^r}^2
      + 2(k+1)
      \sum_{j=0}^k
      h^{2j+1}
      \int_{F}
      | \jump{ \partialbar_n^j v( x_F) }|^2
      \dx_F,
      \\
    &\leqslant
      2 
      \max\{C_2, k+1\}
    \bigg( 
    \|v_2\|_{T_2}^2
      + \sum_{j=0}^k
      h^{2j+1}
      \| \jump{\partialbar_n^j v} \|_F^2
      \biggr),
  \end{align}
  where in the last step, we again used the fact that the inequality
  $\| v_2 \|_{T_1^r} \leqslant C_2 \| v_2 \|_{T_2}$ holds with
  some constant $C_2$ which only depends on the shape regularity parameter, the dimension $d$ and the
  order $k$. 
\qed
\end{proof}
The previous lemma is the main motivation to introduce the set of ghost
penalty faces
\begin{align}
  \mcF_h^g = \{ F \in \mcF_h:  T^+ \cap \Gamma \neq \emptyset \lor T^- \cap \Gamma \neq \emptyset \},
  \label{eq:faces-gp-bvp}
\end{align}
that is, the set of interior faces in active mesh belonging to
elements which are intersected by the boundary $\Gamma$.
See also Figure~\ref{fig:domain-set-up} (right), where
the of ghost penalty faces are represented by dashed lines.
Thanks to the
geometric assumptions~\ref{ass:mesh-quasi-uniformity}
and~\ref{ass:fat-intersection-property},
we deduce that there is a uniformly bounded
maximal number of ghost penalty faces which have to be crossed to
``walk'' from any element $T \in \mcT_{\Gamma}$ to an element $T'$ satisfying
the fat intersection property~(\ref{eq:fat-intersect-prop}),
see also Figure~\ref{fig:boundary-zone-zoom} (right).
Thus we have derived the first estimate of the following lemma.
\begin{lemma}
  \label{lem:face-based-gp}
  For $v \in V_h$, it holds that
  \begin{gather}
    \| v \|_{\mcT_h}^2
    \lesssim \| v \|_{\Omega}^2
    + \sum_{j=0}^k
    \sum_{F \in \mcF_h^g} h^{2j+1}
    (\jump{\partialbar_n^j v}, \jump{\partialbar_n^j v})_F,
    \label{eq:facet-gp-l2-control}
    \\
    \| \nabla v \|_{\mcT_h}^2
    \lesssim \| \nabla v \|_{\Omega}^2
    + \sum_{j=0}^k
    \sum_{F \in \mcF_h^g} h^{2j-1}( \jump{\partialbar_n^j v}, \jump{\partialbar_n^j v})_{F},
    \label{eq:facet-gp-h1-control}
  \end{gather}
  with the hidden constant only depending on the polynomial order $k$, the quasi-uniformity
  of $\mcT_h$ and the dimension $d$, but not on the particular cut configuration.
\end{lemma}
\begin{proof}
  It only remains to show the second inequality.
  The prove \eqref{eq:facet-gp-h1-control}, simply replace $v$ by
  $\nabla v$ in~\eqref{eq:facet-gp-l2-control} in a first step, yielding
  \begin{align}
    \| \nabla v \|_{\mcT_h}^2
    \lesssim \| \nabla v \|_{\Omega}^2
    + \sum_{j=0}^k
    \sum_{F \in \mcF_h^g} h^{2j+1}
    (\jump{\nabla \partialbar_n^j v}, \jump{ \nabla \partialbar_n^j v})_F.
    \label{eq:facet-gp-h1-full-grad}
  \end{align}
  In the second step, decompose
\mbox{$\nabla v = (\partial_n v) n_F + {P}_{F} \nabla v$}
into its facet normal and face tangential part
using the tangential projection
\mbox{${P}_F := I - n_F \otimes n_F$}
and then employ
the inverse estimate
\begin{align}
\|\jump{P_{F} \nabla \partialbar_n^j v} \|_F^2 = \|P_{F} \nabla \jump{\partialbar_n^j v} \|_F^2 \lesssim h^{-2} \|\jump{\partialbar_n^j v} \|_F^2 
\end{align}
on the tangential part to show that on each face $F$, we have
\begin{align}
  h^{2j+1}
  \|
  \jump{\partialbar_n^j \nabla v}
  \|_F^2
  \lesssim
  h^{2j+1}
  \|
  \jump{\partialbar_n^{j+1} v}
  \|_F^2
  + 
  h^{2j-1}
  \|
  \jump{\partialbar_n^{j} v}
  \|_F^2.
\end{align}
which establishes estimate~\eqref{eq:facet-gp-h1-control}.
\qed
\end{proof}
\begin{proposition}
  \label{prop:gh1-properties}
  For any set of positive parameters $\{\gamma_j\}_{j=0}^k$,
  the ghost penalty $g^1_h$ defined by
\begin{align}
  g_h^1 (v,w)
  &\coloneqq \sum_{j=0}^k \sum_{F\in F_h^g}  \gamma_j h_F^{2j-1} (\jump{\partialbar_n^j v}, \jump{\partialbar_n^j w})_F,
  \label{eq:jh-def-face-based}
\end{align}
for $v,w \in V_h$ satisfies Assumption~\ref{ass:ghost-penalty-coerc}--\ref{ass:inverse-estimate-jh}.
\end{proposition}
\begin{remark}
  From a theoretical perspective, any pair of parameters choices
  $\{\gamma_j\}_{j=0}^k$ leads to equivalent discrete norms and thus
  we simply assume that $\gamma_0 = \ldots = \gamma_k = 1$ in all
  relevant proofs presented in this work. From a  practical
  perspective, the choice of $\{\gamma_j\}_{j=0}^k$ will clearly
  affect the final constants in the derived estimates and
  consequently, the numerically observed robustness and accuracy of
  the method.  In all our conducted numerical experiments, we chose
  $\beta=\gamma_0 = 50$ and $\gamma_i = 0.1$ for $i=1,2,3$ as a good
  compromise between accuracy, robustness and conditioning.
\end{remark}
\begin{proof}[Proposition~\ref{prop:gh1-properties}]
  Thanks to Lemma~\ref{lem:face-based-gp}, $g_h^1$ has both the $H^1$ and $L^2$ norm extension
  property and it only remains to show
  that~\ref{ass:ghost-penalty-consist} and \ref{ass:inverse-estimate-jh} are satisfied.
  Starting with the weak consistency estimate~(\ref{eq:ghost-penalty-consist}), we
  let $v \in H^s(\Omega)$ and set $r = \min\{s, k+1\}$.
  Then
  \begin{align}
    |\pi_h^e v|_{g_h^1}^2
    &=
      \sum_{j=0}^k h^{2j-1} \| \jump{\partialbar_n^j \pi_h^e v} \|_{\mcF_h^g}^2
      \\
    &=
      \sum_{j=0}^{r-1} h^{2j-1} \| \jump{\partialbar_n^j ( \pi_h^e v - v^e)} \|_{\mcF_h^g}^2
      +\sum_{j=r}^k h^{2j-1} \| \jump{\partialbar_n^j \pi_h^e v} \|_{\mcF_h^g}^2
      \label{eq:gh1-weak-consist-proof-step-1}
      \\
    &\lesssim
      h^{2r-2} \| v \|_{r, \Omega}^2 
      + h^{2j-2} \| D^r \pi_h^e v \|_{\mcT_h}^2
      \\
    &\lesssim
      h^{2r-2} \| v \|_{r, \Omega}^2.
  \end{align}
  where we combined
  the fact that $\jump{\partialbar_n^j  v^e}|_F = 0$ for $0 \leqslant j \leqslant r-1$
  and the approximation property~(\ref{eq:interpol-est-cut-F})
  to estimate the first sum appearing in~\eqref{eq:gh1-weak-consist-proof-step-1}.
  The second sum in~\eqref{eq:gh1-weak-consist-proof-step-1} was
  treated by successively employing 
  an inverse inequality of the form
  \begin{align}
    \| \partial_n^j v \|_{F} \lesssim 
    h^{r - j- \onehalf} \| D^r v \|_{T}, \quad v \in V_h,
    \label{eq:inverse-est-higher-order}
  \end{align}
  and the stability of the projection operator $\pi_h $ and the Sobolev extension operator
  in the $H^r$ norm. Finally,
  to establish the inverse estimate~\eqref{eq:gh1-weak-consist-proof-step-1},
  simply use~\eqref{eq:inverse-est-higher-order} with $r=0$.
  \qed
\end{proof}
\begin{remark}
  The previous proof shows that if  $v \in H^{k+1}(\Omega)$, the ghost penalty $g_h^1$ is
  in fact consistent since then $|v^e|_{g_h^1} = 0$.
\end{remark}

\begin{remark} \label{bvp:rem:alternative_penalty}
  A closer inspection of the proofs of Lemma~\ref{lem:face-based-gp}
  and Proposition~\ref{prop:gh1-properties} reveals that for
  $v \in \PP_1(\mcT_h)$, the ghost penalty
  \begin{align}
  \widetilde{g}_h^1(v,w)
  =
    \sum_{F \in \mcF_h^g} h( \jump{\nabla v}, \jump{\nabla w})_{F},
  \end{align}
  penalizing the \emph{full} gradient jump
  satisfies all required assumptions except the $L^2$ extension
  property~\ref{ass:assumption-jh-l2norm-est}.
  As a consequence, we would obtain geometrically robust a priori estimates
  but not robust condition number bounds when using $\widetilde{g}_h^1$, see
  also Remark~\ref{rem:difference-cutDG-cutFEM-I} and~\ref{rem:difference-cutDG-cutFEM-II}.
\end{remark}

\subsubsection{Projection-based ghost penalties}
To avoid the inconvenient evaluation of higher order
normal derivatives appearing in the ghost penalty defined
by~(\ref{eq:jh-def-face-based}) for polynomial order $k > 1$,
\citet{Burman2010} and~\citet{BurmanHansbo2013} proposed ghost penalty
based on local projections, which we briefly review here in the
context of cut discontinuous Galerkin methods.  Let $P$ be a patch of
$\diam(P) \lesssim h$ containing the two elements $T_1$ and $T_2$ and
define the projection $\pi_{P} : L^2(P) \to \PP_{k}(P)$ to be the
$L^2$~projection onto the space of polynomials of order $k$ defined on
the patch~$P$.  Then the next lemma first stated and proven
in~\cite{Burman2010} shows that the $L^2$ norm of $v\in V_h$ can be passed
from $T_1$ to $T_2$ by penalizing the deviation of $v |_{P}$ from
its $L^2$ projection $\pi_{\mcP} v \in \PP_k(P)$. 
\begin{lemma}
  \label{eq:gp-patch-transfer}
  Let $T_1, T_2 \in P$ and $\diam(P) \lesssim h$. Then for $v \in V_h$ its holds that
  \begin{align}
    \| v \|_{T_1}^2
    &\lesssim
      \| v \|_{T_2}^2 +  \|v - \pi_{P} v\|_{P}^2,
      \label{eq:gp-l2-ext-projection}
    \\
    \| \nabla v \|_{T_1}^2
    &\lesssim
    \|\nabla  v \|_{T_2}^2 +  h^{-2}\|v - \pi_{P} v\|_{P}^2,
      \label{eq:gp-h1-ext-projection}
  \end{align}
  with the hidden constant only depending on the shape regularity of $\mcT_h$ the polynomials order $k$
  and the dimension $d$.
\end{lemma}
\begin{proof}
  For a proof we refer to~\citet{Burman2010}.
  \qed
\end{proof}
The previous lemma motivates the definition of a patch-wise local
projection stabilization $g_P$ and its corresponding (global) ghost
penalty~$g_h^2$ by setting
\begin{align}
  g_{P}(v,w) = h^{-2} (v - \pi_{P}v, w - \pi_{P}w)_{P},
                  \qquad
  g_h^2(v,w) = \sum_{P \in \mcP} g_{P}(v,w),
\end{align}
with $v,w \in V_h$.  By choosing a suitable collections of patches
$\mcP = \{P\}$, the previous lemma can now be applied in a number of
ways to define ghost penalties for the cutDG
formulation~(\ref{eq-cutdg-formulation}).  For instance, a local
projection stabilized analogue to the jump penalty based stabilization
$g_h^1$ can be obtained by defining the patch $P(F) = T^+ \cup T^-$
for two elements $T^+$, $T^-$ sharing the (interior) face $F$ and
setting
\begin{align}
  \mcP_1 = \{P(F) \}_{F\in \mcF_h^g}.
\end{align}
A second possibility is to simply use neighborhood patches $\omega(T)$,
\begin{align}
  \mcP_2 = \{\omega(T) \}_{T \in \mcT_{\Gamma}}.
\end{align}
Finally, one can mimic the cell agglomeration approach taken in
classical unfitted discontinuous Galerkin
approaches~\cite{BastianEngwer2009,HeimannEngwerIppischEtAl2013,SollieBokhoveVegt2011,JohanssonLarson2013}
by associating to each cut element $T \in \mcT_{\Gamma}$ with a small
intersection $|T \cap \Omega|_d \leqslant c_s h_T^d$ an element
$T' \in \omega(T)$ satisfying the fat intersection property
$|T' \cap \Omega|_d \geqslant c_s h_{T'}^d$.  Setting the ``agglomerated
patch'' $P_a(T)$ to $P_a(T) = T \cup T'$, a proper collection of patches
is given by
\begin{align}
  \mcP_3 = \{P_a(T)  \st T \in \mcT_{\Gamma} \wedge  |T \cap \Omega| \leqslant c_s h_T^d\}.
\end{align}

\begin{lemma}
  For $\mcP \in \{\mcP_1, \mcP_2, \mcP_3\}$, the resulting projection
  based ghost penalty $g_h^2$ satisfies
  Assumption~\ref{ass:ghost-penalty-coerc}--\ref{ass:inverse-estimate-jh}.
\end{lemma}
\begin{proof}
  Assumption~\ref{ass:ghost-penalty-coerc} and~\ref{ass:assumption-jh-l2norm-est}
  are clearly satisfied thanks to Lemma~\ref{eq:gp-patch-transfer}
  and the geometric assumption~\ref{ass:fat-intersection-property}.
  The inverse inequality~(\ref{eq:inverse-estimate-jh}) clearly holds
  by the very definition of $g_{P}$, so we are left with verifying
  Assumption~\ref{ass:ghost-penalty-consist}.
  Again, for $v \in H^s(\Omega)$ set $r = \min\{s, k+1\}$.
  Thanks to the approximation and stability properties of $\pi_h^e$ and $\pi_{P}$,
  the term  $| \pi_h^e v |_{g_h^2}^2$ can be bounded by
  \begin{align}
    | \pi_h^e v |_{g_h^2}^2
    &=  h^{-2}\| \pi_h^ev - \pi_{P} \pi_h^e v \|_{\mcP}^2
      \\
    &\lesssim
      h^{-2}
      \Bigl(
      \| \pi_h^ev - v^e \|_{\mcP}^2
      + \| v^e - \pi_{P} v^e \|_{\mcP}^2
      + \| \pi_{P} v^e - \pi_{P} \pi_h^e v \|_{\mcP}^2
      \Bigr)
      \\
    &\lesssim
      h^{-2}
      \Bigl(
       h^{2r} \| v \|_{r,\Omega}^2
       + h^{2r} \| v \|_{r,\Omega}^2
      + \| v^e - \pi_h^e v \|_{\mcP}^2
      \Bigr)
      \\
    &\lesssim
       h^{2r-2} \| v \|_{r,\Omega}^2,
  \end{align}
  where we also used the fact that the number of patch overlaps is uniformly bounded.
  \qed
\end{proof}
 \subsection{Numerical examples}
 \label{ssec:numerical-examples-bvp}
 This section is devoted to corroborate our theoretical analysis by a
 number of numerical experiments.  First, a convergence rate study
 for various approximation orders is conducted. Afterwards, we take
 a closer look at how the choice of the 
 stabilization parameters affects the geometrical robustness of the
 energy error and the condition number of the overall system matrix.
 
 \subsubsection{Convergence rate studies}
\label{ssec:bvp-conv-analysis}
 As a first test case, we numerically solve the Poisson problem~\eqref{eq:laplace-strong}
 posed on the flower-like domain 
 \begin{align}
   \Omega &= \{ (x,y) \in \RR^2 \st \phi(x,y) < 0 \} \quad \text{with }
   \phi(x,y) = \sqrt{x^2 + y^2} - r_0 - r_1 \cos(\atantwo(y,x)),
 \end{align}
 setting $r_0 = 0.6$ and $r_1 = 0.2$. The analytical reference solution given by
 \begin{align}
   u(z,y) =  \cos(2 \pi x)\cos(2\pi y) + \sin(2\pi x)\sin(2\pi y).
 \end{align}
 Starting from an initial tesselation $\widetilde{\mcT}_0$ of the
 domain $\Omega_0 = [-1.1,1.1]^2 \supset \Omega$, we generate a
 sequence of meshes $\{\mcT_k\}_{k=0}^{4}$ with mesh size
 $h_k = 2.2\cdot 2^{-3-k}$ by successively refining
 $\widetilde{\mcT}_0$ and extracting the resulting active background
 mesh. On each mesh $\mcT_k$, we compute the numerical solution
 $u_k^p \in \PP_p(\mcT_k)$ to~(\ref{eq-cutdg-formulation})
 using the ghost penalty $g_h^1$ defined by~(\ref{eq:jh-def-face-based})
 together with the parameter set $\beta = \gamma_0 = 50.0$
 and $\gamma_p = 0.1$ for $p=1,2,3$.
 Based on the manufactured solution~$u$,
 the experimental order of
 convergence (EOC) is calculated by
\begin{align*}
    \text{EOC}(k,p) = \dfrac{\log(E_{k-1}^p/E_{k}^p)}{\log(h_{k-1}/h_k)},
\end{align*}
where $E_k^p = \| e_k^p \| = \| u - u_k^p \| $ denotes the  error of the numerical
approximation $u_k^p$ measured in a certain norm $\|\cdot \|$. In our convergence tests,
we consider both the $\| \cdot \|_{H^1(\Omega)}$ and the $\|\cdot\|_{L^2(\Omega)}$ norm.
For each polynomial order $p=1,2,3$, we plot the resulting errors
against the corresponding mesh size $h_k$ in a $\log$-$\log$-plot
which confirms the theoretically predicted convergence rates
for both the $H^1$ and $L^2$ error, see Figure~\ref{fig:bvp:convergence_plots}.
\begin{figure}[htb]
  \begin{center}
    \includegraphics[width=0.47\textwidth,page=1]{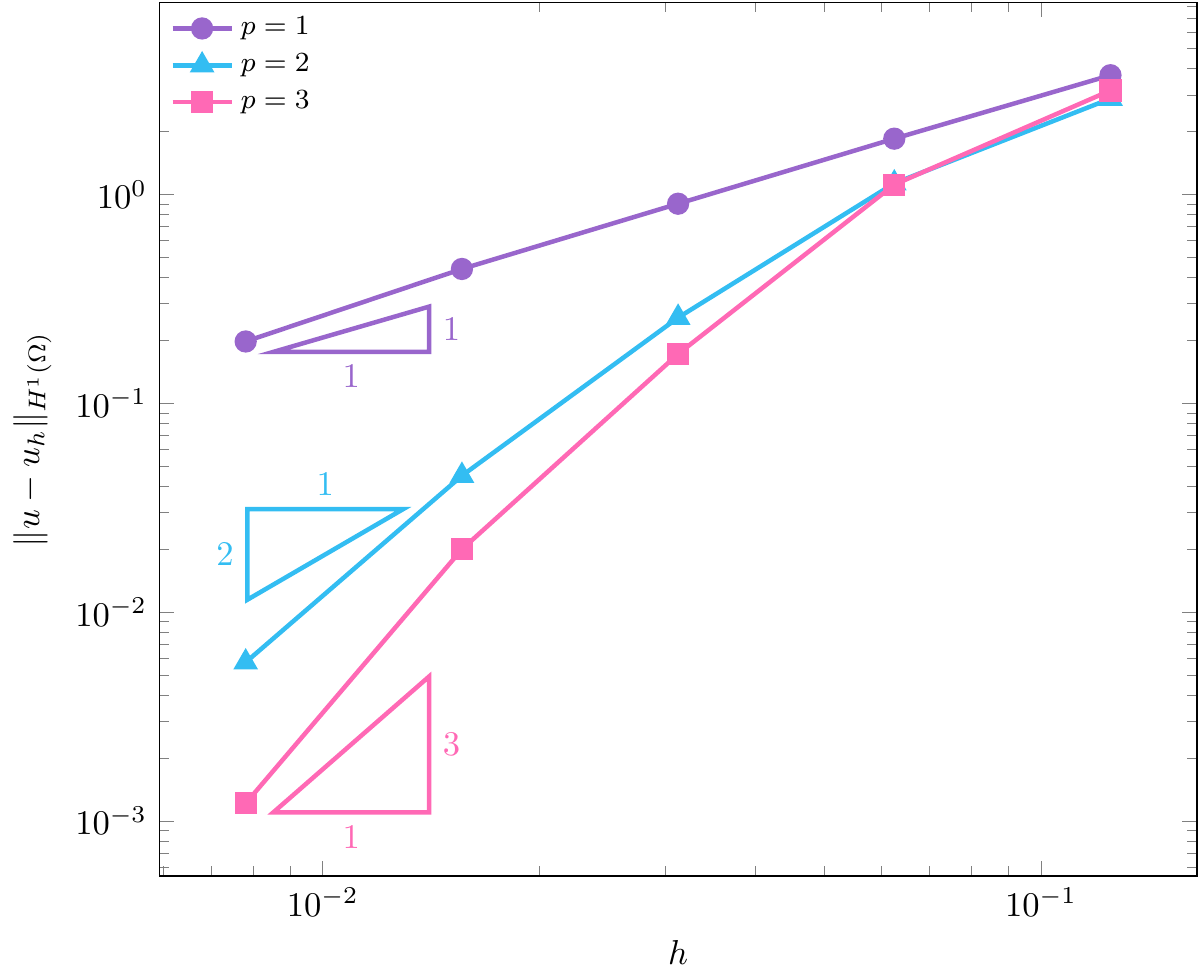}
    \hspace{0.01\textwidth}
    \includegraphics[width=0.47\textwidth,page=2]{figures/convergence_plots_bvp_p1_to_p3_example_1.pdf}
  \end{center}
  \caption{$H^1$ error (left) and $L^2$ error (right) convergence rates for the first, two-dimensional test case using
    different approximation orders $p =1,2,3$.}
  \label{fig:bvp:convergence_plots}
\end{figure}
To demonstrate the applicability of our code to complex
three-dimensional problems, we consider a second test case, where the
model problem~(\ref{eq:laplace-strong}) is solved over a flower
shaped, three-dimensional domain
$\Omega = \{ (x,y,z) \in \RR^3 \st \phi(x,y,z) < 0 \} \subset \RR^3$
defined by 
\begin{equation*}
\phi(x,y,z)=\sqrt{x^2+y^2+z^2} -r+(r/r_0)\cos({5\atantwo(y,x)})\cos({\pi z}),
\end{equation*}
with $r=0.5$ and $r_0=3.5$. This time, we choose an analytic reference solution of the form
\begin{equation}
u(x,y,z)=\exp(x+y+z) \cos(x+y+z) \sin(x+y+z).
\label{eq:bv-analy}
\end{equation}
After embedding $\Omega$ into the domain $\Omega_0=[-0.8,0.8]^3$, a
series of meshes $\{\mcT_k\}_{k=0}^{4}$ is generated with mesh size
$h = 1.6/N$, $N = 6\cdot 2^k$ and the numerical solution is computed
using $V_h = \PP_1(\mcT_k)$ and stabilization parameters $\beta=\gamma_0 = 50.0$, $\gamma_1=0.1$.
The resulting EOC displayed in
Table~\ref{tab:convanalysis2} corroborates the theoretical results
from Section~\ref{ssec::apriori}.
Plots of the computed solutions to both the two- and three-dimensional test problems
can be found in Figure~\ref{fig:boundary3D}.
\begin{table}[htb]
  \small
  \begin{center}
  \begin {tabular}{cr<{\pgfplotstableresetcolortbloverhangright }@{}l<{\pgfplotstableresetcolortbloverhangleft }cr<{\pgfplotstableresetcolortbloverhangright }@{}l<{\pgfplotstableresetcolortbloverhangleft }c}%
\toprule $N_k$&\multicolumn {2}{c}{$\|e_k^1 \|_{H_1(\Omega )}$}&EOC&\multicolumn {2}{c}{$\| e_k^1 \|_{L^2(\Omega )}$}&EOC\\\midrule %
\pgfutilensuremath {6}&$4.53$&$\cdot 10^{-1}$&--&$6.33$&$\cdot 10^{-2}$&--\\%
\pgfutilensuremath {12}&$2.46$&$\cdot 10^{-1}$&\pgfutilensuremath {0.88}&$1.88$&$\cdot 10^{-2}$&\pgfutilensuremath {1.75}\\%
\pgfutilensuremath {24}&$1.26$&$\cdot 10^{-1}$&\pgfutilensuremath {0.97}&$5.04$&$\cdot 10^{-3}$&\pgfutilensuremath {1.90}\\%
\pgfutilensuremath {48}&$6.35$&$\cdot 10^{-2}$&\pgfutilensuremath {0.99}&$1.28$&$\cdot 10^{-3}$&\pgfutilensuremath {1.98}\\%
\pgfutilensuremath {96}&$3.18$&$\cdot 10^{-2}$&\pgfutilensuremath {1.00}&$3.14$&$\cdot 10^{-4}$&\pgfutilensuremath {2.03}\\\bottomrule %
\end {tabular}%

  \caption{Convergence rates for three-dimensional test case using $\PP_1(\mcT_h)$.
  }
  \label{tab:convanalysis2}
  \end{center}
\end{table}

\begin{figure}[htb]
  \begin{subfigure}[t]{0.46\textwidth}
  \vspace{0em}
  \begin{center}
        \includegraphics[width=1.0\textwidth]{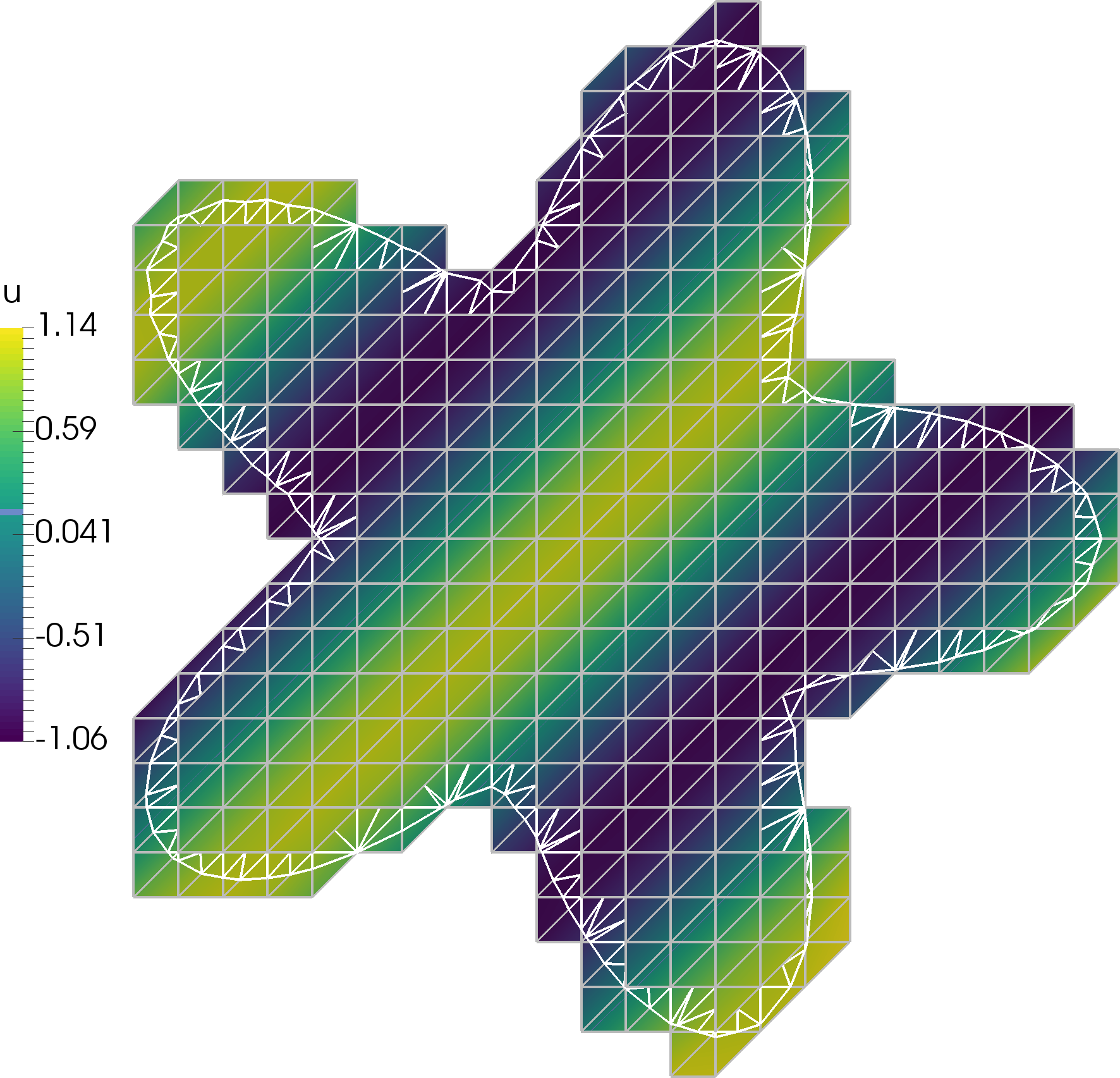}
  \end{center}
\end{subfigure}
\hspace{0.02\textwidth}
\begin{subfigure}[t]{0.52\textwidth}
\vspace{0em}
  \begin{minipage}[c]{0.15\linewidth}
  \vspace{2.5em}
    \includegraphics[width=1.0\textwidth]{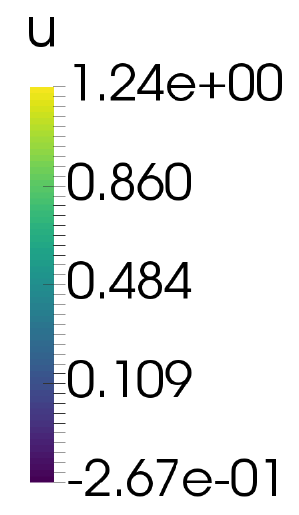}
  \end{minipage}
  \begin{minipage}[c]{0.83\linewidth}
  \vspace{0em}
\hspace{-0.01\textwidth}
    \includegraphics[width=1.0\textwidth]{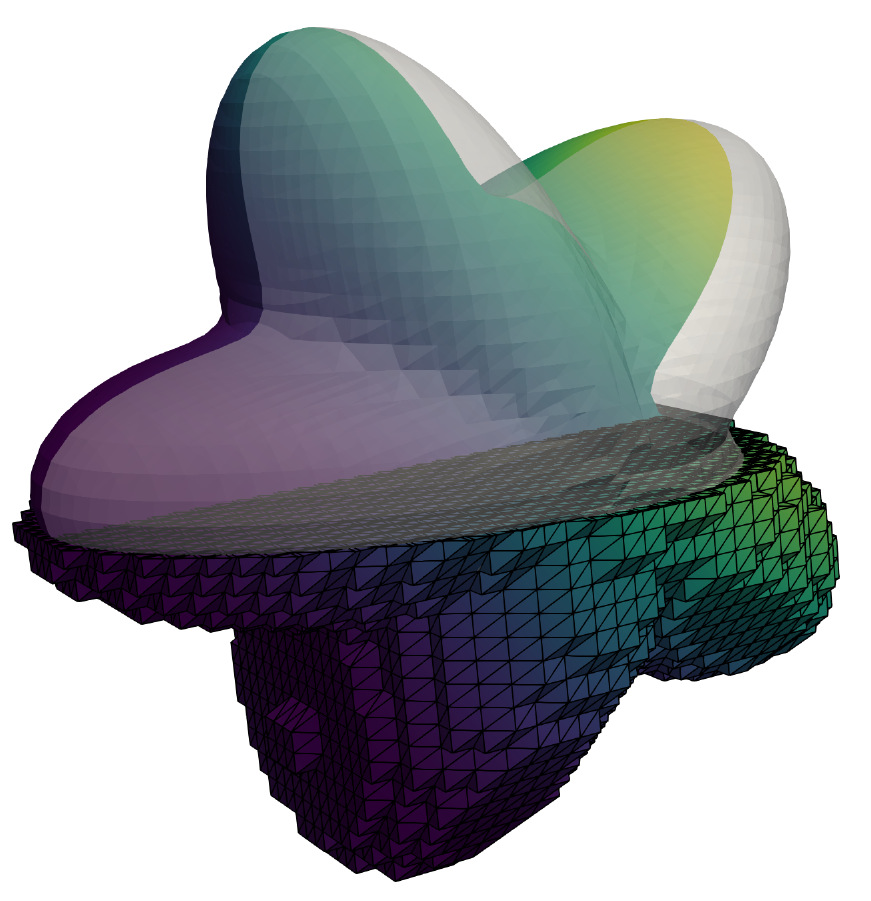}
  \end{minipage}
\end{subfigure}
\caption{Solutions plots for the two-dimensional (left) and the
  three-dimensional (right) convergence study. The solutions are
  plotted over the active background mesh, together with actual the
  physical domain embedded into it. For the two-dimensional problem,
  the $\PP_2(\mcT_h)$ based solution is shown.}
  \label{fig:boundary3D}
\end{figure}

\subsubsection{A numerical look at the $H^1$ extension property}
Next, we illustrate numerically the role of
the $H^1$~extension property defined in
Assumption~\ref{ass:ghost-penalty-coerc}.  In a first experiment, we
repeat the convergence study for the two-dimensional test problem from
Section~\ref{ssec:bvp-conv-analysis} using discontinuous $P_1$
elements and deactivate the ghost penalty by setting
$\gamma_0 = \gamma_1 = 0$.  To trigger critical cut configurations more
easily, we compute the corresponding numerical solutions on a series
of \emph{non-hierarchical} meshes $\{\mcT_k\}_{k=1}^{20}$ for the
domain $[-0.8,0.8]^2$. More precisely, for $k=1,2,\ldots,20$, a
structured mesh $\mcT_k$ was generated by subdividing each axis into
$N = k\cdot 5$ subintervals and subsequently dividing each
quadrilateral into two similar triangles.
Figure~\ref{fig:bvp:convergence_plots_h1_ext_prop} (left) displays the
computed $H^1$ discretization error over the mesh size in a double
logarithmic plot. The erratic convergence curve clearly illustrates
that without properly defined ghost penalties, the particular cut
configuration has a severe impact on the $H^1$ discretization error.
On the contrary, the corresponding convergence curve for stabilized
method with $\gamma_0 = 50$ and $\gamma_1 = 0.1$ has the theoretically
expected slope.  We repeat this experiment with alternative
ghost penalty
$ \widetilde{g}_h^1(v,w) = \gamma_1 h( \jump{\nabla v}, \jump{\nabla
  w})_{\mcF_h^g}$ satisfying only
Assumption~\ref{ass:ghost-penalty-coerc},
\ref{ass:ghost-penalty-consist}, and~\ref{ass:inverse-estimate-jh}
(see Remark~\ref{bvp:rem:alternative_penalty}) and observe that the
resulting convergence plot practically coincides with the one for the
standard stabilized cutDGM. In all three cases, the $L^2$ convergence
rate curve was practically unaffected by the particular cut
configuration and thus was omitted from the convergence rate plots.
\begin{figure}[htb]
  \begin{center}
    \begin{subfigure}[t]{0.472\textwidth}
      \includegraphics[width=1.0\textwidth,page=1]{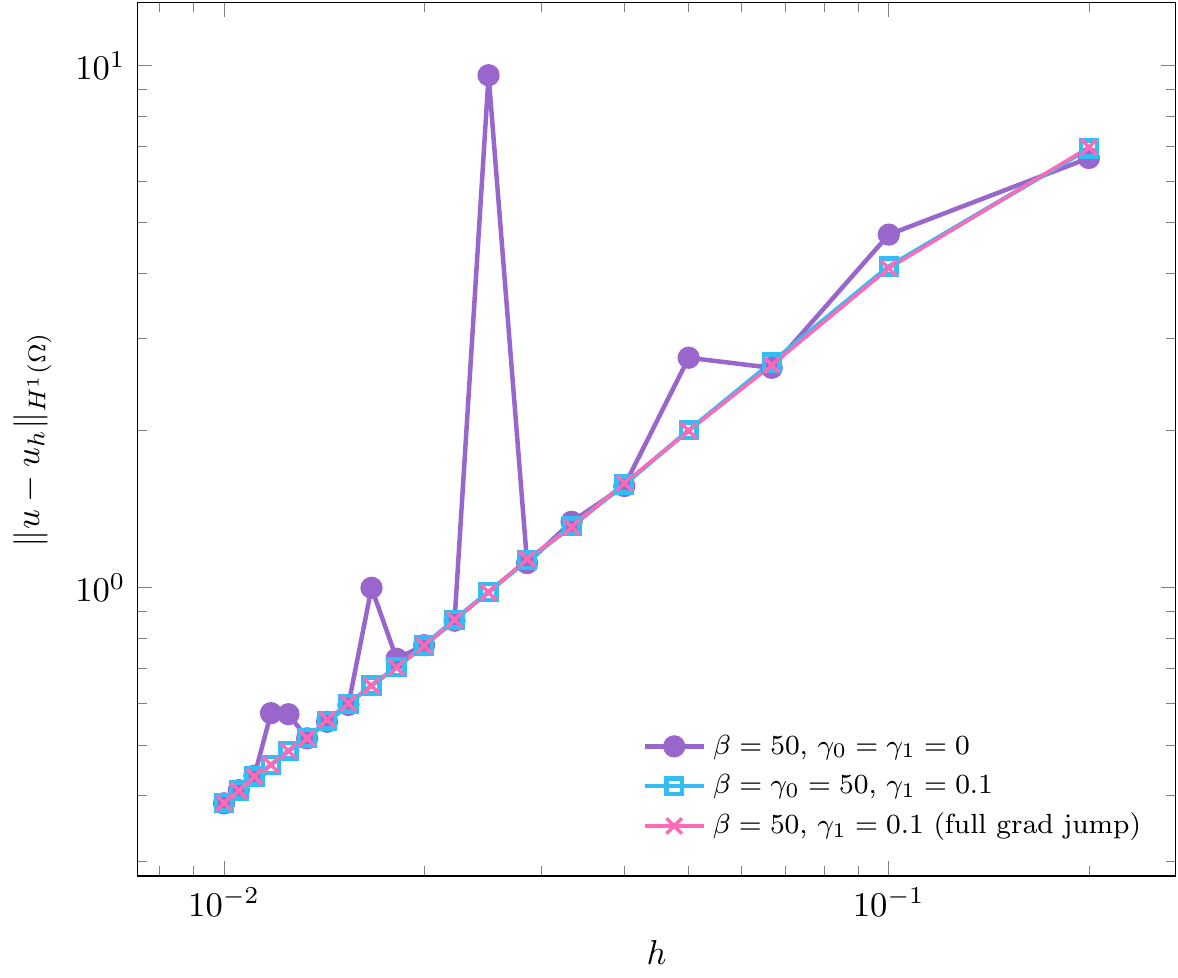}
    \end{subfigure}
    \hspace{0.02\textwidth}
    \begin{subfigure}[t]{0.49\textwidth}
    \includegraphics[width=1.0\textwidth,page=1]{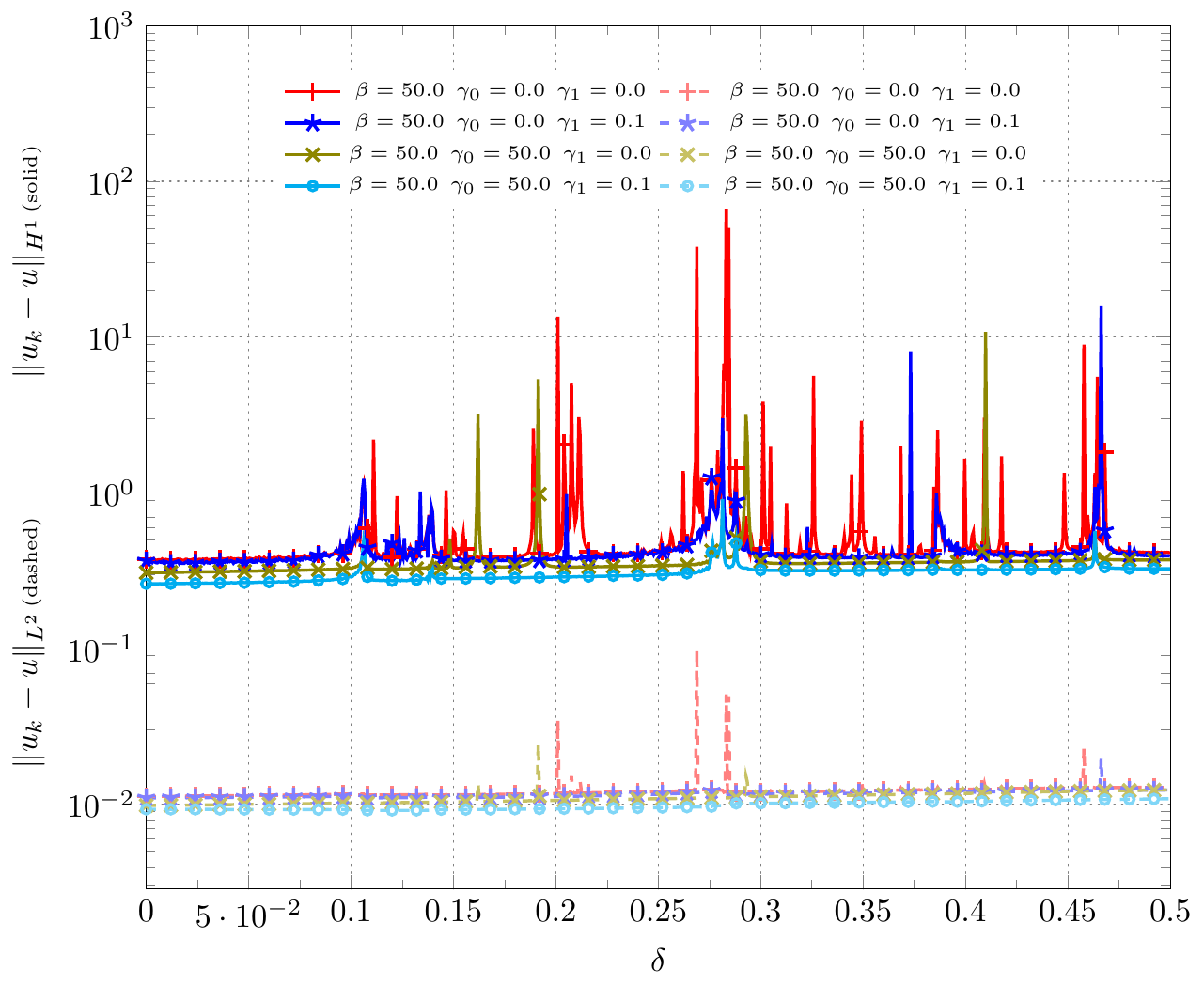}
    \end{subfigure}
  \end{center}
  \caption{$H^1$ error cut configuration dependency studies.  (Left)
    Convergence study for the two-dimensional test case on a series of
    non-hierarchical meshes using $\PP_1(\mcT_h)$ without ghost penalty,
    with ghost penalty $g_h^1$, and the alternative full grad jump ghost penalty
    $\widetilde{g}_h^1$.
    (Right) $H^1$ and $L^2$ error study on a single mesh computed
    for a family of gradually translated domains $\Omega_{1,\delta}$.
  }
  \label{fig:bvp:convergence_plots_h1_ext_prop}
\end{figure}

In a second experiment, we take a closer look at the influence of the
cut configuration on the discretization error for a single fixed mesh.
Again, we start from the two-dimensional test problem from
Section~\ref{ssec:bvp-conv-analysis} using discontinuous $P_1$
and define a $\mcT_h$. for $\Omega_0 = [-0.8, 0.8]^2$ with mesh size $h = 1.6/8$.
To create a large sample of possible cut configurations,
we then generate a family of gradually translated domains
$\{\Omega_{\delta_k}\}_{k=1}^{5000}$ where
$\Omega_{\delta_k} = \Omega + \delta_k (h,h) $ with $\delta_k = k\cdot 2e^{-4}$
and the direction vector $(h,h)$.
For each translated domain, we compute the $H^1$ and $L^2$ discretization
error plot them as functions of $\delta_k$ in a semi-$\log$ plot,
see Figure~\ref{fig:bvp:convergence_plots_h1_ext_prop} (right).
Again, the $H^1$ error for the fully activated ghost penalty $g_h^1$ with $\gamma_0 =  50$ and $\gamma_1 = 0.1$
is nearly completely unaffected by the cut configuration.
When deactivating any of its contribution
by either setting $\gamma_0$ or $\gamma_1$ (or both) to $0$, the $H^1$ error
becomes clearly much more dependent on the particular cut configuration.
In contrast, the $L^2$ error is nearly  unaffected and shows only a few, less drastic
spikes when the ghost penalty stabilization is turned off.

\subsubsection{Condition number studies}
\label{sssec:Cond_number_studies_bvp}
We conclude the presentation of the numerical results
with a series of experiments which illustrate the stabilizing effect
of the ghost penalty on the
condition number of the system matrix associated with the cutDGM~(\ref{eq-cutdg-formulation}).
Following the experimental setup in Section~\ref{ssec:bvp-conv-analysis},
we now pick $\Omega_1 = \{ (x,y,z) \in \RR^3 \st x^2 + y^2 + z^2 - 0.25^2\}$ which is embedded into the domain
$[-0.51, 0.51]^3$. The correspond family of translated domains $\{\Omega_{\delta_k}\}_{k=1}^{500}$
is defined through setting $\delta_k = k \cdot 0.002$. We also refer to
Figure~\ref{fig:condition-number-test-BV} for a principal sketch of the experimental setup.
\begin{figure}[htb]
  \begin{center}
    \begin{minipage}[c]{0.45\linewidth}
      \vspace{0pt}
      \centering
      \includegraphics[width=0.7\textwidth]{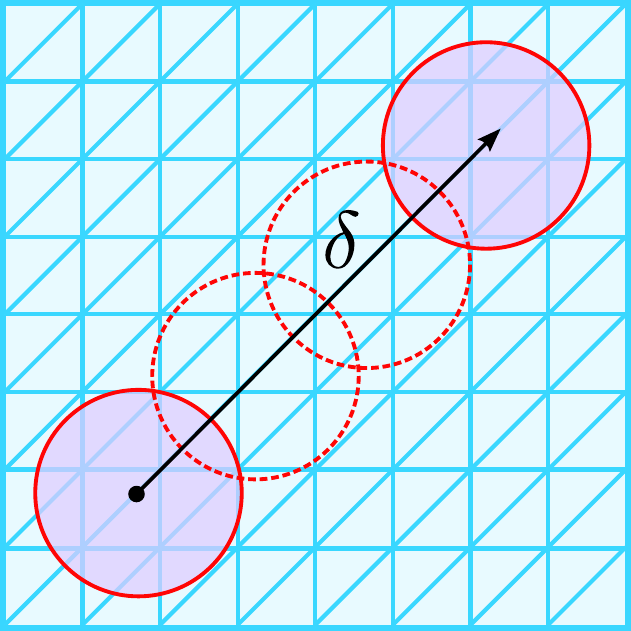}
    \end{minipage}
    \hspace{0.01\linewidth}
    \begin{minipage}[c]{0.5\linewidth}
      \vspace{0pt}
      \centering
      \includegraphics[width=0.7\textwidth]{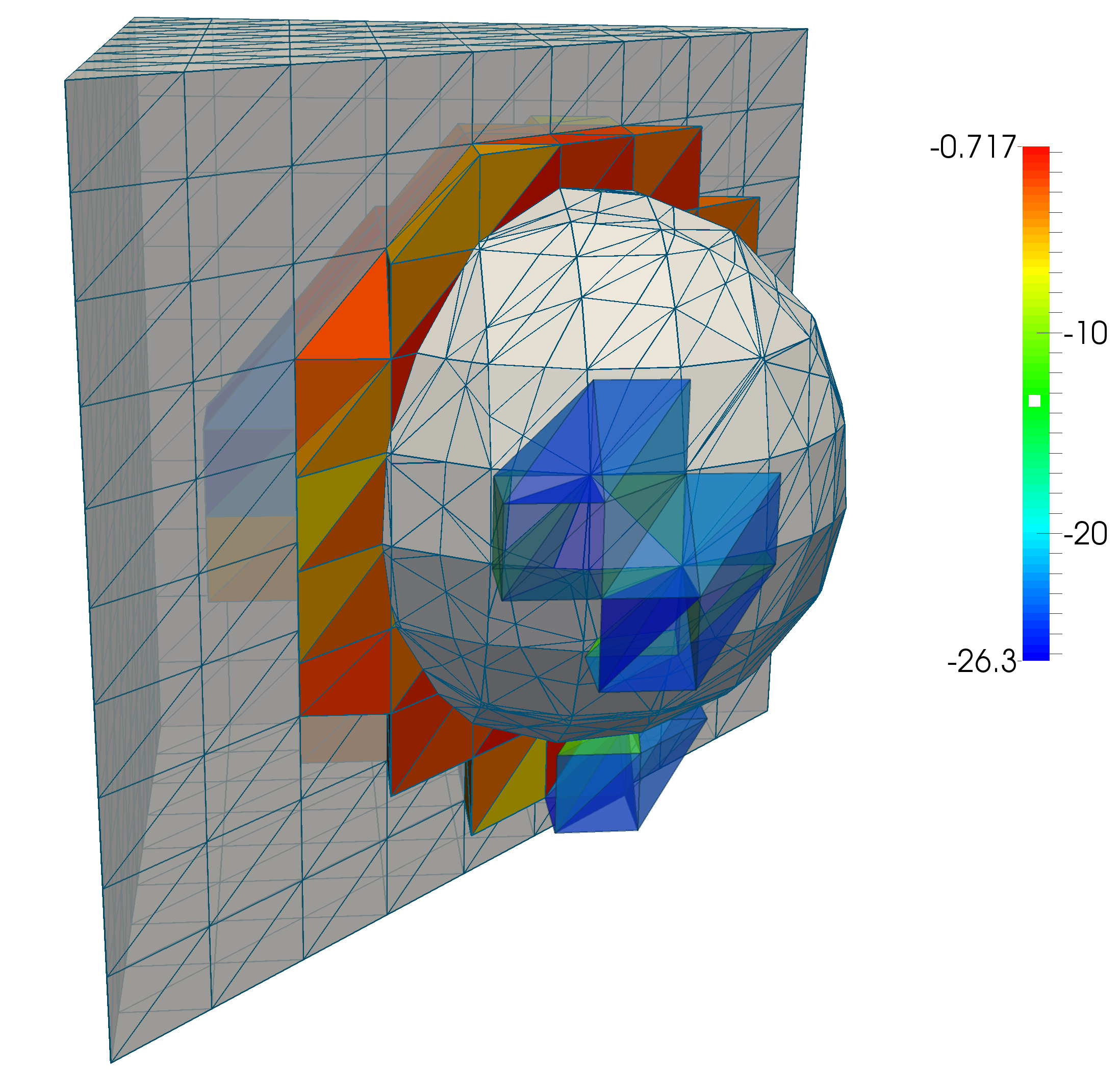}
    \end{minipage}
  \end{center}
  \caption{(Left) Principal experimental set-up to study the sensitivity  of the conditon number
    with respect to the relative $\Gamma$ position.
    (Right): Snapshot of an intersection configuration when moving $\Gamma$
    through the background mesh.
    To visualize ``extreme'' cut configurations,
    the color map plots
    for each intersected mesh element $T$ the value of $\log(\Gamma \cap T/\diam(T)^2)$.
    Thus blue-colored elements contain only an extremely small fraction of the surface.
  }
  \label{fig:condition-number-test-BV}
\end{figure}
For each cut configuration,
the condition number of system matrix associated with formulation~(\ref{eq-cutdg-formulation})
is computed and plotted against $\delta_k$ in a semi-$\log$ plot.
We consider the fully activated ghost penalty $g_h^1$ with $\gamma_0 = 50$ and $\gamma_1 = 0.1$,
partially deactivated versions of it and finally, the alternative ghost penalty $\widetilde{g}_h^1$
which does not satisfy the $L^2$ extension property~\ref{ass:assumption-jh-l2norm-est}.
As predicted by our theoretical analysis, the matrix condition number of all variants
except for the fully activated ghost penalty~$g_h^1$
is highly sensitive to the translation parameter $\delta$, see
Figure~\ref{fig:diff_sphere_stab}~(left).
\begin{figure}[htb]
\begin{center}
    \includegraphics[width=0.48\textwidth]{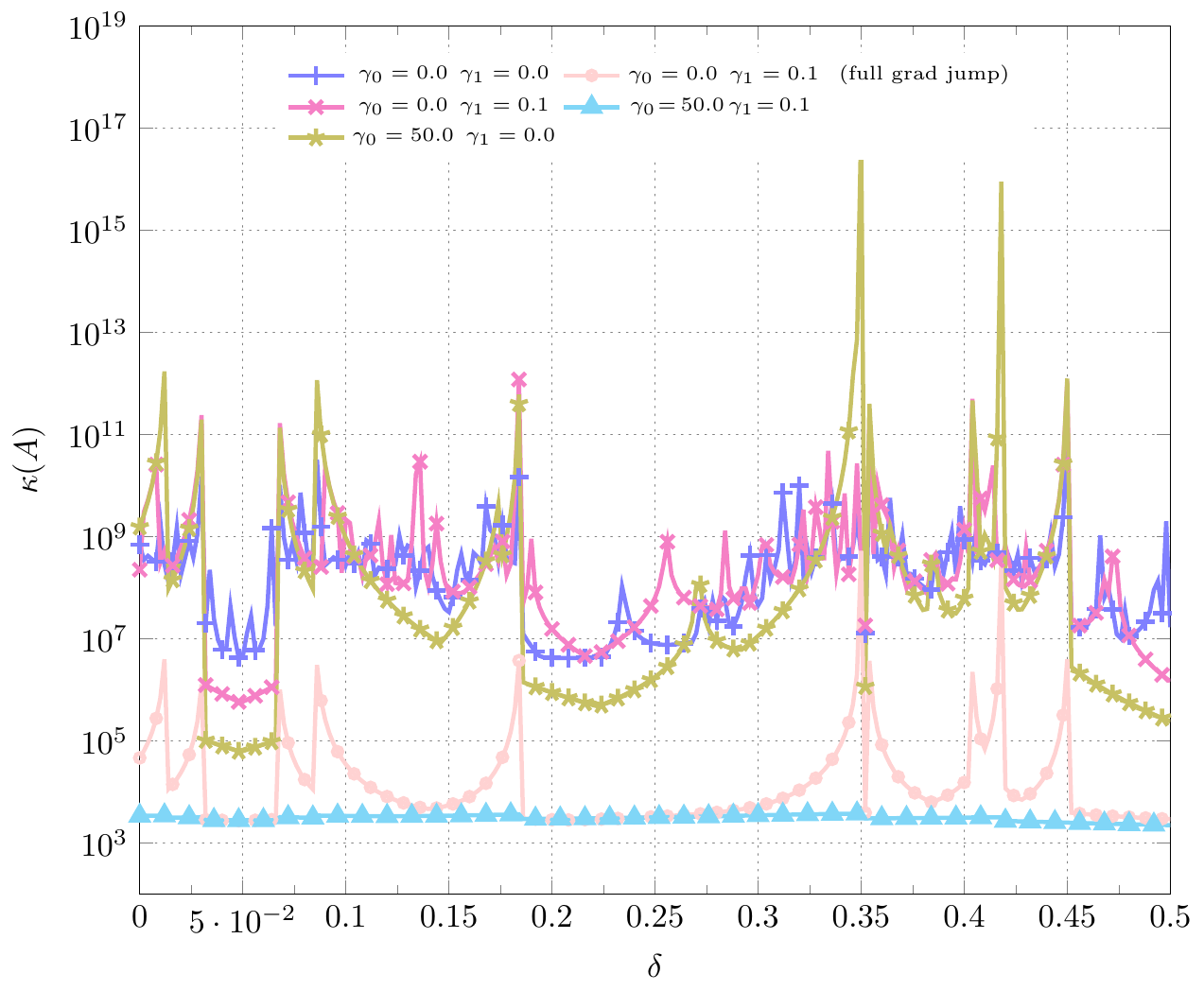} 
    \hspace{0.01\textwidth}
    \includegraphics[width=0.48\textwidth]{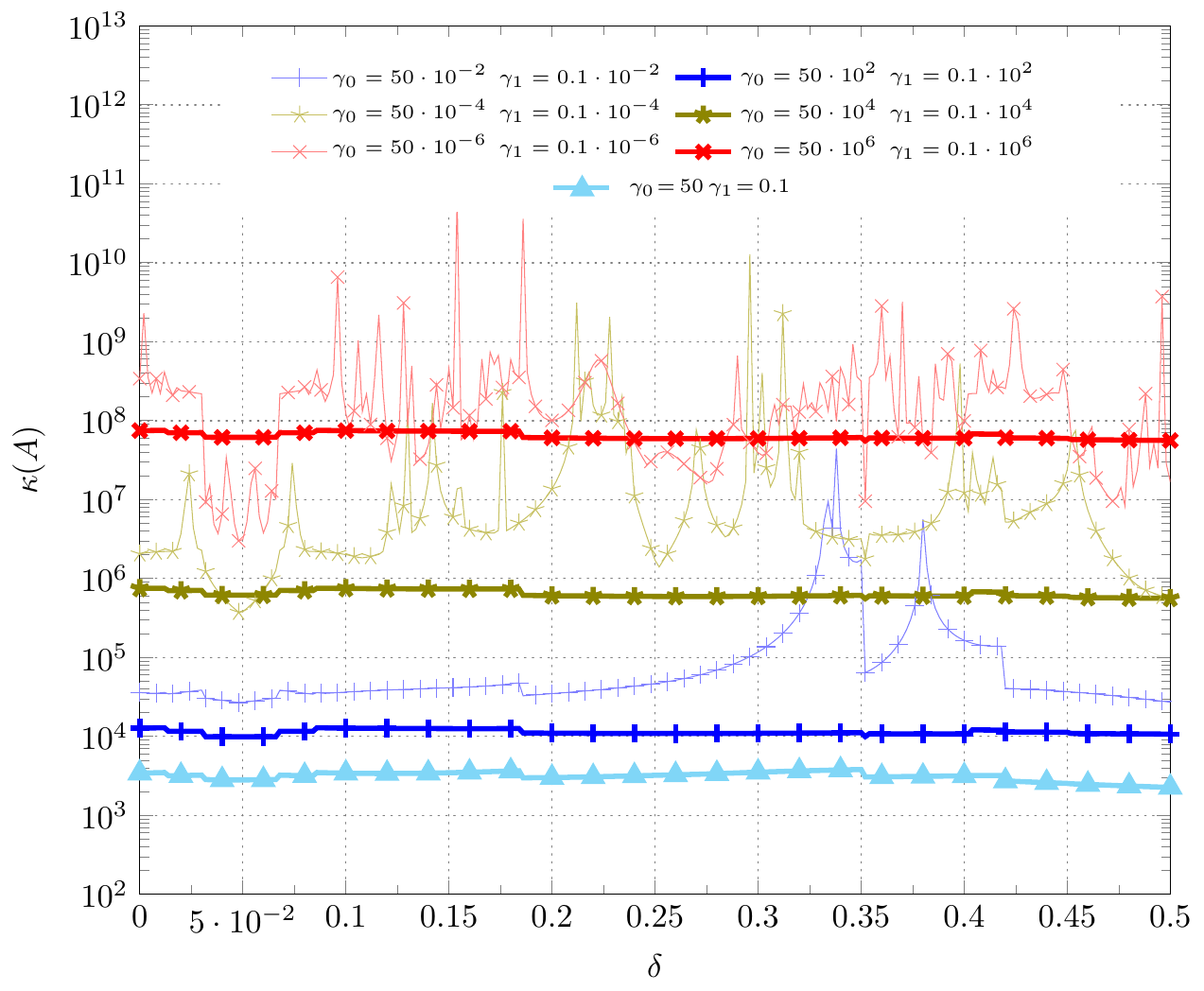}    
    \caption{(Left) Condition number analysis for shifting domain with
      and without ghost penalty stabilization. (Right) Condition
      number analysis for changing values of stabilization parameters
      $\gamma_0$ and $\gamma_1$.}
  \label{fig:diff_sphere_stab}
  \end{center}
\end{figure} 
In a final numerical experiment, we assess the effect of the stability
parameter magnitude on the magnitude and geometric robustness of the
condition number. Starting from our standard parameter choice
$\gamma_0 = 50$ and $\gamma_1 = 0.1$, we rescale the initial pair of
parameters using rescaling factors ranging from $10^{-6}$ to $10^6$
and repeat the previous experiments. The resulting plot in
Figure~\ref{fig:diff_sphere_stab} (right) shows that both the base
line magnitude and the fluctuation of the condition number decrease
with increasing size of the stability parameters with a minimum around
our parameter choice. A further increase of the stability parameters
leaves the condition number insensitive to $\delta$, but leads to an
increase of the overall magnitude. Combined with a series of
convergence experiments (not presented here) for various parameter
choices and combinations , we found that our standard parameter choice
offers a good balance between the accuracy of the numerical scheme and
the magnitude and fluctuation of the condition number.

\section{Interface problems}
\label{sec:interface-problems}
In the final section, we demonstrate how the theoretical framework
developed in
Section~\ref{ssec:cutDGM-bvp}--\ref{ssec:ghost-penalty-realizations}, can
be applied to formulate and analyze a cutDG method for interface
problems.  As before, we depart from a model problem and
define suitable computational domains and approximation spaces leading
us to a symmetric interior penalty based cutDG formulation.  We present a
short but complete theoretical analysis establishing geometrically
robust error estimates in the energy norm, which are complemented with
a number of convergence rate studies.  For a detailed presentation and
analysis of the corresponding cut finite element method for Poisson
interface problems using \emph{continuous} piecewise polynomials, we
refer to~\cite[Section 3]{BurmanZunino2012}.
\subsection{Model problem}
\label{ssec:ifp:model-problem}
Let $\Omega$ be a bounded, closed domain which is divided into two non-overlapping
subdomains $\Omega_1$ and $\Omega_2$ by an interface $\Gamma = \partial \Omega_1 \cap \partial \Omega_2$.
Consider the Poisson interface problem
\begin{subequations}
  \label{eq:interface-problem}
  \begin{alignat}{3}
    -\nabla \cdot (\kappa \nabla u) &= f &&\quad \text{in } \Omega_1
    \cup \Omega_2,
    \label{eq:interface-pde}
    \\
    u &= g &&\quad \text{on } \partial \Omega,
    \label{eq:interface-dirichlet}
    \\
    \jump{u} &= g_D &&\quad \text{on } \Gamma,
    \label{eq:interface-jump}
    \\
    \label{eq:interface-flux-jump}
    \jump{\kappa\partial_n u} &= g_N &&\quad \text{on } \Gamma,
  \end{alignat}
\end{subequations}
where the restricted diffusion coefficient
$\kappa_i  =  \kappa|_{\Omega_i}$ is supposed to be constant for $i=1,2$.
Assuming the interface normal $n$ is pointing outward with respect to $\Omega_1$,
the solution and normal flux jumps across $\Gamma$ are defined as usual by respectively 
\begin{align}
  \jump{u} = u_1|_{\Gamma} - u_2|_{\Gamma},
  \quad \text{and } \quad
\jump{\kappa \partial_n u} =  \kappa_1 \nabla u_1 \cdot n - \kappa_2 \nabla u_2 \cdot n.
\end{align}

\subsection{A cut discontinuous Galerkin method for the Poisson interface
  problem}
\label{ssec:cutDGM-if}
As for the Poisson boundary problem, we assume that
$\Omega$ is covered by a quasi-uniform background mesh $\widetilde{\mcT}_h$
with mesh size $h$. For each subdomain $\Omega_i$, $i=1,2$, 
the \emph{active} background mesh $\mcT_{h,i}$ is again given by
\begin{align}
  \mcT_{h,i} &= \{ T \in {\mcT}_{h} : T \cap \Omega_i^{\circ} \neq \emptyset \}.
\end{align}
Finally, we denote the set of (interior and exterior) faces belonging to $\mcT_{h,i}$
by $\mcF_{h,i}$. Figure~\ref{fig:ip-comp-domains} illustrates the various computational domains
and related mesh entities.
\begin{figure}[htb]
  \centering
  \begin{minipage}[b]{0.33\textwidth}
    \vspace{0pt}
    \includegraphics[width=1.0\textwidth]{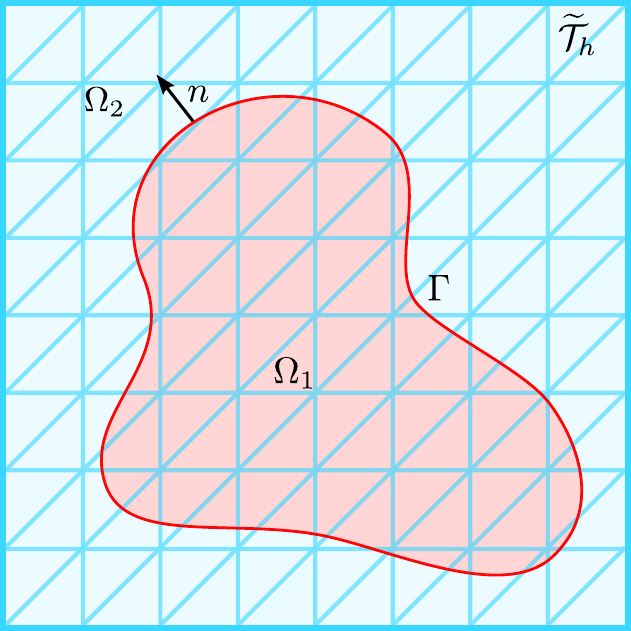}
  \end{minipage}
  \vspace{0.02\textwidth}
  \begin{minipage}[b]{0.285\textwidth}
    \includegraphics[width=1.0\textwidth]{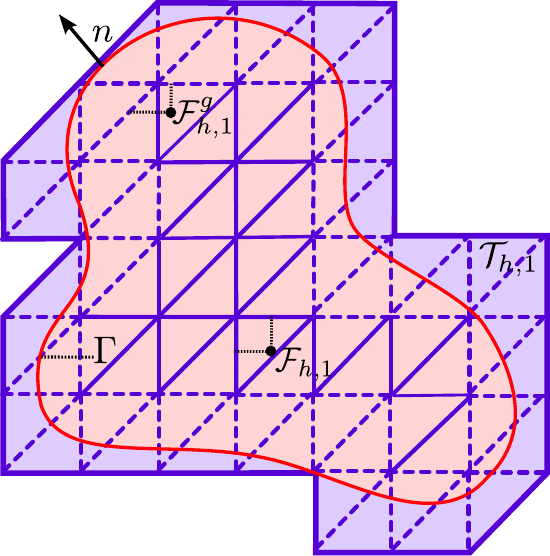}
  \end{minipage}
  \vspace{0.02\textwidth}
  \begin{minipage}[b]{0.33\textwidth}
    \vspace{0pt}
    \includegraphics[width=1.0\textwidth]{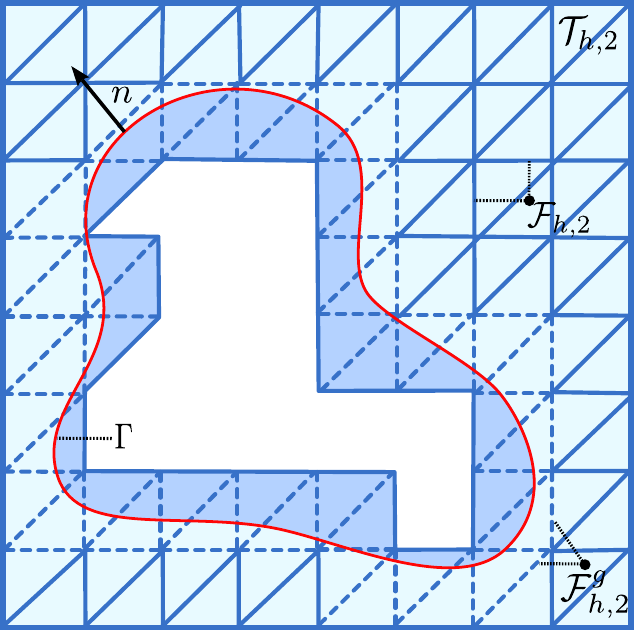}
  \end{minipage}
  \vspace{-2ex}
  \caption{Computational domains for the interface problem. (Left)
    Original background mesh covering $\Omega = \Omega_1 \cup \Omega_2$.
    (Middle)
    The active background mesh $\mcTone$, internal faces $\mcF_{h,1}$ and the
    ghost penalty facets $\mcF_{h,1}^{g}$ associated with $\Omega_1$.
    (Right) Corresponding mesh entities associated with $\Omega_2$.
  }
  \label{fig:ip-comp-domains}
\end{figure}
On the active mesh $\mcT_{h,i}$, we introduce the broken polynomial space of order $k$,
\begin{align}
  V_{h,i} &= \PP_k(\mcT_{h,i}),
            \intertext{and define the resulting total approximation space by}
  V_h &= V_{h,1} \times V_{h,2}.
  \label{eq:Vsh-def-ip}
\end{align}
For $v = (v_1, v_2) \in V_h$, 
the weighted average of the normal flux along $\Gamma$
is given by 
\begin{align}
  \wavg{\kappa \partial_n v}
  &=
    \omega_1 \kappa_1 \partial_n v_1 + \kappa_2 \partial_n v_2,
  \label{eq:weigthed-avg-def}
\end{align}
where the weights satisfy $0\leqslant \omega_1,\omega_2 \leqslant 1$ and  $\omega_1 + \omega_2 = 1$.
In the following, we will also make use of the dual weighted average,
\begin{align}
  \wavgd{v}
  &=
    \omega_2 v_1 + \omega_1 v_2.
    \label{eq:dual-weigthed-avg-def}
\end{align}
Moreover, unifying the notation for interior and exterior faces, we also set
the jump and average operator
on any \emph{exterior} face $F$ belonging $\mcT_{h,i}$ to
\begin{align}
  \avg{v}|_{F} =  \jump{v}|_{F} = v.
\end{align}

To  write our cutDG formulation
in a compact way and to
facilitate the forthcoming numerical analysis, 
we introduce for $i=1,2$ the uncoupled, discrete bilinear forms
\begin{align}
  a_{h,i}(v_i,w_i) &= (\kappa_i \nabla v_i, \nabla w_i)_{\mcT_{h,i}
                     \cap \Omega_i}
                     - (\mean{\kappa_i\partial_n v}, \jump{w})_{\mcF_{h,i} \cap \Omega_i}
  \\
                   &\quad
                     - (\jump{v}, \mean{\kappa_i\partial_n w})_{\mcF_{h,i} \cap \Omega_i}
                     + \beta \kappa_i (h^{-1} \jump{v},\jump{w})_{\mcF_{h,i} \cap \Omega_i},
\end{align}
which involve only geometric quantities defined in the physical domain $\Omega_i$.
The corresponding combined form is then
\begin{align}
 a_{h,\Omega}(v,w) &= \sum_{i=1}^2 a_{h,i} (v_i, w_i),
\end{align}
and the coupling between the domains is encoded
in the interface related bilinear form
\begin{align}
  a_{h,\Gamma}(v,w) &=
                      - (\wavg{\kappa_i\partial_n v}, \jump{w})_{\Gamma}
                      - (\jump{v}, \wavg{\kappa_i\partial_n w})_{\Gamma}
                      + \beta_{\Gamma}(\kappa_1, \kappa_2) (h^{-1}\jump{v},\jump{w})_{\Gamma}.
\end{align}
To account for high contrast interface problems where the ratio
$ \tfrac{\kappa_1}{\kappa_2}$ can become arbitrary large or small, we
employ harmonic weights
following~\cite{Dryja2003,ErnStephansenZunino2008,BurmanZunino2012}
and set the weights $\omega_i$ and the stability parameter~$\beta_{\Gamma}$ to
\begin{gather}
  \omega_1 = \frac{\kappa_2}{\kappa_1+\kappa_2},
  \quad
  \omega_2 = \frac{\kappa_1}{\kappa_1+\kappa_2},
  \quad 
   \beta_{\Gamma}(\kappa_1, \kappa_2) = \widetilde{\beta}_{\Gamma} \dfrac{2 \kappa_1 \kappa_2}{\kappa_1 + \kappa_2},
 \end{gather}
 with $\widetilde{\beta}_{\Gamma}$ independent of $\kappa$.
 Alternative weights were proposed in
~\cite{BarrauBeckerDubachEtAl2012a,AnnavarapuHautefeuilleDolbow2012},
 see also Remark~\ref{rem:alternative-weights}.
Next, similar to the cutDG formulation~(\ref{eq-cutdg-formulation}) for the
 Poisson boundary problem, we introduce ghost penalty enhanced
 versions of $a_{h,i}$ by setting
\begin{align}
  A_{h,i}(v_i, w_i) &= a_{h,i}(v_i, w_i) + g_{h,i}(v_i, w_i) \quad i = 1,2,
\end{align}
and require that each $g_{h,i}$ individually satisfies the
Assumption~\ref{ass:ghost-penalty-coerc}--\ref{ass:inverse-estimate-jh}
with respect to $\Omega_i$ and $\mcT_{h,i}$.
Occasionally, we will also use
the short-hand notation
$ g_{h}(v, w) = \sum_{i=1}^2g_{h,i}(v_i, w_i) $.  Now the cutDG
formulation for the interface problem~\eqref{eq:interface-problem} is
to seek $u_h = (u_{h,1}, u_{h,2}) \in V_h$ such that
\begin{align}
  A_h(u_h,v) \coloneqq
  \sum_{i=1}^2 A_{h,i} (u_{h,i}, v_i)  + a_{h,\Gamma}(u_h, v)
  = l_h(v) 
\quad  \foralls v = (v_1, v_2) \in V_h,
 \label{eq:Ah-IP-def} 
\end{align}
where the linear form $l_h$ is given by
\begin{align}
    l_h(v) &=\sum_{i=1}^2 (f_i,v_i)_{\mcT_{h,i} \cap \Omega_i} -
    (g_D, \wavg{\kappa_i\partial_n v})_{\Gamma}
    + \beta_{\Gamma}(h^{-1}
    {g_D} , \jump{v})_{\Gamma}
    \\
    &\quad + ( g_N,\wavgd{v})_{\Gamma}
    - (g, \kappa_2\partial_n v)_{\partial\Omega} + \beta_{F,2} \kappa_2 (h^{-1}
    g,v)_{\partial\Omega}.
\end{align}

\subsection{Stability and convergence properties}
\label{sec:ip-stabilityprop}
As the theoretical investigation of the cutDG formulation of the Poisson interface problem
will heavily rest upon the numerical analysis presented in Section~\ref{sec:bound-value-probl},
we decompose the energy norm for the final formulation~(\ref{eq:Ah-IP-def})
into domain specific and interface specific contributions. To this end, we define the
discrete energy norms for $v_i \in V_{h,i}$, $i=1,2$ by
\begin{align}
  \tn v_i \tn_{a_{h,i}}^2 
&= \| \nabla v_i \|^2_{\mcT_{h,i} \cap \Omega_i}
  + \|h^{-1/2} [v_i] \|^2_{\mcF_{h,i} \cap \Omega_i},
  \\
  \tn  v_i  \tn^2_{A_{h,i}} &= \tn  v_i  \tn^2_{a_{h,i}} + |v|^2_{g_{h,i}},
\end{align}
and the discrete energy norms associated with the total bilinear form $a_h$ and
$A_h$ by
\begin{align}
  \tn  v  \tn^2_{a_{h}} &= \sum_{i=1}^2 \tn  v_i  \tn^2_{a_{h,i}} + \beta_{\Gamma} \|\jump{h^{-1} v}\|_{\Gamma}^2,
                          \\
  \tn  v  \tn^2_{A_{h}} &= \sum_{i=1}^2 \tn  v_i  \tn^2_{A_{h,i}} + \beta_{\Gamma} \|\jump{h^{-1} v}\|_{\Gamma}^2,
\end{align}
respectively. As an immediate result of the numerical analysis
presented in Section~\ref{ssec::stab-prop}, we record the following
corollary, stating that the bulk forms $a_{h,i}$ are coercive with
respect to $\tn \cdot \tn_{a_{h,i}}$.
\begin{corollary}
  For $v = (v_1,v_2) \in V_h$, it holds that
  \begin{align}
    \tn v_i \tn_{A_h,i}^2 \lesssim A_{h,i}(v_i, v_i) \quad i = 1,2.
  \end{align}
\end{corollary}
This observation allows us to give a short proof showing that the
total discrete form $A_h$ defined by~(\ref{eq:Ah-IP-def}) is coercive
in the discrete energy norm $\tn \cdot \tn_{A_h}$.
\begin{proposition}
  \label{prop:Ah-IF-coerciv}
  It holds that
  \begin{alignat}{3}
    \tn  v  \tn_{A_h}^2 &\lesssim A_h(v, v) && \quad \foralls v \in V_h.
    \label{eq:coercivity-Ah-IP}
  \end{alignat}
\end{proposition}
\begin{proof}
  Since 
  \begin{align}
    A_h(v, v) &= \sum_{i=1}^2 A_{h,i}(v_i,v_i) + a_{h,\Gamma}(v,v)
                \gtrsim \sum_{i=1}^2 \tn v_i \tn_{A_{h,i}}^2  + a_{h,\Gamma}(v,v),
                \label{eq:Ah-IP-coerc-step-1}
  \end{align}
  it only remains to treat the interface related contribution in~\eqref{eq:Ah-IP-coerc-step-1}.
  Recalling the definition of the weighted average $\wavg{\cdot}$,
  and successively applying an $\epsilon$-Young inequality and 
  the trace estimate~(\ref{eq:inverse-est-cut-grad-and-trace}), we deduce that
  \begin{align}
      2 (\wavg{\kappa \partial_n v}, \jump{v})_{\Gamma}
    &=
      2 ( \omega_1 \kappa_1 \partial_n v_1 + \omega_2 \kappa_2 \partial_n v_2, \jump{v})_{\Gamma}
      \\
    &=
      \dfrac{2\kappa_1 \kappa_2}{\kappa_1 + \kappa_2}
       ( \partial_n v_1 +\partial_n v_2, \jump{v})_{\Gamma}
    \\
    &\lesssim
      \epsilon \dfrac{\kappa_1 \kappa_2}{\kappa_1 + \kappa_2}
      (  \| h^{\onehalf} \partial_n v_1\|_{\Gamma}^2 
      +  \| h^{\onehalf} \partial_n v_2\|_{\Gamma}^2 )
      + \epsilon^{-1} \dfrac{2\kappa_1 \kappa_2}{\kappa_1 + \kappa_2}
      \| h^{-\onehalf} \jump{v}\|_{\Gamma}
    \\
    &\lesssim
      \epsilon \sum_{i=1}^2
      \kappa_i \| \nabla v_i\|_{\mcT_h}^2
      + \epsilon^{-1} \dfrac{2\kappa_1 \kappa_2}{\kappa_1 + \kappa_2}
      \| h^{-\onehalf} \jump{v}\|_{\Gamma}
    \\
    &\lesssim
      \epsilon \sum_{i=1}^2
      \tn v_i \tn_{A_{h,i}}^2
      + \epsilon^{-1} \dfrac{2\kappa_1 \kappa_2}{\kappa_1 + \kappa_2}
      \| h^{-\onehalf} \jump{v}\|_{\Gamma},
  \end{align}
  where in the last two steps, we used the fact that
  $ \tfrac{\kappa_1 \kappa_2}{\kappa_1 + \kappa_2} \leqslant
  \min\{\kappa_1, \kappa_2\}$ and the $H^1$ extension property of
  $g_{h,i}$.  Now recalling that
  $\beta_{\Gamma} = \widetilde{\beta}_{\Gamma} \tfrac{2\kappa_1
    \kappa_2}{\kappa_1 + \kappa_2}$, it follows that
  \begin{align}
    a_{h,\Gamma}(v,v)
    &=
      -2 (\wavg{\kappa \partial_n v} \jump{v})_{\Gamma}
      + \beta_{\Gamma} \| h^{-\onehalf}\jump{v}\|_{\Gamma}
    \\
    &\gtrsim
      - \epsilon \sum_{i=1}^2
      \tn v_i \tn_{A_{h,i}}^2
      +  (\widetilde{\beta}_{\Gamma} - \epsilon^{-1}) \dfrac{2\kappa_1 \kappa_2}{\kappa_1 + \kappa_2}
      \| h^{-\onehalf} \jump{v}\|_{\Gamma},
  \end{align}
  which together with~\eqref{eq:Ah-IP-coerc-step-1} gives the desired result for $\epsilon$ 
  small enough and $\widetilde{\beta}$ large enough.
  \qed
\end{proof}
\begin{remark}
  The previous proof reveals that the key role of the ghost penalty
  forms $g_{h,i}$ is to extend the control of the domain-related
  energy norms $\tn \cdot \tn_{a_{h,i}}$ from physical subdomains $\Omega_i$ to the
  respective active meshes $\mcT_{h,i}$.  As a consequence, we can 
  closely follow the standard derivation of coercivity
  results for discontinuous Galerkin methods for high contrast
  interface and heterogeneous diffusion problems presented in~\cite{DiPietroErn2012,BurmanZunino2012}.
\end{remark}
\begin{remark}
  \label{rem:alternative-weights}
  \citet{BarrauBeckerDubachEtAl2012a}
  and~\citet{AnnavarapuHautefeuilleDolbow2012} proposed an unfitted
  Nitsche-based formulation for large contrast interface problems
  without ghost penalties using the weights
  \begin{gather}
    \omega_1 =  \dfrac{\kappa_2 |T\cap \Omega_1|_{d}}{\kappa_2  |T\cap \Omega_1|_{d} + \kappa_1  |T\cap \Omega_2|_{d}},
    \qquad
    \omega_2 =  \dfrac{\kappa_1 |T\cap \Omega_2|_{d}}{\kappa_2  |T\cap \Omega_1|_{d} + \kappa_1  |T\cap \Omega_2|_{d}},
    \\
     \intertext{and instead of $\beta_{\Gamma}h^{-1}$, a stability parameter of the form}
    \beta^{*} =  \widetilde{\beta_{\Gamma}}\dfrac{\kappa_1 \kappa_2 |\Gamma \cap T|_{d-1}}{\kappa_2  |T\cap \Omega_1|_{d} + \kappa_1  |T\cap \Omega_2|_{d}}.
    \label{eq:alt-weights}
  \end{gather}
  was employed.
  By incorporating the area of the physical element parts
  $T\cap \Omega_i$, this choice of weights accounts for both the contrast in the diffusion
  coefficient and the particular cut configuration.
  Consequently, this technique is not completely robust in the most extreme cases,
  where both a large contrast and a bad cut configuration must be handled simultaneously.
  For instance, let us consider again the sliver cut case in Figure~\ref{fig:critical-cut-cases} (left)
  with the normal $n$ pointing outwards with respect to $\Omega_1$.
  For a fixed mesh and mesh size, we see that
  if $\kappa_2 \delta \geqslant \kappa_1 h$ or equivalently,
  $\tfrac{\kappa_2}{\kappa_1} \geqslant \tfrac{h}{\delta}$,
  the stability parameter scales like
  \begin{align}
  \dfrac{\kappa_1 \kappa_2 |\Gamma \cap T|_{d-1}}
    {\kappa_2  |T\cap \Omega_1|_{d} + \kappa_1  |T\cap \Omega_2|_{d}}
    \sim
  \dfrac{\kappa_1 \kappa_2 h^{d-1}}
    {\kappa_2  \delta h^{d-1} + \kappa_1  h^{d}}
    \geqslant
  \dfrac{\kappa_1 \kappa_2 h^{d-1}}
    {2\kappa_2  \delta h^{d-1}}
    =
  \dfrac{\kappa_1}
    {2\delta}
  \end{align}
  and thus it can become arbitrary large in case of large contrast
  and a bad cut configuration.  Using ghost penalty enhanced
  (continuous or discontinuous) unfitted finite element methods on the
  other hand, the stability parameter $\beta_{\Gamma}h^{-1}$ scales
  like $\beta_{\Gamma}h^{-1} \sim \min\{\kappa_1, \kappa_2\} h^{-1}$
  and thus is not affected by a large diffusion parameter
  $\kappa_2 \gg \kappa_1$ or a particular bad cut configuration with
  $\delta \ll h$.
\end{remark}

Thanks to the discrete coercivity estimate~\eqref{eq:coercivity-Ah-IP},
we can now follow the derivation in Section~\ref{ssec::apriori}
to prove optimal and geometrically robust a priori error estimates.
To do so, we first define a suitable interpolation operator similar as
in Section~\ref{ssec::apriori} by setting
\begin{align}
  \pi_h^e u = (\pi_{h,1} u_1^e, \pi_{h,2} u_2^e)    
\end{align}
where $\pi_{h,i} : L^2(\mcT_{h,i}) \to V_{h,i}$ for $i=1,2$.
Then we have the following result. 
\begin{theorem}
  Let $u = (u_1 , u_2) \in H^{s}(\Omega_1) \times H^{s}(\Omega_2)$
  be the solution to the interface problem~(\ref{eq:interface-problem})
  and let $u_h = (u_{h,1}, u_{h,2}) \in \PP_{k}(\mcT_{h,1}) \times \PP_k(\mcT_{h,2})$
  be the solution to the cut discontinuous Galerkin formulation~(\ref{eq:Ah-IP-def}).
  Setting $r = \min\{s, k+1\}$, it holds that
  \begin{align}
    \tn  u - u_h \tn_{a_h} 
    \lesssim
     \sum_{i=1}^{2} \kappa_1^{\onehalf} h^{r-1} \| u_i \|_{r, \Omega_i}.
  \end{align}
\end{theorem}
\begin{proof}
  Following closely the derivation of the a priori estimate~(\ref{eq:apriori-est-energy}),
  we only sketch the proof. As before, by decomposing
  $ u - u_h = (u - \pi_h^e u) + (\pi_h^e u -    u_h) = e_{\pi} + e_h$,
  it is enough to focus on the discrete error $e_h$.
  Combining the discrete coercivity and the weak Galerkin orthogonality
  of $A_h$,
  we see that
  \begin{align}
    \tn  e_h  \tn_{A_h}^2
    &\lesssim
      \sum_{i=1}^2
      \bigl(
      a_{h,i} (e_{\pi, i}, e_{h,i}) + g_{h,i}(\pi^{e}_{h,i} u_i, e_{h,i})
      \bigr)
      + a_{h,\Gamma} (e_{\pi, i}, e_h)
      \\
    &\lesssim
      \sum_{i=1}^2 h^{r-1}\kappa_i^{\onehalf} \| u\|_{r,\Omega} \tn e_{h,i} \tn_{A_{h,i}}
      + a_{h,\Gamma} (e_{\pi}, e_h).
    \label{eq:discrete-error-est-step-last-IF}
  \end{align}
  Unwinding the definition of $a_{h,\Gamma}$ and
  recalling that
  $
  \omega_i \kappa_i \partial_n v_i
  =
  \tfrac{\kappa_1 \kappa_2}{\kappa_1 + \kappa_2} \partial_n v_i
  $
  and that
  $\tfrac{\kappa_1 \kappa_2}{\kappa_1 + \kappa_2} \leqslant \min\{\kappa_1, \kappa_2\}$,
  the remaining interface term can be estimated as in Corollary~\ref{cor:projection-error},
  \begin{align}
    a_{h,\Gamma} (e_{\pi}, e_h)
    &=
      - (\wavg{\kappa_i\partial_n e_{\pi}}, \jump{e_h})_{\Gamma}
      - (\jump{e_{\pi}}, \wavg{\kappa_i\partial_n e_{h}})_{\Gamma}
      + \widetilde{\beta}_{\Gamma} \dfrac{2 \kappa_1 \kappa_2}{\kappa_1 + \kappa_2} (h^{-1}\jump{e_{\pi}},\jump{e_h})_{\Gamma}
      \\
    &\lesssim
      \min\{\kappa_1, \kappa_2\}
      h^{r-1} \| u \|_{r, \Omega_1 \cup \Omega_2}
      \| e_h \|_{A_h}.
  \end{align}
\end{proof}
\subsection{Numerical examples}
\label{ssec:ifp:numerical-examples-ifp}
We conclude this work by briefly presenting a number of convergence
rates studies conducted for two and three-dimensional interface
problems with various contrast ratios.  Throughout the numerical
studies, we employ as ghost penalty the analog of $g_h^1$ defined
in~(\ref{eq:jh-def-face-based}) and set for $i=1,2$,
\begin{align}
  g_{h,i} (v,w)
  &\coloneqq \sum_{j=0}^k \sum_{F\in F_{h,i}^g} \gamma_j h_F^{2j-1} (\jump{\partialbar_n^j v}, \jump{\partialbar_n^j w})_F,
\end{align}
with the ghost penalty  faces 
\begin{align}
  \mcF_{h,i}^{g} &= \{ F = T^+ \cap T^- \in \mcF_{h,i} \st  T^+ \cap \Gamma \neq \emptyset \lor T^- \Gamma \neq \emptyset \}.
  \label{eq:faces-gp-ip}
\end{align}

\subsubsection{Convergence rate studies for 2D interface problems}
Based on geometry for the two-dimensional test case presented in Section~\ref{ssec:bvp-conv-analysis},
we now consider the interface problem~\eqref{eq:interface-problem}
where the domains $\Omega_1$ and $\Omega_2$ are given by
 \begin{align}
   \Omega_1 &= \{ (x,y) \in \RR^2 \st \phi(x,y) < 0 \} \quad \text{with }
              \phi(x,y) = \sqrt{x^2 + y^2} - r_0 - r_1 \cos(\atantwo(y,x)),
   \\
   \Omega_2 &= [-1.1, 1.1]^2 \setminus \Omega_1.
 \end{align}
 As before, we set $r_0 = 0.6$ and $r_1 = 0.2$.
 In our convergence study, we consider two cases.
 First, we define the simplest possible interface problem and set
 $\kappa_1 = \kappa_2 = 1$ and $g_D = g_N = 0$. We construct a manufactured solution
 based on the smooth analytical reference solution
 \begin{align}
   u_i (x,y) =  \cos(2 \pi x)\cos(2\pi y) + \sin(2\pi x)\sin(2\pi y) \quad i =1,2
 \end{align}
 and compute the right-hand side $f$ and the boundary data $g$ accordingly.
 In the second test case, a high contrast interface problem is
 considered, setting $\kappa_1 = 1$ and $\kappa_2 = 10^6$ and employing the analytical reference
 solutions
 \begin{align}
   u_1 (x,y) &=  \sin(\pi (x-y))\cos(\pi (x+y)),
               \\
   u_2 (x,y) &=  \frac{1}{\kappa_1} \sin(0.5 \pi (x+y))\cos(0.5 \pi (x+y)).
 \end{align}
 This time, the interface data $g_D$ and $g_N$ are non-trivial functions
 as the chosen combination of $u_1$ and $u_2$ results in discontinuities
 in both the solution and the normal flux across the interface.
 
 Following the presentation in Section~\ref{ssec:bvp-conv-analysis}, we conduct a convergence
 study for both cases, employing
 $V_{h}^p = \PP_{p}(\mcT_{h,1}) \times \PP_{p}(\mcT_{h,2})$ for
 $p=1,2,3$. The resulting convergence curves
 associated with the $\|\kappa^{\onehalf} \nabla (\cdot)\|_{L^2(\Omega_1 \cup \Omega_2)}$ norm
 are plotted in Figure~\ref{fig:if:convergence_plots} (left) and have the predicted optimal slopes.
 We also plot the convergence curves for the error measured in the
 $\| \cdot \|_{L^2(\Omega_1 \cup \Omega_2)}$ norm in Figure~\ref{fig:if:convergence_plots} (right),
 also showing optimal convergence rates for the chosen examples.
 
\begin{figure}[htb]
  \begin{subfigure}{1.0\textwidth}
  \begin{center}
    \includegraphics[width=0.47\textwidth,page=1]{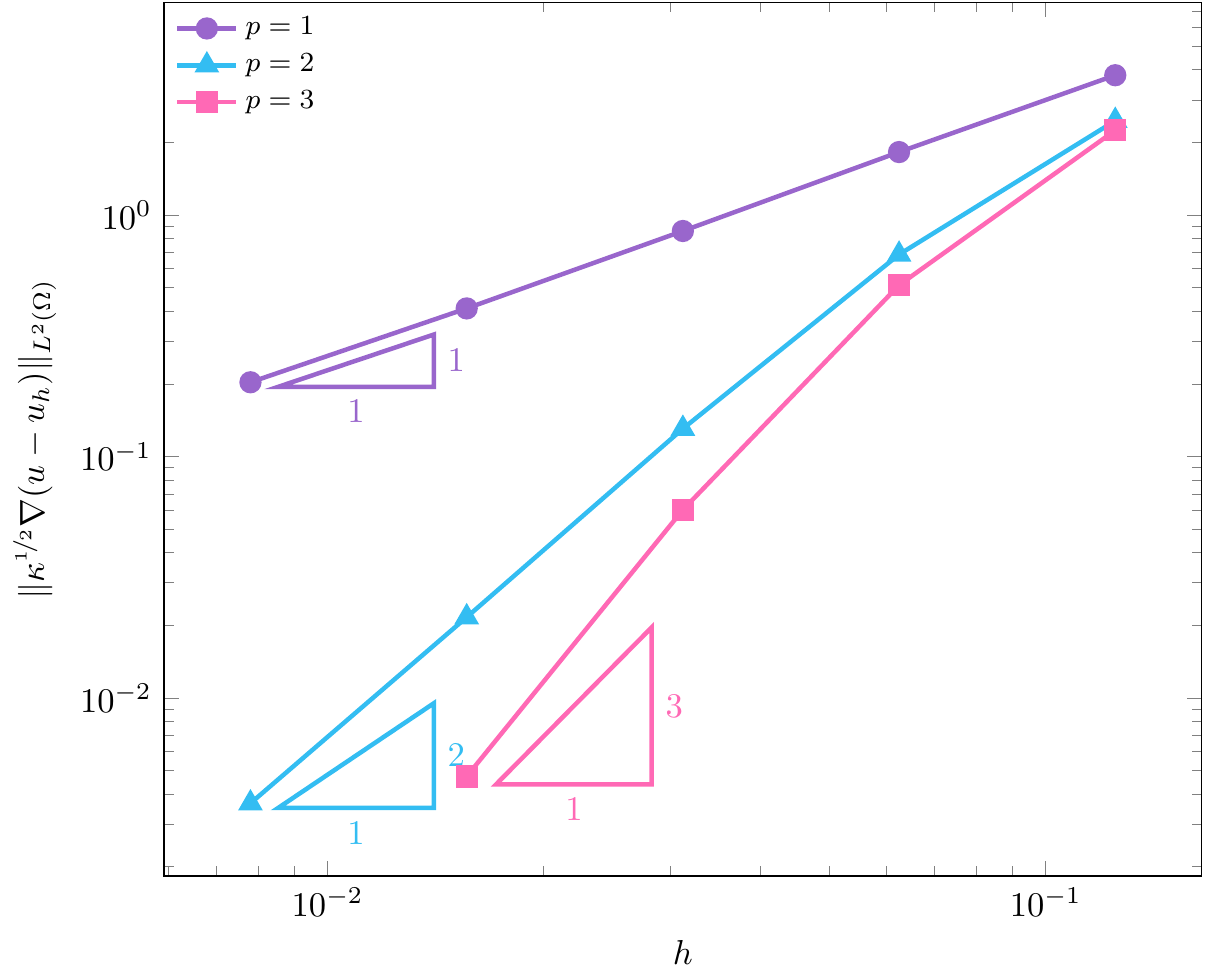}
    \hspace{0.01\textwidth}                                                 
    \includegraphics[width=0.47\textwidth,page=2]{figures/convergence_plots_ip_p1_to_p3_example_3_and_5.pdf}
  \end{center}
\end{subfigure}
  \begin{subfigure}{1.0\textwidth}
\vspace{2em}
  \begin{center}
    \includegraphics[width=0.47\textwidth,page=3]{figures/convergence_plots_ip_p1_to_p3_example_3_and_5.pdf}
    \hspace{0.01\textwidth}                                                 
    \includegraphics[width=0.47\textwidth,page=4]{figures/convergence_plots_ip_p1_to_p3_example_3_and_5.pdf}
  \end{center}
\end{subfigure}
\caption{Convergence rate plots for the first (top) and second (bottom) two-dimensional test cases. Both the
  $\| \kappa^{\onehalf} \nabla (\cdot)\|_{\Omega}$ (left) and $\|\cdot \|_{\Omega}$ (right) error plots
  show optimal convergence rates.}
  \label{fig:if:convergence_plots}
\end{figure}

 \subsubsection{Convergence rate studies for 3D interface problems}
In a final convergence study, we consider two test cases involving a three-dimensional interface problem
posed in the domain $\Omega_0 = [-1.1, 1.1]^3$.
For the first test problem, we define $\Omega_1$ as the union of $8$ balls $B_r(\bfx_k)$
with $r = 0.3$ and the~$8$ center points $\bfx_k = (x_k, y_k, z_k)$ given by $(\pm 0.5, \pm 0.5, \pm 0.5)$.
Thus
\begin{align}
  \Omega_1 &= \{ (x,y,z) \in \RR^3 \st \phi(x,y,z) < 1 \}, \quad \Omega_2 = \Omega_0 \setminus \Omega_1,
\end{align}
with the level set function $\phi$ being defined by
\begin{equation*}
  \phi (x_1,x_2,x_2) =  \min_{0\leqslant k \leqslant 7} {\sqrt{(x-x_k)^2+(y-y_k)^2+(z-z_k)^2}} - r.
\end{equation*}
We set $k_1 = k_2 = 1$ and  construct a manufactured solution
from
\begin{equation}
    u_1 = u_2 =\sin{(3\pi x)} +\sin{(3\pi y)} +\sin{(3\pi z)}  \quad \mbox{ in } \Omega_1 \cup \Omega_2,
    \label{eq:ip:soln-no-cont}
  \end{equation}
leading to a smooth solution and solution gradient across the
interface $\Gamma$ and thus $g_D= g_N=0$.

 \begin{figure}[htb]
   \begin{center}
     \begin{minipage}[c]{0.4\linewidth}
       \includegraphics[width=1.0\textwidth]{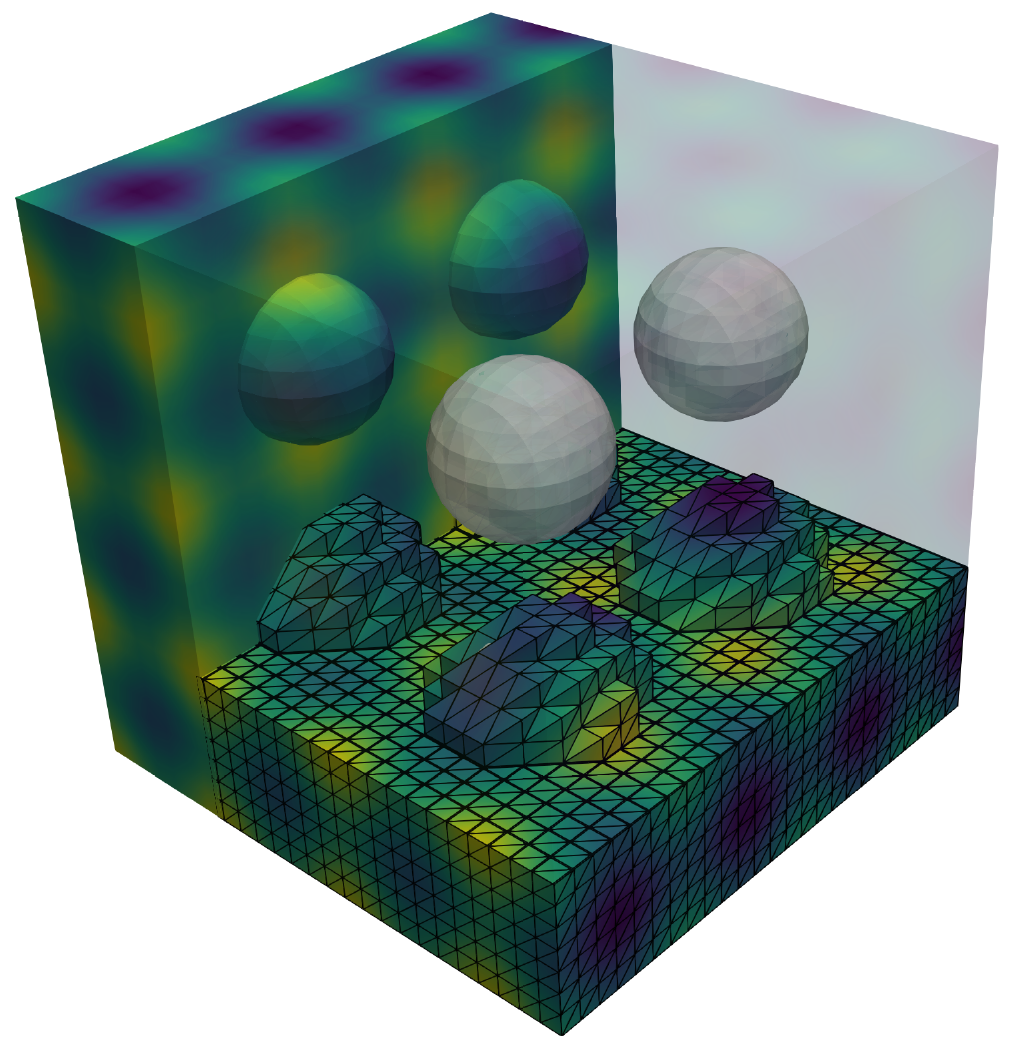}
     \end{minipage}
     \hspace{0.1cm}
     \begin{minipage}[c]{0.1\linewidth}
       \includegraphics[width=1.0\textwidth]{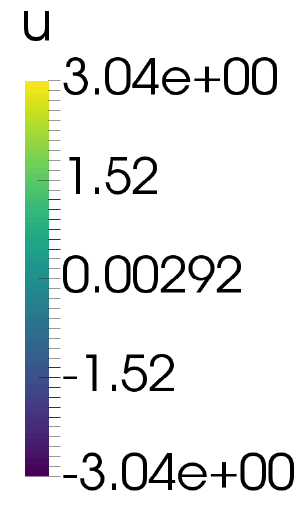}
     \end{minipage}
   \end{center}
   \caption{Solution plot for the first three-dimensional test case with
     no contrast and homogeneous jump conditions.
     The computed solution shows as expected a smooth transition across the interface.}
   \label{fig:sol_if_3d_no_constrast}
 \end{figure}
In the second test case, we place $8$ balls of radius $r_k=0.8$,
$k=0,\ldots 7$ at the $8$ corner points $(\pm 1, \pm 1, \pm 1)$ of
$\Omega_0$ and add a cylinder with radius $r_8 = 0.6$ centered around
$x$-axis. Using the standard notation $\delta_{ij}$ with  $\delta_{ij} = 1$ if $i=j$ and $0$ else,
we can define the level set function $\phi$ and corresponding domains $\Omega_1$ and $\Omega_2$ by
\begin{gather}
  \phi (x_1,x_2,x_2) =  \min_{0\leqslant k \leqslant 8} {\sqrt{(x-x_k)^2\delta_{k8} +(y-y_k)^2+(z-z_k)^2}} - r_k,
  \\
  \Omega_1 = \{ (x,y,z) \in \Omega_0 \st \phi(x,y,z) < 0\}, \quad \Omega_2 = \Omega_0 \setminus \Omega_1,
\end{gather}
respectively.
This time, we construct a manufactured solution with solution jumps
and gradient jumps across the interface by setting
\begin{alignat}{3}
    u_1 &= 1/\kappa_1(\sin{(3\pi x)} +\sin{(3\pi y)} +\sin{(3\pi z)}), \quad  &&\kappa_1=1  && \quad \mbox{ in } \Omega_1,  \\
    u_2 &=  1/\kappa_2(\cos{(3\pi x)} +\cos{(3\pi y)} +\cos{(3\pi z)}), \quad  &&\kappa_2= 10  && \quad \mbox{ in } \Omega_2.
    \label{eq:ip:soln-high-cont}
  \end{alignat}

\begin{figure}[htb]
  \begin{subfigure}{1.0\textwidth}
    \begin{center}
      \begin{minipage}[T]{0.47\linewidth}
        \vspace{0pt}
        \includegraphics[width=1.0\textwidth]{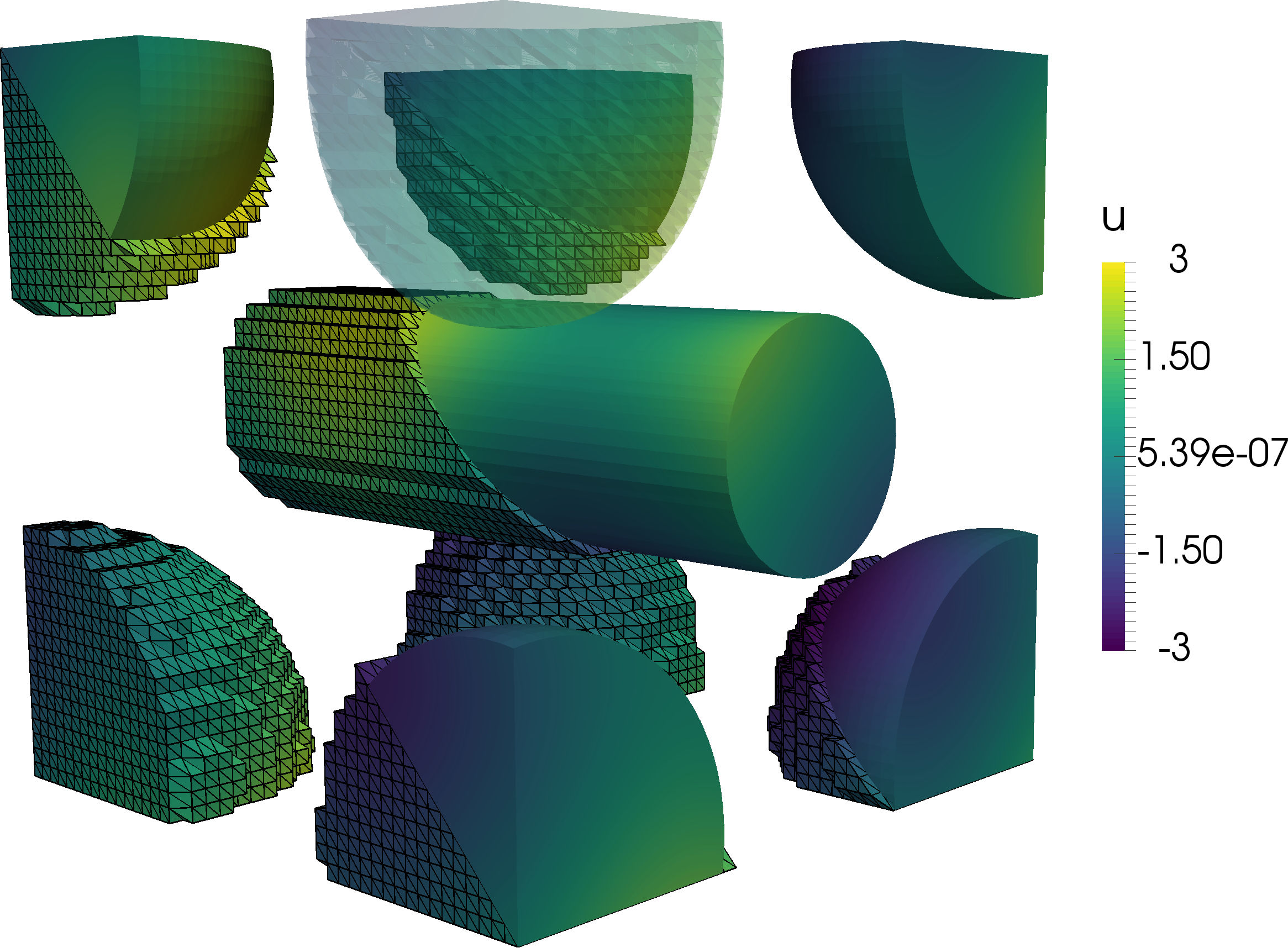}
      \end{minipage}
      \begin{minipage}[T]{0.47\linewidth}
        \vspace{0pt}
        \includegraphics[width=1.0\textwidth]{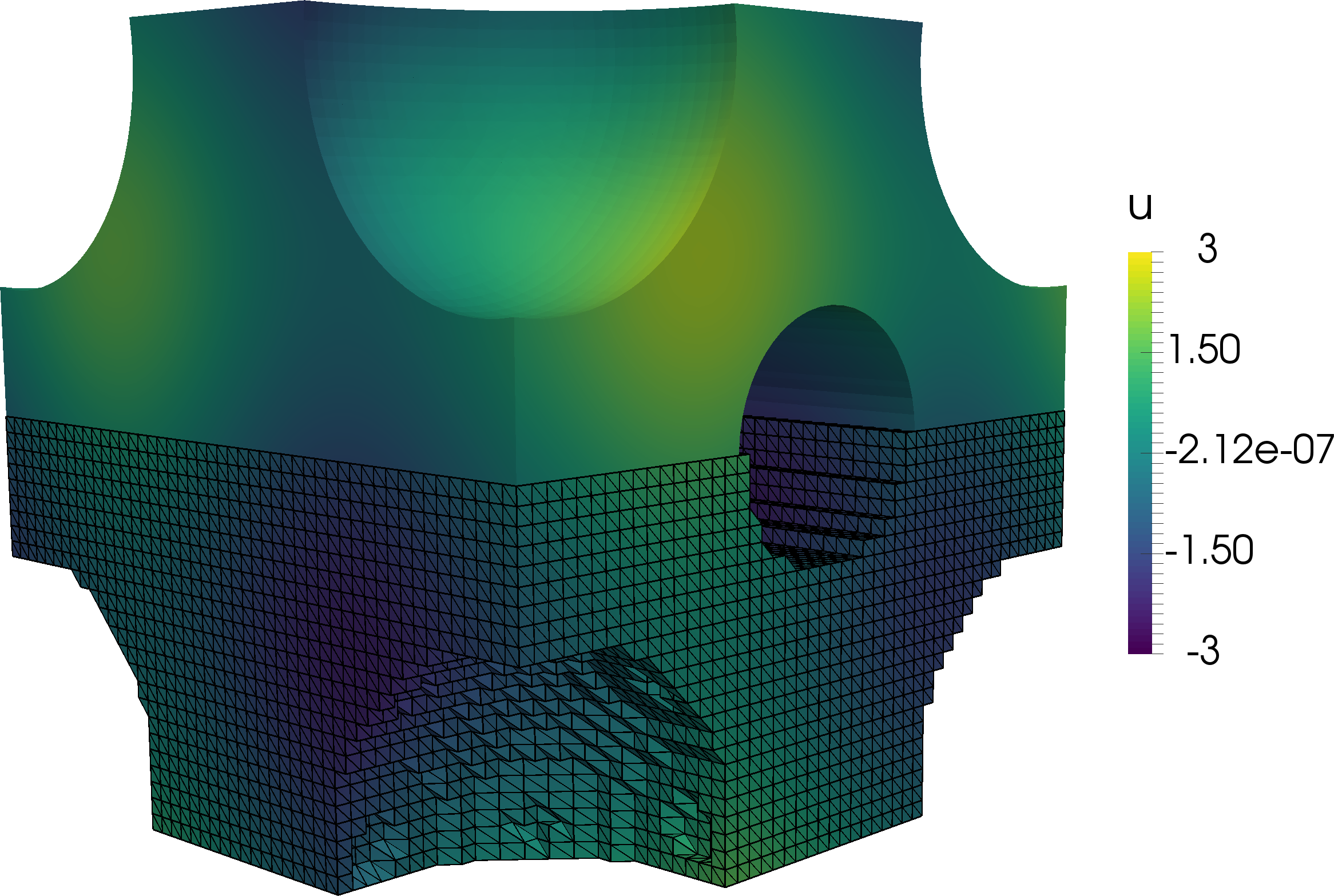}
      \end{minipage}
    \end{center}
  \end{subfigure}
   \caption{Solution plot for the second three-dimensional test case with
     a mild contrast and inhomogeneous jump conditions.
     The plots for $u_1$ (left) and $u_2$ (right)
     show the computed solution on parts of the physical domain as well as on its corresponding
     active mesh.}
      \label{fig:sol_if_3d_mild_constrast}
\end{figure}
Focusing on piecewise linear approximations,
we conduct a convergence study using $V_h = \PP_1(\mcT_{h,1}) \times \PP_1(\mcT_{h,2})$
on  a series of background  meshes $\{\mcT_k\}_{k=1}^7$
with mesh size $h_k = 2.2/N_k$ and $N_k \in \{6, 9, 12, 18, 24, 36 , 48\}$.
The computed EOC for both three-dimensional test cases are summarized in
Table~\eqref{tab:convanalysis_if_3d} and show optimal convergence
with respect to the $\|\kappa^{\onehalf}\nabla(\cdot)\|_{L^2(\Omega_1 \cup \Omega_2)}$ and
$\| \cdot \|_{L^2(\Omega_1 \cup \Omega_2)}$ norm. Solution plots
for the first and second test examples can be found in Figure~\ref{fig:sol_if_3d_no_constrast}
and~\ref{fig:sol_if_3d_mild_constrast}.

\begin{table}[htb]
  \begin{center}
    \begin {tabular}{r<{\pgfplotstableresetcolortbloverhangright }@{}l<{\pgfplotstableresetcolortbloverhangleft }|r<{\pgfplotstableresetcolortbloverhangright }@{}l<{\pgfplotstableresetcolortbloverhangleft }cr<{\pgfplotstableresetcolortbloverhangright }@{}l<{\pgfplotstableresetcolortbloverhangleft }r<{\pgfplotstableresetcolortbloverhangright }@{}l<{\pgfplotstableresetcolortbloverhangleft }|r<{\pgfplotstableresetcolortbloverhangright }@{}l<{\pgfplotstableresetcolortbloverhangleft }cr<{\pgfplotstableresetcolortbloverhangright }@{}l<{\pgfplotstableresetcolortbloverhangleft }c}%
\toprule \multicolumn {2}{c|}{$N_k$}&\multicolumn {2}{c}{$\|\kappa ^{\onehalf } \nabla e_k^1 \|_{\Omega }$}&EOC&\multicolumn {2}{c}{$\|e_k^1\|_{\Omega }$}&\multicolumn {2}{c|}{EOC}&\multicolumn {2}{c}{$\|\kappa ^{\onehalf } \nabla e_k^1 \|_{\Omega }$}&EOC&\multicolumn {2}{c}{$\|e_k^1\|_{\Omega }$}&EOC\\\midrule %
$6$&$$&$3.68$&$\cdot 10^{1}$&--&$4.20$&$\cdot 10^{0}$&--&&$1.06$&$\cdot 10^{1}$&--&$6.93$&$\cdot 10^{-1}$&--\\%
$9$&$$&$2.51$&$\cdot 10^{1}$&\pgfutilensuremath {0.94}&$2.47$&$\cdot 10^{0}$&$1$&$.31$&$6.94$&$\cdot 10^{0}$&\pgfutilensuremath {1.04}&$3.46$&$\cdot 10^{-1}$&\pgfutilensuremath {1.71}\\%
$12$&$$&$1.98$&$\cdot 10^{1}$&\pgfutilensuremath {0.83}&$1.76$&$\cdot 10^{0}$&$1$&$.19$&$5.20$&$\cdot 10^{0}$&\pgfutilensuremath {1.01}&$1.96$&$\cdot 10^{-1}$&\pgfutilensuremath {1.97}\\%
$18$&$$&$1.29$&$\cdot 10^{1}$&\pgfutilensuremath {1.05}&$9.09$&$\cdot 10^{-1}$&$1$&$.62$&$3.44$&$\cdot 10^{0}$&\pgfutilensuremath {1.02}&$8.45$&$\cdot 10^{-2}$&\pgfutilensuremath {2.08}\\%
$24$&$$&$9.48$&$\cdot 10^{0}$&\pgfutilensuremath {1.08}&$5.24$&$\cdot 10^{-1}$&$1$&$.92$&$2.57$&$\cdot 10^{0}$&\pgfutilensuremath {1.01}&$4.66$&$\cdot 10^{-2}$&\pgfutilensuremath {2.07}\\%
$36$&$$&$6.20$&$\cdot 10^{0}$&\pgfutilensuremath {1.05}&$2.36$&$\cdot 10^{-1}$&$1$&$.97$&$1.73$&$\cdot 10^{0}$&\pgfutilensuremath {0.98}&$2.01$&$\cdot 10^{-2}$&\pgfutilensuremath {2.07}\\%
$48$&$$&$4.63$&$\cdot 10^{0}$&\pgfutilensuremath {1.02}&$1.32$&$\cdot 10^{-1}$&$2$&$.01$&$1.29$&$\cdot 10^{0}$&\pgfutilensuremath {1.02}&$1.11$&$\cdot 10^{-2}$&\pgfutilensuremath {2.06}\\\bottomrule %
\end {tabular}%

    \caption{Convergence rates for the first (left) and second (right) three-dimensional interface test problem using $\PP_1(\mcT_h)$.}
  \label{tab:convanalysis_if_3d}
  \end{center}
\end{table}

\section*{Acknowledgments}
The authors gratefully acknowledge financial support 
from the Kempe Foundation Postdoc Scholarship JCK-1612
and from the Swedish Research Council under Starting Grant 2017-05038.

\bibliographystyle{elsarticle-num-names.bst}
\bibliography{bibliography}

\begin{thebibliography}{108}
\providecommand{\natexlab}[1]{#1}
\providecommand{\url}[1]{\texttt{#1}}
\providecommand{\urlprefix}{URL }
\expandafter\ifx\csname urlstyle\endcsname\relax
  \providecommand{\doi}[1]{doi:\discretionary{}{}{}#1}\else
  \providecommand{\doi}[1]{doi:\discretionary{}{}{}\begingroup
  \urlstyle{rm}\url{#1}\endgroup}\fi
\providecommand{\bibinfo}[2]{#2}

\bibitem[{Allaire et~al.(2014)Allaire, Dapogny, and
  Frey}]{AllaireDapognyFrey2014}
\bibinfo{author}{G.~Allaire}, \bibinfo{author}{C.~Dapogny},
  \bibinfo{author}{P.~Frey}, \bibinfo{title}{Shape optimization with a level
  set based mesh evolution method}, \bibinfo{journal}{Comput. Methods Appl.
  Mech. Engrg.} \bibinfo{volume}{282} (\bibinfo{year}{2014})
  \bibinfo{pages}{22--53}.

\bibitem[{Burman et~al.(2018)Burman, Elfverson, Hansbo, Larson, and
  Larsson}]{BurmanElfversonHansboEtAl2018}
\bibinfo{author}{E.~Burman}, \bibinfo{author}{D.~Elfverson},
  \bibinfo{author}{P.~Hansbo}, \bibinfo{author}{M.~G. Larson},
  \bibinfo{author}{K.~Larsson}, \bibinfo{title}{Shape optimization using the
  cut finite element method}, \bibinfo{journal}{Comput. Methods Appl. Mech.
  Engrg.} \bibinfo{volume}{328} (\bibinfo{year}{2018})
  \bibinfo{pages}{242--261}.

\bibitem[{Bernland et~al.(2017)Bernland, Wadbro, and
  Berggren}]{BernlandWadbroBerggren2017}
\bibinfo{author}{A.~Bernland}, \bibinfo{author}{E.~Wadbro},
  \bibinfo{author}{M.~Berggren}, \bibinfo{title}{Acoustic shape optimization
  using cut finite elements}, \bibinfo{journal}{Int. J. Numer. Methods Eng.}
  \bibinfo{volume}{113} (\bibinfo{year}{2017}) \bibinfo{pages}{432--449}.

\bibitem[{Tezduyar et~al.(2006)Tezduyar, Sathe, Keedy, and
  Stein}]{TezduyarSatheKeedyEtAl2006}
\bibinfo{author}{T.~E. Tezduyar}, \bibinfo{author}{S.~Sathe},
  \bibinfo{author}{R.~Keedy}, \bibinfo{author}{K.~Stein},
  \bibinfo{title}{Space--time finite element techniques for computation of
  fluid--structure interactions}, \bibinfo{journal}{Comput. Methods Appl. Mech.
  Engrg.} \bibinfo{volume}{195}~(\bibinfo{number}{17}) (\bibinfo{year}{2006})
  \bibinfo{pages}{2002--2027}.

\bibitem[{MR and K.(1986)}]{AlbertOneil1986}
\bibinfo{author}{A.~MR}, \bibinfo{author}{O.~K.}, \bibinfo{title}{Moving
  boundary-moving mesh analysis of phase change problems using finite elements
  and transfinite mappings}, \bibinfo{journal}{Int. J. Numer. Methods Eng.}
  \bibinfo{volume}{23} (\bibinfo{year}{1986}) \bibinfo{pages}{591--607}.

\bibitem[{Gro{\ss} et~al.(2006)Gro{\ss}, Reichelt, and
  Reusken}]{GrosReicheltReusken2006}
\bibinfo{author}{S.~Gro{\ss}}, \bibinfo{author}{V.~Reichelt},
  \bibinfo{author}{A.~Reusken}, \bibinfo{title}{A finite element based level
  set method for two-phase incompressible flows}, \bibinfo{journal}{Computing
  and Visualization in Science} \bibinfo{volume}{9}~(\bibinfo{number}{4})
  (\bibinfo{year}{2006}) \bibinfo{pages}{239--257}.

\bibitem[{Marchandise and Remacle(2006)}]{MarchandiseRemacle2006}
\bibinfo{author}{E.~Marchandise}, \bibinfo{author}{J.-F. Remacle},
  \bibinfo{title}{A stabilized finite element method using a discontinuous
  level set approach for solving two phase incompressible flows},
  \bibinfo{journal}{J. Comp. Phys.} \bibinfo{volume}{219}~(\bibinfo{number}{2})
  (\bibinfo{year}{2006}) \bibinfo{pages}{780--800}.

\bibitem[{Ganesan and Tobiska(2009)}]{GanesanTobiska2009}
\bibinfo{author}{S.~Ganesan}, \bibinfo{author}{L.~Tobiska}, \bibinfo{title}{A
  coupled arbitrary Lagrangian--Eulerian and Lagrangian method for computation
  of free surface flows with insoluble surfactants}, \bibinfo{journal}{J.
  Comput. Phys.} \bibinfo{volume}{228}~(\bibinfo{number}{8})
  (\bibinfo{year}{2009}) \bibinfo{pages}{2859--2873}.

\bibitem[{Dassi et~al.(2014)Dassi, Perotto, Formaggia, and
  Ruffo}]{DassiPerottoFormaggiaEtAl2014}
\bibinfo{author}{F.~Dassi}, \bibinfo{author}{S.~Perotto},
  \bibinfo{author}{L.~Formaggia}, \bibinfo{author}{P.~Ruffo},
  \bibinfo{title}{Efficient geometric reconstruction of complex geological
  structures}, \bibinfo{journal}{Math. Comput. Simul} \bibinfo{volume}{106}
  (\bibinfo{year}{2014}) \bibinfo{pages}{163--184}.

\bibitem[{Antiga et~al.(2009)Antiga, Peir{\'o}, and
  Steinman}]{AntigaPeiroSteinman2009}
\bibinfo{author}{L.~Antiga}, \bibinfo{author}{J.~Peir{\'o}},
  \bibinfo{author}{D.~A. Steinman}, \bibinfo{title}{From image data to
  computational domains}, in: \bibinfo{booktitle}{Cardiovascular Mathematics},
  \bibinfo{publisher}{Springer}, \bibinfo{pages}{123--175},
  \bibinfo{year}{2009}.

\bibitem[{Ural et~al.(2011)Ural, Zioupos, Buchanan, and
  Vashishth}]{UralZiouposBuchananetal2011}
\bibinfo{author}{A.~Ural}, \bibinfo{author}{P.~Zioupos},
  \bibinfo{author}{D.~Buchanan}, \bibinfo{author}{D.~Vashishth},
  \bibinfo{title}{The effect of strain rate on fracture toughness of human
  cortical bone: a finite element study}, \bibinfo{journal}{J. Mech. Behav.
  Biomed. Mater.} \bibinfo{volume}{4}~(\bibinfo{number}{7})
  (\bibinfo{year}{2011}) \bibinfo{pages}{1021--1032}.

\bibitem[{Cattaneo and Zunino(2014)}]{CattaneoZunino2014a}
\bibinfo{author}{L.~Cattaneo}, \bibinfo{author}{P.~Zunino},
  \bibinfo{title}{Computational models for fluid exchange between
  microcirculation and tissue interstitium}, \bibinfo{journal}{Networks and
  Heterogeneous Media} \bibinfo{volume}{9}~(\bibinfo{number}{1}).

\bibitem[{Dinh et~al.(1992)Dinh, Glowinski, He, Kwock, Pan, and
  Periaux}]{DinhGlowinskiHeEtAl1992}
\bibinfo{author}{Q.~V. Dinh}, \bibinfo{author}{R.~Glowinski},
  \bibinfo{author}{J.~He}, \bibinfo{author}{V.~Kwock}, \bibinfo{author}{T.~W.
  Pan}, \bibinfo{author}{J.~Periaux}, \bibinfo{title}{{Lagrange multiplier
  approach to fictitious domain methods: Application to fluid dynamics and
  electromagnetics}}, in: \bibinfo{booktitle}{Fifth Int. Symp. Domain Decompos.
  Methods Partial Differ. Equations, T. Chan, D. Keyes, G. Meurant, J. Scroggs,
  R. Voigt, eds., Philadelphia}, \bibinfo{pages}{151--194},
  \bibinfo{year}{1992}.

\bibitem[{Glowinski et~al.(1994)Glowinski, Pan, and
  P\'{e}riaux}]{GlowinskiPanPeriaux1994}
\bibinfo{author}{R.~Glowinski}, \bibinfo{author}{T.~Pan},
  \bibinfo{author}{J.~P\'{e}riaux}, \bibinfo{title}{A fictitious domain method
  for {D}irichlet problem and applications}, \bibinfo{journal}{Comput. Methods
  Appl. Mech. Engrg.} \bibinfo{volume}{111} (\bibinfo{year}{1994})
  \bibinfo{pages}{283--303}.

\bibitem[{Glowinski et~al.(1995)Glowinski, Pan, and
  P\'{e}riaux}]{GlowinskiPanPeriaux1995}
\bibinfo{author}{R.~Glowinski}, \bibinfo{author}{T.~Pan},
  \bibinfo{author}{J.~P\'{e}riaux}, \bibinfo{title}{A {L}agrange
  multiplier/fictitious domain method for {D}irichlet problem. Generalization
  to some flow problems}, \bibinfo{journal}{Jpn. J. Ind. Appl. Math.}
  \bibinfo{volume}{12} (\bibinfo{year}{1995}) \bibinfo{pages}{87--108}.

\bibitem[{Glowinski and Kuznetsov(2007)}]{GlowinskiKuznetsov2007}
\bibinfo{author}{R.~Glowinski}, \bibinfo{author}{Y.~Kuznetsov},
  \bibinfo{title}{{Distributed Lagrange multipliers based on fictitious domain
  method for second order elliptic problems}}, \bibinfo{journal}{Comput.
  Methods Appl. Mech. Engrg.} \bibinfo{volume}{196}~(\bibinfo{number}{8})
  (\bibinfo{year}{2007}) \bibinfo{pages}{1498--1506}.

\bibitem[{Glowinski et~al.(1997)Glowinski, Pan, and Periaux}]{Glowinski1997}
\bibinfo{author}{R.~Glowinski}, \bibinfo{author}{T.~Pan},
  \bibinfo{author}{J.~Periaux}, \bibinfo{title}{{A Lagrange
  multiplier/fictitious domain method for the numerical simulation of
  incompressible viscous flow around moving rigid bodies: (I) case where the
  rigid body motions are known a priori}}, \bibinfo{journal}{Comptes Rendus
  l'Acad{\'{e}}mie des Sci. - Ser. I - Math.}
  \bibinfo{volume}{324}~(\bibinfo{number}{3}) (\bibinfo{year}{1997})
  \bibinfo{pages}{361--369}.

\bibitem[{Glowinski et~al.(1999{\natexlab{a}})Glowinski, Pan, and
  Periaux}]{GlowinskiPanPeriaux1999}
\bibinfo{author}{R.~Glowinski}, \bibinfo{author}{T.~W. Pan},
  \bibinfo{author}{J.~Periaux}, \bibinfo{title}{{A distributed Lagrange
  multiplier/fictitious domain method for flows around moving rigid bodies:
  Application to particulate flow}}, \bibinfo{journal}{Int. J. Numer. Meth.
  Fluids} \bibinfo{volume}{30}~(\bibinfo{number}{3})
  (\bibinfo{year}{1999}{\natexlab{a}}) \bibinfo{pages}{1043--1066}.

\bibitem[{Glowinski et~al.(1999{\natexlab{b}})Glowinski, Pan, Hesla, and
  Joseph}]{GlowinskiPanHeslaEtAl1999}
\bibinfo{author}{R.~Glowinski}, \bibinfo{author}{T.-W. Pan},
  \bibinfo{author}{T.~Hesla}, \bibinfo{author}{D.~Joseph}, \bibinfo{title}{A
  distributed Lagrange multiplier/fictitious domain method for particulate
  flows}, \bibinfo{journal}{Int. J. Multiphase Flow}
  \bibinfo{volume}{25}~(\bibinfo{number}{5})
  (\bibinfo{year}{1999}{\natexlab{b}}) \bibinfo{pages}{755--794}.

\bibitem[{Glowinski et~al.(2001)Glowinski, Pan, Hesla, Joseph, and
  P\'{e}riaux}]{GlowinskiPanHeslaEtAl2001}
\bibinfo{author}{R.~Glowinski}, \bibinfo{author}{T.~Pan},
  \bibinfo{author}{T.~Hesla}, \bibinfo{author}{D.~Joseph},
  \bibinfo{author}{J.~P\'{e}riaux}, \bibinfo{title}{{A fictitious domain
  approach to the direct numerical simulation of incompressible viscous flow
  past moving rigid bodies: application to particulate flow}},
  \bibinfo{journal}{J. Comp. Phys.} \bibinfo{volume}{169}~(\bibinfo{number}{2})
  (\bibinfo{year}{2001}) \bibinfo{pages}{363--426}.

\bibitem[{Mo{\"e}s et~al.(1999)Mo{\"e}s, Dolbow, and
  Belytschko}]{MoesDolbowBelytschko1999}
\bibinfo{author}{N.~Mo{\"e}s}, \bibinfo{author}{J.~Dolbow},
  \bibinfo{author}{T.~Belytschko}, \bibinfo{title}{A finite element method for
  crack growth without remeshing}, \bibinfo{journal}{Int. J. Numer. Meth.
  Engng.} \bibinfo{volume}{46}~(\bibinfo{number}{1}) (\bibinfo{year}{1999})
  \bibinfo{pages}{131--150}.

\bibitem[{Melenk and Babu{\v{s}}ka(1996)}]{MelenkBabuska1996}
\bibinfo{author}{J.~Melenk}, \bibinfo{author}{I.~Babu{\v{s}}ka},
  \bibinfo{title}{The partition of unity finite element method: Basic theory
  and applications}, \bibinfo{journal}{Comput. Methods Appl. Mech. Engrg.}
  \bibinfo{volume}{139}~(\bibinfo{number}{1}) (\bibinfo{year}{1996})
  \bibinfo{pages}{289--314}.

\bibitem[{Chessa and Belytschko(2003)}]{ChessaBelytschko2003}
\bibinfo{author}{J.~Chessa}, \bibinfo{author}{T.~Belytschko},
  \bibinfo{title}{An extended finite element method for two-phase fluids},
  \bibinfo{journal}{J. Appl. Mech.} \bibinfo{volume}{70}~(\bibinfo{number}{1})
  (\bibinfo{year}{2003}) \bibinfo{pages}{10--17}.

\bibitem[{Zabaras et~al.(2006)Zabaras, Ganapathysubramanian, and
  Tan}]{ZabarasGanapathysubramanianTan2006}
\bibinfo{author}{N.~Zabaras}, \bibinfo{author}{B.~Ganapathysubramanian},
  \bibinfo{author}{L.~Tan}, \bibinfo{title}{Modelling dendritic solidification
  with melt convection using the extended finite element method},
  \bibinfo{journal}{J. Comput. Phys.}
  \bibinfo{volume}{218}~(\bibinfo{number}{1}) (\bibinfo{year}{2006})
  \bibinfo{pages}{200--227}.

\bibitem[{Abbas et~al.(2009)Abbas, Alizada, and Fries}]{AbbasAlizadaFries2009}
\bibinfo{author}{S.~Abbas}, \bibinfo{author}{A.~Alizada},
  \bibinfo{author}{T.~Fries}, \bibinfo{title}{{T}he {XFEM} for high-gradient
  solutions in convection-dominated problems}, \bibinfo{journal}{Int. J. Numer.
  Methods Eng.} \bibinfo{volume}{82}~(\bibinfo{number}{5})
  (\bibinfo{year}{2009}) \bibinfo{pages}{1044--1072}.

\bibitem[{Gerstenberger and Wall(2008)}]{GerstenbergerWall2008}
\bibinfo{author}{A.~Gerstenberger}, \bibinfo{author}{W.~A. Wall},
  \bibinfo{title}{An extended finite element method/Lagrange multiplier based
  approach for fluid--structure interaction}, \bibinfo{journal}{Comput. Methods
  Appl. Mech. Engrg.} \bibinfo{volume}{197}~(\bibinfo{number}{19})
  (\bibinfo{year}{2008}) \bibinfo{pages}{1699--1714}.

\bibitem[{Fumagalli(2012)}]{Fumagalli2012}
\bibinfo{author}{A.~Fumagalli}, \bibinfo{title}{Numerical modelling of flows in
  fractured porous media by the XFEM method}, Ph.D. thesis,
  \bibinfo{school}{Italy}, \bibinfo{year}{2012}.

\bibitem[{Fumagalli and Scotti(2014)}]{FumagalliScotti2014}
\bibinfo{author}{A.~Fumagalli}, \bibinfo{author}{A.~Scotti}, \bibinfo{title}{An
  Efficient XFEM Approximation of Darcy Flows in Arbitrarily Fractured Porous
  Media}, \bibinfo{journal}{Oil \& Gas Science and Technology--Revue d'IFP
  Energies nouvelles} \bibinfo{volume}{69}~(\bibinfo{number}{4})
  (\bibinfo{year}{2014}) \bibinfo{pages}{555--564}.

\bibitem[{Fries and Belytschko(2000)}]{FriesBelytschko2000}
\bibinfo{author}{T.-P. Fries}, \bibinfo{author}{T.~Belytschko},
  \bibinfo{title}{The extended/generalized finite element method: {A}n overview
  of the method and its applications}, \bibinfo{journal}{Int. J. Numer. Methods
  Eng.} \bibinfo{volume}{00} (\bibinfo{year}{2000}) \bibinfo{pages}{1--6}.

\bibitem[{Hansbo and Hansbo(2002)}]{HansboHansbo2002}
\bibinfo{author}{A.~Hansbo}, \bibinfo{author}{P.~Hansbo}, \bibinfo{title}{{An
  unfitted finite element method, based on Nitsche's method, for elliptic
  interface problems}}, \bibinfo{journal}{Comput. Methods Appl. Mech. Engrg.}
  \bibinfo{volume}{191}~(\bibinfo{number}{47-48}) (\bibinfo{year}{2002})
  \bibinfo{pages}{5537--5552}.

\bibitem[{Nitsche(1971)}]{Nitsche1971}
\bibinfo{author}{J.~Nitsche}, \bibinfo{title}{{\"{U}ber ein Variationsprinzip
  zur L\"{o}sung von Dirichlet-Problemen bei Verwendung von Teilr\"{a}umen, die
  keinen Randbedingungen unterworfen sind}}, \bibinfo{journal}{Abhandlungen aus
  dem Mathematischen Seminar der Universit\"{a}t Hamburg}
  \bibinfo{volume}{36}~(\bibinfo{number}{1}) (\bibinfo{year}{1971})
  \bibinfo{pages}{9--15}.

\bibitem[{Hansbo et~al.(2003)Hansbo, Hansbo, and
  Larson}]{HansboHansboLarson2003}
\bibinfo{author}{A.~Hansbo}, \bibinfo{author}{P.~Hansbo},
  \bibinfo{author}{M.~G. Larson}, \bibinfo{title}{A Finite Element Method on
  Composite Grids based on {N}itsche's Method}, \bibinfo{journal}{ESAIM: Math.
  Model. Numer. Anal.} \bibinfo{volume}{37}~(\bibinfo{number}{3})
  (\bibinfo{year}{2003}) \bibinfo{pages}{495--514}.

\bibitem[{Hansbo and Hansbo(2004)}]{HansboHansbo2004}
\bibinfo{author}{A.~Hansbo}, \bibinfo{author}{P.~Hansbo}, \bibinfo{title}{A
  finite element method for the simulation of strong and weak discontinuities
  in solid mechanics}, \bibinfo{journal}{Comput. Methods Appl. Mech. Engrg.}
  \bibinfo{volume}{193}~(\bibinfo{number}{33}) (\bibinfo{year}{2004})
  \bibinfo{pages}{3523--3540}.

\bibitem[{Areias and Belytschko(2006)}]{AreiasBelytschko2006}
\bibinfo{author}{P.~M. Areias}, \bibinfo{author}{T.~Belytschko},
  \bibinfo{title}{A comment on the article “A finite element method for
  simulation of strong and weak discontinuities in solid mechanics” by A.
  Hansbo and P. Hansbo [Comput. Methods Appl. Mech. Engrg. 193 (2004)
  3523--3540]}, \bibinfo{journal}{Comput. Methods Appl. Mech. Engrg.}
  \bibinfo{volume}{195}~(\bibinfo{number}{9}) (\bibinfo{year}{2006})
  \bibinfo{pages}{1275--1276}.

\bibitem[{Becker et~al.(2009)Becker, Burman, and
  Hansbo}]{BeckerBurmanHansbo2009}
\bibinfo{author}{R.~Becker}, \bibinfo{author}{E.~Burman},
  \bibinfo{author}{P.~Hansbo}, \bibinfo{title}{{A Nitsche extended finite
  element method for incompressible elasticity with discontinuous modulus of
  elasticity}}, \bibinfo{journal}{Comput. Methods Appl. Mech. Engrg.}
  \bibinfo{volume}{198}~(\bibinfo{number}{41-44}) (\bibinfo{year}{2009})
  \bibinfo{pages}{3352--3360}.

\bibitem[{Burman and Hansbo(2010)}]{BurmanHansbo2010}
\bibinfo{author}{E.~Burman}, \bibinfo{author}{P.~Hansbo},
  \bibinfo{title}{{Fictitious domain finite element methods using cut elements:
  I. A stabilized Lagrange multiplier method}}, \bibinfo{journal}{Comput.
  Methods Appl. Mech. Engrg.} \bibinfo{volume}{199} (\bibinfo{year}{2010})
  \bibinfo{pages}{2680--2686}.

\bibitem[{Burman(2010)}]{Burman2010}
\bibinfo{author}{E.~Burman}, \bibinfo{title}{{Ghost penalty}},
  \bibinfo{journal}{C.R. Math.} \bibinfo{volume}{348}~(\bibinfo{number}{21-22})
  (\bibinfo{year}{2010}) \bibinfo{pages}{1217--1220}.

\bibitem[{Burman and Hansbo(2012)}]{BurmanHansbo2012}
\bibinfo{author}{E.~Burman}, \bibinfo{author}{P.~Hansbo},
  \bibinfo{title}{{Fictitious domain finite element methods using cut elements:
  II. A stabilized Nitsche method}}, \bibinfo{journal}{Appl. Numer. Math.}
  \bibinfo{volume}{62}~(\bibinfo{number}{4}) (\bibinfo{year}{2012})
  \bibinfo{pages}{328--341}.

\bibitem[{Burman et~al.(2015{\natexlab{a}})Burman, Claus, Hansbo, Larson, and
  Massing}]{BurmanClausHansboEtAl2014}
\bibinfo{author}{E.~Burman}, \bibinfo{author}{S.~Claus},
  \bibinfo{author}{P.~Hansbo}, \bibinfo{author}{M.~G. Larson},
  \bibinfo{author}{A.~Massing}, \bibinfo{title}{{CutFEM: discretizing geometry
  and partial differential equations}}, \bibinfo{journal}{Internat. J. Numer.
  Meth. Engrg} \bibinfo{volume}{104}~(\bibinfo{number}{7})
  (\bibinfo{year}{2015}{\natexlab{a}}) \bibinfo{pages}{472--501}.

\bibitem[{Bordas et~al.(2018)Bordas, Burman, Larson, and
  Olshanskii}]{BordasBurmanLarsonEtAl2018}
\bibinfo{editor}{S.~Bordas}, \bibinfo{editor}{E.~Burman},
  \bibinfo{editor}{M.~Larson}, \bibinfo{editor}{M.~Olshanskii} (Eds.),
  \bibinfo{title}{Geometrically Unfitted Finite Element Methods and
  Applications}, \bibinfo{publisher}{Springer}, \bibinfo{year}{2018}.

\bibitem[{Burman and Zunino(2012)}]{BurmanZunino2012}
\bibinfo{author}{E.~Burman}, \bibinfo{author}{P.~Zunino},
  \bibinfo{title}{{Numerical approximation of large contrast problems with the
  unfitted Nitsche method}}, \bibinfo{journal}{Front. Numer. Anal. 2010}
  (\bibinfo{year}{2012}) \bibinfo{pages}{1--54}.

\bibitem[{{Guzman} et~al.(2015){Guzman}, {Sanchez}, and
  {Sarkis}}]{GuzmanSanchezSarkis2015}
\bibinfo{author}{J.~{Guzman}}, \bibinfo{author}{M.~A. {Sanchez}},
  \bibinfo{author}{M.~{Sarkis}}, \bibinfo{title}{{A finite element method for
  high-contrast interface problems with error estimates independent of
  contrast}}, \bibinfo{journal}{ArXiv e-prints} .

\bibitem[{Burman et~al.(2017)Burman, Guzmán, Sánchez, and
  Sarkis}]{BurmanGuzmanSanchezEtAl2017}
\bibinfo{author}{E.~Burman}, \bibinfo{author}{J.~Guzmán},
  \bibinfo{author}{M.~A. Sánchez}, \bibinfo{author}{M.~Sarkis},
  \bibinfo{title}{Robust flux error estimation of an unfitted Nitsche method
  for high-contrast interface problems}, \bibinfo{journal}{IMA J. Numer. Anal.}
  .

\bibitem[{Burman and Hansbo(2014)}]{BurmanHansbo2013}
\bibinfo{author}{E.~Burman}, \bibinfo{author}{P.~Hansbo},
  \bibinfo{title}{Fictitious domain methods using cut elements: {III}. {A
  stabilized Nitsche method for Stokes' problem}}, \bibinfo{journal}{ESAIM:
  Math. Model. Numer. Anal.} \bibinfo{volume}{48}~(\bibinfo{number}{3})
  (\bibinfo{year}{2014}) \bibinfo{pages}{859--874}.

\bibitem[{Massing et~al.(2014{\natexlab{a}})Massing, Larson, Logg, and
  Rognes}]{MassingLarsonLoggEtAl2013}
\bibinfo{author}{A.~Massing}, \bibinfo{author}{M.~G. Larson},
  \bibinfo{author}{A.~Logg}, \bibinfo{author}{M.~E. Rognes}, \bibinfo{title}{{A
  stabilized Nitsche overlapping mesh method for the Stokes problem}},
  \bibinfo{journal}{Numer. Math.} \bibinfo{volume}{128}~(\bibinfo{number}{1})
  (\bibinfo{year}{2014}{\natexlab{a}}) \bibinfo{pages}{73--101}.

\bibitem[{Burman et~al.(2015{\natexlab{b}})Burman, Claus, and
  Massing}]{BurmanClausMassing2015}
\bibinfo{author}{E.~Burman}, \bibinfo{author}{S.~Claus},
  \bibinfo{author}{A.~Massing}, \bibinfo{title}{A Stabilized Cut Finite Element
  Method for the Three Field {S}tokes Problem}, \bibinfo{journal}{SIAM J. Sci.
  Comput.} \bibinfo{volume}{37}~(\bibinfo{number}{4})
  (\bibinfo{year}{2015}{\natexlab{b}}) \bibinfo{pages}{A1705--A1726}.

\bibitem[{Cattaneo et~al.(2014)Cattaneo, Formaggia, Iori, Scotti, and
  Zunino}]{CattaneoFormaggiaIoriEtAl2014}
\bibinfo{author}{L.~Cattaneo}, \bibinfo{author}{L.~Formaggia},
  \bibinfo{author}{G.~F. Iori}, \bibinfo{author}{A.~Scotti},
  \bibinfo{author}{P.~Zunino}, \bibinfo{title}{Stabilized extended finite
  elements for the approximation of saddle point problems with unfitted
  interfaces}, \bibinfo{journal}{Calcolo}  (\bibinfo{year}{2014})
  \bibinfo{pages}{1--30}.

\bibitem[{Massing et~al.(2018)Massing, Schott, and
  Wall}]{MassingSchottWall2017}
\bibinfo{author}{A.~Massing}, \bibinfo{author}{B.~Schott},
  \bibinfo{author}{W.~Wall}, \bibinfo{title}{{A stabilized Nitsche cut finite
  element method for the Oseen problem}}, \bibinfo{journal}{Comput. Methods
  Appl. Mech. Engrg.} \bibinfo{volume}{328} (\bibinfo{year}{2018})
  \bibinfo{pages}{262--300}.

\bibitem[{Winter et~al.(2017)Winter, Schott, Massing, and
  Wall}]{WinterSchottMassingEtAl2017}
\bibinfo{author}{M.~Winter}, \bibinfo{author}{B.~Schott},
  \bibinfo{author}{A.~Massing}, \bibinfo{author}{W.~Wall}, \bibinfo{title}{{A
  Nitsche cut finite element method for the Oseen problem with general Navier
  boundary conditions}}, \bibinfo{journal}{Comput. Methods Appl. Mech. Engrg.}
  \bibinfo{volume}{330} (\bibinfo{year}{2017}) \bibinfo{pages}{220--252}.

\bibitem[{Kirchhart et~al.(2016)Kirchhart, Gro{\ss}, and
  Reusken}]{KirchhartGrosReusken2016}
\bibinfo{author}{M.~Kirchhart}, \bibinfo{author}{S.~Gro{\ss}},
  \bibinfo{author}{A.~Reusken}, \bibinfo{title}{Analysis of an XFEM
  discretization for Stokes interface problems}, \bibinfo{journal}{SIAM J. Sci.
  Comput.} \bibinfo{volume}{38}~(\bibinfo{number}{2}) (\bibinfo{year}{2016})
  \bibinfo{pages}{A1019--A1043}.

\bibitem[{Gro{\ss} et~al.(2016)Gro{\ss}, Ludescher, Olshanskii, and
  Reusken}]{GrosLudescherOlshanskiiEtAl2016}
\bibinfo{author}{S.~Gro{\ss}}, \bibinfo{author}{T.~Ludescher},
  \bibinfo{author}{M.~Olshanskii}, \bibinfo{author}{A.~Reusken},
  \bibinfo{title}{Robust preconditioning for XFEM applied to time-dependent
  Stokes problems}, \bibinfo{journal}{SIAM J. Sci. Comput.}
  \bibinfo{volume}{38}~(\bibinfo{number}{6}) (\bibinfo{year}{2016})
  \bibinfo{pages}{A3492--A3514}.

\bibitem[{Guzm{\'a}n and Olshanskii(2017)}]{GuzmanOlshanskii2016}
\bibinfo{author}{J.~Guzm{\'a}n}, \bibinfo{author}{M.~Olshanskii},
  \bibinfo{title}{{Inf-sup stability of geometrically unfitted Stokes finite
  elements}}, \bibinfo{journal}{To appear in Math. Comp.} .

\bibitem[{Schott et~al.(2015)Schott, Rasthofer, Gravemeier, and
  Wall}]{SchottRasthoferGravemeierEtAl2015}
\bibinfo{author}{B.~Schott}, \bibinfo{author}{U.~Rasthofer},
  \bibinfo{author}{V.~Gravemeier}, \bibinfo{author}{W.~A. Wall},
  \bibinfo{title}{{A face-oriented stabilized Nitsche-type extended variational
  multiscale method for incompressible two-phase flow}}, \bibinfo{journal}{Int.
  J. Numer. Methods Eng.} \bibinfo{volume}{104}~(\bibinfo{number}{7})
  (\bibinfo{year}{2015}) \bibinfo{pages}{721--748}.

\bibitem[{Gro{\ss} and Reusken(2011)}]{GrossReusken2011}
\bibinfo{author}{S.~Gro{\ss}}, \bibinfo{author}{A.~Reusken},
  \bibinfo{title}{Numerical methods for two-phase incompressible flows},
  vol.~\bibinfo{volume}{40}, \bibinfo{publisher}{Springer},
  \bibinfo{year}{2011}.

\bibitem[{Massing et~al.(2015)Massing, Larson, Logg, and
  Rognes}]{MassingLarsonLoggEtAl2015}
\bibinfo{author}{A.~Massing}, \bibinfo{author}{M.~G. Larson},
  \bibinfo{author}{A.~Logg}, \bibinfo{author}{M.~Rognes}, \bibinfo{title}{{A
  Nitsche-based cut finite element method for a fluid-structure interaction
  problem}}, \bibinfo{journal}{Commun. Appl. Math. Comput. Sci.}
  \bibinfo{volume}{10}~(\bibinfo{number}{2}) (\bibinfo{year}{2015})
  \bibinfo{pages}{97--120}.

\bibitem[{Olshanskii et~al.(2009)Olshanskii, Reusken, and
  Grande}]{OlshanskiiReuskenGrande2009}
\bibinfo{author}{M.~A. Olshanskii}, \bibinfo{author}{A.~Reusken},
  \bibinfo{author}{J.~Grande}, \bibinfo{title}{A finite element method for
  elliptic equations on surfaces}, \bibinfo{journal}{SIAM J. Numer. Anal.}
  \bibinfo{volume}{47}~(\bibinfo{number}{5}) (\bibinfo{year}{2009})
  \bibinfo{pages}{3339--3358}.

\bibitem[{Olshanskii and Reusken(2010)}]{OlshanskiiReusken2010}
\bibinfo{author}{M.~A. Olshanskii}, \bibinfo{author}{A.~Reusken},
  \bibinfo{title}{A finite element method for surface {PDE}s: matrix
  properties}, \bibinfo{journal}{Numer. Math.}
  \bibinfo{volume}{114}~(\bibinfo{number}{3}) (\bibinfo{year}{2010})
  \bibinfo{pages}{491--520}.

\bibitem[{Burman et~al.(2015{\natexlab{c}})Burman, Hansbo, and
  Larson}]{BurmanHansboLarson2015}
\bibinfo{author}{E.~Burman}, \bibinfo{author}{P.~Hansbo},
  \bibinfo{author}{M.~G. Larson}, \bibinfo{title}{{A stabilized cut finite
  element method for partial differential equations on surfaces: The
  Laplace--Beltrami operator}}, \bibinfo{journal}{Comput. Methods Appl. Mech.
  Engrg.} \bibinfo{volume}{285} (\bibinfo{year}{2015}{\natexlab{c}})
  \bibinfo{pages}{188--207}.

\bibitem[{{Burman} et~al.(2016{\natexlab{a}}){Burman}, {Hansbo}, {Larson}, and
  {Massing}}]{BurmanHansboLarsonEtAl2016}
\bibinfo{author}{E.~{Burman}}, \bibinfo{author}{P.~{Hansbo}},
  \bibinfo{author}{M.~G. {Larson}}, \bibinfo{author}{A.~{Massing}},
  \bibinfo{title}{{Cut Finite Element Methods for Partial Differential
  Equations on Embedded Manifolds of Arbitrary Codimensions}},
  \bibinfo{journal}{ArXiv e-prints} .

\bibitem[{{Burman} et~al.(2016{\natexlab{b}}){Burman}, {Hansbo}, {Larson},
  {Massing}, and {Zahedi}}]{BurmanHansboLarsonEtAl2016c}
\bibinfo{author}{E.~{Burman}}, \bibinfo{author}{P.~{Hansbo}},
  \bibinfo{author}{M.~G. {Larson}}, \bibinfo{author}{A.~{Massing}},
  \bibinfo{author}{S.~{Zahedi}}, \bibinfo{title}{Full gradient stabilized cut
  finite element methods for surface partial differential equations},
  \bibinfo{journal}{Comput. Methods Appl. Mech. Engrg.} \bibinfo{volume}{310}
  (\bibinfo{year}{2016}{\natexlab{b}}) \bibinfo{pages}{278--296}.

\bibitem[{Hansbo et~al.(2017)Hansbo, Larson, and
  Massing}]{HansboLarsonMassing2017a}
\bibinfo{author}{P.~Hansbo}, \bibinfo{author}{M.~Larson},
  \bibinfo{author}{A.~Massing}, \bibinfo{title}{{A stabilized cut finite
  element method for the Darcy problem on surfaces}}, \bibinfo{journal}{Comput.
  Methods Appl. Mech. Engrg.} \bibinfo{volume}{326} (\bibinfo{year}{2017})
  \bibinfo{pages}{298--318}.

\bibitem[{{Grande} et~al.(2016){Grande}, {Lehrenfeld}, and
  {Reusken}}]{GrandeLehrenfeldReusken2016}
\bibinfo{author}{J.~{Grande}}, \bibinfo{author}{C.~{Lehrenfeld}},
  \bibinfo{author}{A.~{Reusken}}, \bibinfo{title}{{Analysis of a high order
  Trace Finite Element Method for PDEs on level set surfaces}},
  \bibinfo{journal}{ArXiv e-prints} .

\bibitem[{Hansbo et~al.(2016)Hansbo, Larson, and
  Zahedi}]{HansboLarsonZahedi2016}
\bibinfo{author}{P.~Hansbo}, \bibinfo{author}{M.~G. Larson},
  \bibinfo{author}{S.~Zahedi}, \bibinfo{title}{A cut finite element method for
  coupled bulk-surface problems on time-dependent domains},
  \bibinfo{journal}{Comput. Methods Appl. Mech. Engrg.} \bibinfo{volume}{307}
  (\bibinfo{year}{2016}) \bibinfo{pages}{96--116}.

\bibitem[{Reusken(2014)}]{Reusken2014}
\bibinfo{author}{A.~Reusken}, \bibinfo{title}{Analysis of trace finite element
  methods for surface partial differential equations}, \bibinfo{journal}{IMA J.
  Numer. Anal.} \bibinfo{volume}{35} (\bibinfo{year}{2014})
  \bibinfo{pages}{1568--1590}.

\bibitem[{Olshanskii and Reusken(2014)}]{OlshanskiiReusken2014}
\bibinfo{author}{M.~A. Olshanskii}, \bibinfo{author}{A.~Reusken},
  \bibinfo{title}{Error analysis of a space-time finite element method for
  solving {PDE}s on evolving surfaces}, \bibinfo{journal}{SIAM J. Numer. Anal.}
  \bibinfo{volume}{52}~(\bibinfo{number}{4}) (\bibinfo{year}{2014})
  \bibinfo{pages}{2092--2120}.

\bibitem[{Flemisch et~al.(2016)Flemisch, Fumagalli, and
  Scotti}]{FlemischFumagalliScotti2016}
\bibinfo{author}{B.~Flemisch}, \bibinfo{author}{A.~Fumagalli},
  \bibinfo{author}{A.~Scotti}, \bibinfo{title}{A Review of the XFEM-Based
  Approximation of Flow in Fractured Porous Media}, in:
  \bibinfo{booktitle}{Advances in Discretization Methods: Discontinuities,
  Virtual Elements, Fictitious Domain Methods}, vol.~\bibinfo{volume}{12},
  \bibinfo{publisher}{Springer}, \bibinfo{pages}{47--76}, \bibinfo{year}{2016}.

\bibitem[{D'Angelo and Scotti(2012)}]{DAngeloScotti2012}
\bibinfo{author}{C.~D'Angelo}, \bibinfo{author}{A.~Scotti}, \bibinfo{title}{A
  mixed finite element method for Darcy flow in fractured porous media with
  non-matching grids}, \bibinfo{journal}{ESAIM: Math. Model. Numer. Anal.}
  \bibinfo{volume}{46}~(\bibinfo{number}{02}) (\bibinfo{year}{2012})
  \bibinfo{pages}{465--489}.

\bibitem[{Formaggia et~al.(2013)Formaggia, Fumagalli, Scotti, and
  Ruffo}]{FormaggiaFumagalliScottiEtAl2013}
\bibinfo{author}{L.~Formaggia}, \bibinfo{author}{A.~Fumagalli},
  \bibinfo{author}{A.~Scotti}, \bibinfo{author}{P.~Ruffo}, \bibinfo{title}{{A
  reduced model for Darcy's problem in networks of fractures}},
  \bibinfo{journal}{ESAIM: Math. Model. Numer. Anal.}
  \bibinfo{volume}{48}~(\bibinfo{number}{4}) (\bibinfo{year}{2013})
  \bibinfo{pages}{1089--1116}.

\bibitem[{Parvizian et~al.(2007)Parvizian, D{\"u}ster, and
  Rank}]{ParvizianDuesterRank2007}
\bibinfo{author}{J.~Parvizian}, \bibinfo{author}{A.~D{\"u}ster},
  \bibinfo{author}{E.~Rank}, \bibinfo{title}{Finite cell method},
  \bibinfo{journal}{Comput. Mech.} \bibinfo{volume}{41}~(\bibinfo{number}{1})
  (\bibinfo{year}{2007}) \bibinfo{pages}{121--133}.

\bibitem[{Varduhn et~al.(2016)Varduhn, Hsu, Ruess, and
  Schillinger}]{VarduhnHsuRuessEtAl2016}
\bibinfo{author}{V.~Varduhn}, \bibinfo{author}{M.-C. Hsu},
  \bibinfo{author}{M.~Ruess}, \bibinfo{author}{D.~Schillinger},
  \bibinfo{title}{The tetrahedral finite cell method: Higher-order
  immersogeometric analysis on adaptive non-boundary-fitted meshes},
  \bibinfo{journal}{Int. J. Numer. Meth. Engng.} \bibinfo{volume}{107}
  (\bibinfo{year}{2016}) \bibinfo{pages}{1054--1079}.

\bibitem[{Xu et~al.(2016)Xu, Schillinger, Kamensky, Varduhn, Wang, and
  Hsu}]{XuSchillingerKamenskyEtAl2015}
\bibinfo{author}{F.~Xu}, \bibinfo{author}{D.~Schillinger},
  \bibinfo{author}{D.~Kamensky}, \bibinfo{author}{V.~Varduhn},
  \bibinfo{author}{C.~Wang}, \bibinfo{author}{M.-C. Hsu}, \bibinfo{title}{The
  tetrahedral finite cell method for fluids: Immersogeometric analysis of
  turbulent flow around complex geometries}, \bibinfo{journal}{Comput. Fluids}
  \bibinfo{volume}{141} (\bibinfo{year}{2016}) \bibinfo{pages}{135--154}.

\bibitem[{Schillinger and Ruess(2014)}]{SchillingerRuess2014}
\bibinfo{author}{D.~Schillinger}, \bibinfo{author}{M.~Ruess},
  \bibinfo{title}{The finite cell method: A review in the context of
  higher-order structural analysis of cad and image-based geometric models},
  \bibinfo{journal}{Arch. Comput. Methods Eng.}  (\bibinfo{year}{2014})
  \bibinfo{pages}{1--65}.

\bibitem[{Bastian and Engwer(2009)}]{BastianEngwer2009}
\bibinfo{author}{P.~Bastian}, \bibinfo{author}{C.~Engwer}, \bibinfo{title}{{An
  unfitted finite element method using discontinuous Galerkin}},
  \bibinfo{journal}{Internat. J. Numer. Meth. Engrg}
  \bibinfo{volume}{79}~(\bibinfo{number}{12}) (\bibinfo{year}{2009})
  \bibinfo{pages}{1557--1576}.

\bibitem[{Bastian et~al.(2011)Bastian, Engwer, Fahlke, and
  Ippisch}]{BastianEngwerFahlkeEtAl2011}
\bibinfo{author}{P.~Bastian}, \bibinfo{author}{C.~Engwer},
  \bibinfo{author}{J.~Fahlke}, \bibinfo{author}{O.~Ippisch}, \bibinfo{title}{An
  Unfitted Discontinuous Galerkin method for pore-scale simulations of solute
  transport}, \bibinfo{journal}{Math. Comput. Simul}
  \bibinfo{volume}{81}~(\bibinfo{number}{10}) (\bibinfo{year}{2011})
  \bibinfo{pages}{2051--2061}.

\bibitem[{Saye(2015)}]{Saye2015}
\bibinfo{author}{R.~I. Saye}, \bibinfo{title}{High-Order Quadrature Methods for
  Implicitly Defined Surfaces and Volumes in Hyperrectangles},
  \bibinfo{journal}{SIAM J. Sci. Comput.}
  \bibinfo{volume}{37}~(\bibinfo{number}{2}) (\bibinfo{year}{2015})
  \bibinfo{pages}{A993--A1019}.

\bibitem[{Sollie et~al.(2011)Sollie, Bokhove, and van~der
  Vegt}]{SollieBokhoveVegt2011}
\bibinfo{author}{W.~E.~H. Sollie}, \bibinfo{author}{O.~Bokhove},
  \bibinfo{author}{J.~J.~W. van~der Vegt}, \bibinfo{title}{Space--time
  discontinuous {G}alerkin finite element method for two-fluid flows},
  \bibinfo{journal}{J. Comput. Phys.}
  \bibinfo{volume}{230}~(\bibinfo{number}{3}) (\bibinfo{year}{2011})
  \bibinfo{pages}{789--817}.

\bibitem[{Heimann et~al.(2013)Heimann, Engwer, Ippisch, and
  Bastian}]{HeimannEngwerIppischEtAl2013}
\bibinfo{author}{F.~Heimann}, \bibinfo{author}{C.~Engwer},
  \bibinfo{author}{O.~Ippisch}, \bibinfo{author}{P.~Bastian},
  \bibinfo{title}{An unfitted interior penalty discontinuous {G}alerkin method
  for incompressible {N}avier--{S}tokes two--phase flow},
  \bibinfo{journal}{Internat. J. Numer. Methods Fluids}
  \bibinfo{volume}{71}~(\bibinfo{number}{3}) (\bibinfo{year}{2013})
  \bibinfo{pages}{269--293}.

\bibitem[{Saye(2017)}]{Saye2017}
\bibinfo{author}{R.~Saye}, \bibinfo{title}{Implicit mesh discontinuous Galerkin
  methods and interfacial gauge methods for high-order accurate interface
  dynamics, with applications to surface tension dynamics, rigid body
  fluid-structure interaction, and free surface flow: Part I},
  \bibinfo{journal}{J. Comput. Phys.} \bibinfo{volume}{344}
  (\bibinfo{year}{2017}) \bibinfo{pages}{647--682}.

\bibitem[{M{\"u}ller et~al.(2016)M{\"u}ller, Kr{\"a}mer-Eis, Kummer, and
  Oberlack}]{MuellerKraemer-EisKummerEtAl2016}
\bibinfo{author}{B.~M{\"u}ller}, \bibinfo{author}{S.~Kr{\"a}mer-Eis},
  \bibinfo{author}{F.~Kummer}, \bibinfo{author}{M.~Oberlack},
  \bibinfo{title}{{A high-order Discontinuous Galerkin method for compressible
  flows with immersed boundaries}}, \bibinfo{journal}{Int. J. Numer. Methods
  Eng.} \bibinfo{volume}{110}~(\bibinfo{number}{1}) (\bibinfo{year}{2016})
  \bibinfo{pages}{3--30}, \bibinfo{note}{nme.5343}.

\bibitem[{Krause and Kummer(2017)}]{KrauseKummer2017}
\bibinfo{author}{D.~Krause}, \bibinfo{author}{F.~Kummer}, \bibinfo{title}{An
  Incompressible Immersed Boundary Solver for Moving Body Flows using a Cut
  Cell Discontinuous Galerkin Method}, \bibinfo{journal}{Comput. Fluids}
  \bibinfo{volume}{153} (\bibinfo{year}{2017}) \bibinfo{pages}{118--129}.

\bibitem[{Badia et~al.(2017)Badia, Verdugo, and
  Mart{\'\i}n}]{BadiaVerdugoMartin2017}
\bibinfo{author}{S.~Badia}, \bibinfo{author}{F.~Verdugo},
  \bibinfo{author}{A.~F. Mart{\'\i}n}, \bibinfo{title}{The aggregated unfitted
  finite element method for elliptic problems}, \bibinfo{journal}{arXiv
  preprint arXiv:1709.09122} .

\bibitem[{Badia and Verdugo(2017)}]{BadiaVerdugo2017a}
\bibinfo{author}{S.~Badia}, \bibinfo{author}{F.~Verdugo},
  \bibinfo{title}{Robust and scalable domain decomposition solvers for unfitted
  finite element methods}, \bibinfo{journal}{J. Comput. Appl. Math.} .

\bibitem[{Massjung(2012)}]{Massjung2012}
\bibinfo{author}{R.~Massjung}, \bibinfo{title}{{An unfitted discontinuous
  Galerkin method applied to elliptic interface problems}},
  \bibinfo{journal}{SIAM J. Numer. Anal.}
  \bibinfo{volume}{50}~(\bibinfo{number}{6}) (\bibinfo{year}{2012})
  \bibinfo{pages}{3134--3162}.

\bibitem[{Johansson and Larson(2013)}]{JohanssonLarson2013}
\bibinfo{author}{A.~Johansson}, \bibinfo{author}{M.~Larson}, \bibinfo{title}{A
  High Order Discontinuous {G}alerkin {N}itsche Method for Elliptic Problems
  with Fictitious Boundary}, \bibinfo{journal}{Numer. Math.}
  \bibinfo{volume}{123}~(\bibinfo{number}{4}).

\bibitem[{Antonietti et~al.(2016{\natexlab{a}})Antonietti, Cangiani, Collis,
  Dong, Georgoulis, Giani, and Houston}]{AntoniettiCangianiCollisEtAl2016}
\bibinfo{author}{P.~F. Antonietti}, \bibinfo{author}{A.~Cangiani},
  \bibinfo{author}{J.~Collis}, \bibinfo{author}{Z.~Dong},
  \bibinfo{author}{E.~H. Georgoulis}, \bibinfo{author}{S.~Giani},
  \bibinfo{author}{P.~Houston}, \bibinfo{title}{Review of discontinuous
  Galerkin finite element methods for partial differential equations on
  complicated domains}, in: \bibinfo{booktitle}{Building Bridges: Connections
  and Challenges in Modern Approaches to Numerical Partial Differential
  Equations}, \bibinfo{publisher}{Springer}, \bibinfo{pages}{279--308},
  \bibinfo{year}{2016}{\natexlab{a}}.

\bibitem[{Antonietti et~al.(2016{\natexlab{b}})Antonietti, Facciola, Russo, and
  Verani}]{AntoniettiFacciolaRussoEtAl2016a}
\bibinfo{author}{P.~F. Antonietti}, \bibinfo{author}{C.~Facciola},
  \bibinfo{author}{A.~Russo}, \bibinfo{author}{M.~Verani},
  \bibinfo{title}{Discontinuous Galerkin approximation of flows in fractured
  porous media on polytopic grids}, \bibinfo{type}{Tech. Rep.},
  \bibinfo{institution}{MOX, Dipartimento di Matematica, Politecnico di
  Milano}, \bibinfo{year}{2016}{\natexlab{b}}.

\bibitem[{Giani and Houston(2014)}]{GianiHouston2014a}
\bibinfo{author}{S.~Giani}, \bibinfo{author}{P.~Houston},
  \bibinfo{title}{Goal-oriented adaptive composite discontinuous Galerkin
  methods for incompressible flows}, \bibinfo{journal}{J. Comput. Appl. Math.}
  \bibinfo{volume}{270} (\bibinfo{year}{2014}) \bibinfo{pages}{32--42}.

\bibitem[{M{\"u}ller et~al.(2013)M{\"u}ller, Kummer, and
  Oberlack}]{MuellerKummerOberlack2013}
\bibinfo{author}{B.~M{\"u}ller}, \bibinfo{author}{F.~Kummer},
  \bibinfo{author}{M.~Oberlack}, \bibinfo{title}{Highly accurate surface and
  volume integration on implicit domains by means of moment-fitting},
  \bibinfo{journal}{Internat. J. Numer. Meth. Engrg} \bibinfo{volume}{6}
  (\bibinfo{year}{2013}) \bibinfo{pages}{10--16}.

\bibitem[{Lehrenfeld(2016)}]{Lehrenfeld2016}
\bibinfo{author}{C.~Lehrenfeld}, \bibinfo{title}{High order unfitted finite
  element methods on level set domains using isoparametric mappings},
  \bibinfo{journal}{Comput. Methods Appl. Mech. Engrg.} \bibinfo{volume}{300}
  (\bibinfo{year}{2016}) \bibinfo{pages}{716--733}.

\bibitem[{Fries and Omerovi{\'c}(2016)}]{FriesOmerovic2015}
\bibinfo{author}{T.-P. Fries}, \bibinfo{author}{S.~Omerovi{\'c}},
  \bibinfo{title}{Higher-order accurate integration of implicit geometries},
  \bibinfo{journal}{Int. J. Numer. Methods Eng.} \bibinfo{volume}{106}
  (\bibinfo{year}{2016}) \bibinfo{pages}{323--371}.

\bibitem[{Fries et~al.(2017)Fries, Omerovi{\'c}, Sch{\"o}llhammer, and
  Steidl}]{FriesOmerovicSchoellhammerEtAl2017}
\bibinfo{author}{T.~Fries}, \bibinfo{author}{S.~Omerovi{\'c}},
  \bibinfo{author}{D.~Sch{\"o}llhammer}, \bibinfo{author}{J.~Steidl},
  \bibinfo{title}{Higher-order meshing of implicit geometries—Part I:
  Integration and interpolation in cut elements}, \bibinfo{journal}{Comput.
  Methods Appl. Mech. Engrg.} \bibinfo{volume}{313} (\bibinfo{year}{2017})
  \bibinfo{pages}{759--784}.

\bibitem[{G\"urkan and Massing(2018{\natexlab{a}})}]{GuerkanMassing2018a}
\bibinfo{author}{C.~G\"urkan}, \bibinfo{author}{A.~Massing}, \bibinfo{title}{{A
  stabilized cut discontinuous Galerkin framework: II. Hyperbolic and
  advection-dominated problems}}, \bibinfo{journal}{Submitted.} .

\bibitem[{G\"urkan and Massing(2018{\natexlab{b}})}]{GuerkanMassing2018b}
\bibinfo{author}{C.~G\"urkan}, \bibinfo{author}{A.~Massing}, \bibinfo{title}{{A
  stabilized cut discontinuous Galerkin framework: III. Surface and mixed
  dimensional coupled problems}}, \bibinfo{journal}{In preparation.} .

\bibitem[{Burman et~al.(2016{\natexlab{a}})Burman, Hansbo, Larson, and
  Massing}]{BurmanHansboLarsonEtAl2016a}
\bibinfo{author}{E.~Burman}, \bibinfo{author}{P.~Hansbo},
  \bibinfo{author}{M.~G. Larson}, \bibinfo{author}{A.~Massing},
  \bibinfo{title}{{A cut discontinuous Galerkin method for the
  Laplace--Beltrami operator}}, \bibinfo{journal}{IMA J. Numer. Anal.}
  \bibinfo{volume}{37}~(\bibinfo{number}{1})
  (\bibinfo{year}{2016}{\natexlab{a}}) \bibinfo{pages}{138--169}.

\bibitem[{Massing(2017)}]{Massing2017}
\bibinfo{author}{A.~Massing}, \bibinfo{title}{A Cut Discontinuous Galerkin
  Method for Coupled Bulk-Surface Problems}, chap. \bibinfo{chapter}{Chapter in
  UCL Workshop volumen on "Geometrically Unfitted Finite Element Methods"}, to
  appear in Lecture Notes in Computational Science and Engineering,
  \bibinfo{publisher}{Springer}, \bibinfo{pages}{1--19}, \bibinfo{year}{2017}.

\bibitem[{Gilbarg and Trudinger(2001)}]{GilbargTrudinger2001}
\bibinfo{author}{D.~Gilbarg}, \bibinfo{author}{N.~S. Trudinger},
  \bibinfo{title}{Elliptic Partial Differential Equations of Second Order},
  Classics in Mathematics, \bibinfo{publisher}{Springer-Verlag, Berlin},
  \bibinfo{year}{2001}.

\bibitem[{Li et~al.(2010)Li, Melenk, Wohlmuth, and
  Zou}]{LiMelenkWohlmuthEtAl2010}
\bibinfo{author}{J.~Li}, \bibinfo{author}{J.~Melenk},
  \bibinfo{author}{B.~Wohlmuth}, \bibinfo{author}{J.~Zou},
  \bibinfo{title}{Optimal a priori estimates for higher order finite elements
  for elliptic interface problems}, \bibinfo{journal}{Appl. Numer. Math.}
  \bibinfo{volume}{60}~(\bibinfo{number}{1}) (\bibinfo{year}{2010})
  \bibinfo{pages}{19--37}.

\bibitem[{Burman et~al.(2016{\natexlab{b}})Burman, Hansbo, Larson, and
  Zahedi}]{BurmanHansboLarsonEtAl2014}
\bibinfo{author}{E.~Burman}, \bibinfo{author}{P.~Hansbo},
  \bibinfo{author}{M.~Larson}, \bibinfo{author}{S.~Zahedi}, \bibinfo{title}{Cut
  Finite Element Methods for Coupled Bulk-Surface Problems},
  \bibinfo{journal}{Numer. Math.} \bibinfo{volume}{133}
  (\bibinfo{year}{2016}{\natexlab{b}}) \bibinfo{pages}{203--231}.

\bibitem[{Gro{\ss} et~al.(2015)Gro{\ss}, {Olshanskii}, and
  Reusken}]{GrossOlshanskiiReusken2014}
\bibinfo{author}{S.~Gro{\ss}}, \bibinfo{author}{M.~A. {Olshanskii}},
  \bibinfo{author}{A.~Reusken}, \bibinfo{title}{A trace finite element method
  for a class of coupled bulk-interface transport problems},
  \bibinfo{journal}{ESAIM: Math. Model. Numer. Anal.}
  \bibinfo{volume}{49}~(\bibinfo{number}{5}) (\bibinfo{year}{2015})
  \bibinfo{pages}{1303--1330}.

\bibitem[{Stein(1970)}]{Stein1970}
\bibinfo{author}{E.~Stein}, \bibinfo{title}{{Singular Integrals and
  Differentiability Properties of Functions}}, \bibinfo{publisher}{Princeton
  University Press}, \bibinfo{year}{1970}.

\bibitem[{Ern and Guermond(2006)}]{ErnGuermond2006}
\bibinfo{author}{A.~Ern}, \bibinfo{author}{J.-L. Guermond},
  \bibinfo{title}{{Evaluation of the condition number in linear systems arising
  in finite element approximations}}, \bibinfo{journal}{ESAIM: Math. Model.
  Numer. Anal.} \bibinfo{volume}{40}~(\bibinfo{number}{1})
  (\bibinfo{year}{2006}) \bibinfo{pages}{29--48}.

\bibitem[{Arnold(1982)}]{Arnold1982}
\bibinfo{author}{D.~Arnold}, \bibinfo{title}{An interior penalty finite element
  method with discontinuous elements}, \bibinfo{journal}{SIAM J. Num. Anal.}
  \bibinfo{volume}{19}~(\bibinfo{number}{4}) (\bibinfo{year}{1982})
  \bibinfo{pages}{742--760}.

\bibitem[{Di~Pietro and Ern(2012)}]{DiPietroErn2012}
\bibinfo{author}{D.~A. Di~Pietro}, \bibinfo{author}{A.~Ern},
  \bibinfo{title}{Mathematical aspects of discontinuous Galerkin methods},
  vol.~\bibinfo{volume}{69}, \bibinfo{publisher}{Springer},
  \bibinfo{year}{2012}.

\bibitem[{Massing et~al.(2014{\natexlab{b}})Massing, Larson, Logg, and
  Rognes}]{MassingLarsonLoggEtAl2013a}
\bibinfo{author}{A.~Massing}, \bibinfo{author}{M.~Larson},
  \bibinfo{author}{A.~Logg}, \bibinfo{author}{M.~Rognes}, \bibinfo{title}{{A
  stabilized Nitsche fictitious domain method for the Stokes problem}},
  \bibinfo{journal}{J. Sci. Comput.} \bibinfo{volume}{61}~(\bibinfo{number}{3})
  (\bibinfo{year}{2014}{\natexlab{b}}) \bibinfo{pages}{604--628}.

\bibitem[{Dryja(2003)}]{Dryja2003}
\bibinfo{author}{M.~Dryja}, \bibinfo{title}{On discontinuous Galerkin methods
  for elliptic problems with discontinuous coefficients},
  \bibinfo{journal}{Comput. Methods Appl. Math.}
  \bibinfo{volume}{3}~(\bibinfo{number}{1}) (\bibinfo{year}{2003})
  \bibinfo{pages}{76--85}.

\bibitem[{Ern et~al.(2009)Ern, Stephansen, and
  Zunino}]{ErnStephansenZunino2008}
\bibinfo{author}{A.~Ern}, \bibinfo{author}{A.~F. Stephansen},
  \bibinfo{author}{P.~Zunino}, \bibinfo{title}{A discontinuous Galerkin method
  with weighted averages for advection--diffusion equations with locally small
  and anisotropic diffusivity}, \bibinfo{journal}{IMA J. Numer. Anal.}
  \bibinfo{volume}{29} (\bibinfo{year}{2009}) \bibinfo{pages}{235--256}.

\bibitem[{Barrau et~al.(2012)Barrau, Becker, Dubach, and
  Luce}]{BarrauBeckerDubachEtAl2012a}
\bibinfo{author}{N.~Barrau}, \bibinfo{author}{R.~Becker},
  \bibinfo{author}{E.~Dubach}, \bibinfo{author}{R.~Luce}, \bibinfo{title}{A
  robust variant of NXFEM for the interface problem}, \bibinfo{journal}{C.R.
  Math.} \bibinfo{volume}{350}~(\bibinfo{number}{15}) (\bibinfo{year}{2012})
  \bibinfo{pages}{789--792}.

\bibitem[{Annavarapu et~al.(2012)Annavarapu, Hautefeuille, and
  Dolbow}]{AnnavarapuHautefeuilleDolbow2012}
\bibinfo{author}{C.~Annavarapu}, \bibinfo{author}{M.~Hautefeuille},
  \bibinfo{author}{J.~E. Dolbow}, \bibinfo{title}{A robust Nitsche's
  formulation for interface problems}, \bibinfo{journal}{Comput. Methods Appl.
  Mech. Engrg.} \bibinfo{volume}{225} (\bibinfo{year}{2012})
  \bibinfo{pages}{44--54}.

\end{thebibliography}

\end{document}